%% file: aby.tex
\documentclass[10pt,a4paper]{article}
\usepackage{amsmath,amssymb,amsthm} %if documentclass is amsart, amsmath package isn't needed
\usepackage[dvips]{graphicx}
\usepackage{nicefrac}  %fracoes bonitinhas - existe um pacote melhor, xfrac, mas nao funcionou!
\usepackage[all]{xy}   %diagramas
\usepackage{pxfonts}   %Nice fonts! The pxfonts package must be loaded AFTER amssymb, amsmath etc
\usepackage{psfrag}    %TeX fonts in the figures

\newtheorem{thm}{Theorem}[section]
\newtheorem{corol}[thm]{Corollary}
\newtheorem{prop}[thm]{Proposition}
\newtheorem{lemma}[thm]{Lemma}
\newtheorem*{ithm}{Theorem}
\newtheorem*{icorol}{Corollary}

\theoremstyle{definition}

\newtheorem{question}{Question}
\newtheorem*{ack}{Acknowledgements}

\theoremstyle{remark}
\newtheorem{rem}[thm]{Remark}

\def\R{\mathbb{R}}  \def\Z{\mathbb{Z}}
\def\N{\mathbb{N}} \def\Q{\mathbb{Q}}
\def\G{\mathit{SL}(2,\mathbb{R})} \renewcommand\P{\mathbb{P}}
\def\F{\mathbb{F}}  
\newcommand{\cB}{\mathcal{B}}\newcommand{\cC}{\mathcal{C}}
\newcommand{\cE}{\mathcal{E}}\newcommand{\cF}{\mathcal{F}}\newcommand{\cH}{\mathcal{H}}

\newcommand{\cM}{\mathcal{M}}\newcommand{\cO}{\mathcal{O}}

\newcommand{\M}{\mathbf{M}}
\def\eps{\varepsilon}
\renewcommand{\setminus}{\smallsetminus}
\renewcommand{\emptyset}{\varnothing}
\renewcommand{\epsilon}{\varepsilon}
\def\SL{\mathit{SL}}

\newcommand{\id}{\mathrm{id}}
\newcommand{\tr}{\operatorname{tr}}
\renewcommand{\Im}{\operatorname{Im}} %normalmente eu usaria \DeclareMathOperator

\newcommand{\cl}{\operatorname{cl}}
\newcommand{\Fix}{\operatorname{Fix}}

\begin{document}

\title{Uniformly Hyperbolic Finite-Valued $\SL(2,\R)$-Cocycles}
\author{A.~Avila, J.~Bochi, and J.-C.~Yoccoz}

\date{\today}

\maketitle %\margem{titulo ok?}

\begin{abstract}%\margem{abstract ok?}
We consider finite families of $\SL(2,\R)$ matrices whose products display uniform exponential growth.
These form open subsets of $(SL(2,\R))^N$, and
we study their components, boundary, and complement.
We also consider the more general situation where the allowed products of matrices satisfy a Markovian rule.
\end{abstract}

\tableofcontents

%%TABELA DE NOTACOES
%\medskip
%
%\begin{center}
%{\large \bf Alguns nomes mudaram:} \margem{obviamente isso sera apagado depois}
%
%\smallskip
%
%\begin{tabular}{|l|l|r|}
%\hline
%\emph{Concept} & \emph{Notation} & \emph{Page}\\
%\hline
%family of multicones              & $(M_\alpha)$                & ? \\
%multicone                    & $M$                         & ? \\
%unstable, stable families of cores & $(U_\alpha)$, $(S_\alpha)$  & ? \\
%unstable, stable core        & $U$, $S$                    & ? \\
%pair of combinatorial multicones  & $\M = \M_u \sqcup \M_s$     & ? \\
%\hline
%\end{tabular}
%\end{center}

\input{aby_intro.tex}
\input{aby_multicones.tex}

\input{aby_full_two.tex}

\input{aby_boundary.tex}
\input{aby_combin.tex}
%%\input{aby_groups.tex}
\input{aby_questions.tex}

\appendix
\input{aby_appendix.tex}

%%%%%%%%%%%%%%%%%%%%%%%%%%%%%%%%%%%%%%%%%%%%%%%%%%%%%%%%

\bibliography{../bib/tudo_mathscinet}
\bibliographystyle{abbrv}  %abbrv, alpha

\vfill

\noindent  Artur Avila (artur@math.sunysb.edu)

\noindent CNRS UMR 7599, Laboratoire de Probabilit\'es et Mod\`eles al\'eatoires -- Universit\'e Pierre et Marie Curie--Bo\^\i te courrier 188 --
75252 Paris Cedex 05, France.

\noindent \textit{Current address:} IMPA. Estrada Dona Castorina 110, Rio de Janeiro, 22460-320, Brazil.

\medskip

\noindent Jairo Bochi (jairo@mat.puc-rio.br)

\noindent PUC-Rio, Departamento de Matem\'atica. Rua Marqu\^es de S\~ao Vicente~225.  Rio de Janeiro, 22453-900, Brazil.

\medskip

\noindent Jean-Christophe Yoccoz (jean-c.yoccoz@college-de-france.fr)

\noindent Coll\`ege de France. 3, rue d'Ulm. 75005 Paris, France.

\end{document}

%% file: aby_intro.tex
%%%%%%%%%%%%%%%%%%%%%%%%%%%%%%%%%%%%%%%%%%%%%%%%%%%%%%%%
\section{Introduction}
%%%%%%%%%%%%%%%%%%%%%%%%%%%%%%%%%%%%%%%%%%%%%%%%%%%%%%%%

Let $\pi:E \to X$ be a vector bundle over a compact metric space $X$
and let $f: X \to X$ be a homeomorphism defining a dynamical systems in $X$.
A \emph{linear cocycle} mover $f$ is a vector bundle map $F:E\to E$ which is fibered over $F$.
The most important example occurs when $X$ us a manifold, $f$ is a diffeomorphism,
$E$ is the tangent bundle $TX$, and $F$ is the tangent map $Tf$.
But it is very profitable to consider larger classes of linear cocycles, allowing in particular
to separate the base dynamics from the fiber dynamics.

The most powerful tool
in the study of linear cocycles is Oseledets' Multiplicative Ergodic Theorem; see e.g.~\cite{Arnold_RDS}.
Given a probability measure on $X$ which is invariant
and ergodic under the basic dynamics $f$,
it allows to define Lyapunov exponents and split accordingly the fiber $E_x$
over almost all points of $x$.
In this context, one says that $F$ is hyperbolic if none of the Lyapunov exponents is
equal to zero.

There is a stronger notion of hyperbolicity, called uniform hyperbolicity,
which is of purely topological nature.
One requires that $E$ splits into a continuous direct sum $E^s \oplus E^u$,
with both $E^s$, $E^u$ invariant under $F$,
$E^s$ being contracted under $F$ and $E^u$ contracted under $F^{-1}$
(after suitable choices of norms on $E$).

The easiest non-commutative setting, and one of the most studied, is when
$E = X \times \R^2$ is trivial and $2$-dimensional,
and $F$ comes from a continuous map $A:X \to \SL(2,\R)$.
In this case, one is led to consider the products
\begin{equation}\label{e.cocycle}
A^n(x) :=
\begin{cases}
A(f^{n-1} x) \cdots A(x) &\text{for $n \ge 0$,}\\
A(f^n x)^{-1} \cdots A(f^{-1} x)^{-1} &\text{for $n <0$.}
\end{cases}
\end{equation}
The case where $X$ is a torus and $f$ is an irrational rotation
has attracted a lot of attention in recent years,
in particular in connection with the spectral properties of $1$-d discrete
Schr\"odinger operators with quasiperiodic potential: see for instance
\cite{Bourgain_book}, \cite{Damanik_survey}, \cite{Eliasson_ICMP05}
and references therein.
The values of the spectral parameter (energy)
corresponding to uniform hyperbolicity are those in the resolvent,
and the Lyapunov exponent is the main tool to study the spectrum.

The case where the base dynamics are chaotic is obviously also important.
Starting from the fundamental work of Furstenberg~\cite{Furstenberg_Noncommuting},
control of Lyapunov exponents has been obtained in several more general
settings: see \cite{GuivarchRaugi_84}, \cite{GoldsheidMargulis_89}, 
\cite{BonattiGMViana}, \cite{BonattiViana_ETDS04}.

In this work, we will consider, after \cite{Yoccoz_SL2R},
$\SL(2,\R)$-valued cocycles over chaotic base dynamics from the point of view of uniform hyperbolicity.
More precisely, $N$ will be an integer~$\ge 2$,
and the base $X = \Sigma \subset N^\Z$
will be a transitive subshift of finite type
(also called topological Markov chain),
equipped with the shift map
$\sigma: \Sigma \to \Sigma$.
We will only consider cocycles defines by a map $A: \Sigma \to \SL(2,\R)$
depending only on the letter in position zero.
The parameter space will be therefore the product $(\SL(2,\R))^N$.
The parameters $(A_1, \ldots, A_N)$ which correspond to a uniformly hyperbolic cocycle
form an open set $\cH$ which is the object of our study:
we would like to describe its boundary, its connected components, and its complement.
Roughly speaking, we will see that this goal is attained for the full shift on two symbols,
but that new phenomena appear with at least $3$ symbols such make such a complete
description much more difficult and complicated.

\medskip

Let us now review the contents of the following sections.

Associated to a $\SL(2,\R)$-valued cocycle
$A:X \to \SL(2,\R)$ over a base $f:X \to X$,
we have a fibered map $\bar A: X \to \P^1 \to X \times \P^1$.
The standard cone criterion says that $A$ is uniformly hyperbolic iff
one can find an open interval $I(x) \subset \P^1$ depending continuously on $x$
such that $A(x) I(x)$ is compactly contained in $I(f(x))$ for all $x \in X$.
In our setting, $A$ depends only on the zero coordinate $x_0$ of $x \in \Sigma$
and we would like for $I(x)$ to do the same.
This is in general not possible but nevertheless a result in this direction
exists if one allows several components for $I(x)$,
leading to the notion of multicone.
In the full shift  case the result is as follows:

\begin{ithm}[\ref{t.multicone full}]
A parameter $(A_1, \ldots, A_N)$ is uniformly hyperbolic (over the full shift $N^\Z$)
iff there exists a non-empty open set $M \neq \P^1$
with finitely many components having disjoint closures
which satisfies $A_\alpha M \Subset M$ for $1 \le \alpha \le N$.
\end{ithm}

There is a similar statement (Theorem~\ref{t.multicone sub})
for general subshifts of finite type.

Section~\ref{s.full 2} is dedicated to the case where $\Sigma$
is the full shift on two symbols.
We have a rather complete understanding of the hyperbolicity locus $\cH$ in this case.
The simplest components of $\cH$ are the $4$ principal components;
they consist of parameters for which the multicone $M$ in
Theorem~\ref{t.multicone full} is connected and are deduced from each other by change of
signs of the matrices.
Next there are the so-called free components of $\cH$ ($8$ of them), consisting
of parameters for which the multicone has two components.
All the other non-principal components of $\cH$ are obtained by taking
the preimage of one of the free components by a diffeomorphism of $(\SL(2,\R))^2$
belonging to the free monoid generated by
$$
F_+(A,B) = (A,AB), \qquad F_-(A,B) = (BA, B) \, .
$$
Moreover, any two distinct components of $\cH$ have disjoint closures,
and any compact set in parameter space meets only finitely many components of $\cH$.
In Subsection~\ref{ss.full2 combin}, the combinatorics
and dynamics of the multicones are described for each component of $\cH$.

Recall that a matrix $A \in \SL(2,\R)$ is said to be
hyperbolic (resp.\ parabolic, resp.\ elliptic)
if $|\tr A|>2$  (resp.\ $|\tr A|=2$, resp.\ $|\tr A|<2$).
Denote by $\cE$ the set of parameters for which there exists a periodic point
$x\in \Sigma$ (of period $k$) such that
$A^k(x)$ is elliptic.
Obviously, $\cE$ is an open set disjoint from $\cH$.
Avila has proved that for a general subshift of finite type, the closure
of $\cE$ is equal to the complement of $\cH$.
When $\Sigma$ is the full shift on two symbols,
we prove the stronger statement that
$\cE$ and $\cH$ have the same boundary, the complement of their union.

The main result of Section~\ref{s.boundaries}
is the following result (for general subshifts of finite type):

\begin{ithm}[\ref{t.general boundary}]
Let $(A_1, \ldots, A_N)$ belong to the boundary of a component of $\cH$.
Then one of the following possibilities hold:
\begin{itemize}
\item There exists a a periodic point $x$ of $\Sigma$, of period $k$,
such that $A^k(x)$ is parabolic;
\item There exist periodic points $x$, $y$ of $\Sigma$,
of respective periods $k$, $\ell$,
an integer $n \ge 0$,
and a point $z \in W^u_\text{loc}(x) \cap \sigma^{-n} W^s_\text{loc}(y)$
such that $A^k(x)$, $A^\ell(y)$ are hyperbolic and
$$
A^n(z) u(A^k(x)) = s(A^\ell(y)) \, .
$$
\end{itemize}
\end{ithm}

We denote here by $u(A)$ or $u_A$ (resp.\ $s(A)$ or $s_A$)
the unstable (resp.\ stable) direction of a hyperbolic matrix $A$.
(When $A$ is parabolic and $A \neq \pm \id$, we still write
$u_A = s_A$ for the unique invariant direction.)
The second case in the statement of the theorem is called an heteroclinic connection.
The integers $k$, $\ell$, $n$ occurring in Theorem~\ref{t.general boundary}
are actually bounded by a constant depending only on the component
of $\cH$ considered in the statement.
It follows easily that:

\begin{icorol}[\ref{c.semialgebraic}]
Every connected component of $\cH$ is a semialgebraic set.
\end{icorol}

In the full-shift case, for parameters on the boundary of non-principal components,
no product of the matrices can be equal to $\pm \id$.
The result we prove in Subsection~\ref{ss.nonprincipal},
together with similar results,
is actually stronger.

In Subsections~\ref{ss.ex connection}--\ref{ss.bifurcation},
we investigate what happens along parameter families
going through an heteroclinic connection.
Starting with a single component of $\cH$ (for the full shift on $3$ symbols),
it may happen that the complement of $\cH \cup \cE$ is locally a smooth hypersurface;
but it may also happen that the boundary of the starting component is
accumulated by a sequence of distinct components of $\cH$.

In Section~\ref{s.abstract combinatorics},
we consider from a purely combinatorial point of view the dynamics on the components
of the multicones for positive and negative iteration:
this leads to the concept of combinatorial multicones and monotone correspondences.
Necessary conditions on these objects
to come from a matrix realization are introduced.
It is shown that these conditions are also sufficient in the case of the full-shift on two symbols.
An example is provided to show that the conditions are no longer sufficient
for full-shifts with more symbols.

Except for the case of the full-shift on two symbols,
many questions are still open and are discussed in Section~\ref{s.questions}.

In Annex~\ref{ss.compactness criterium},
a criterium characterizing relative compactness modulo conjugacy in parameter space is
proved: $\tr A_i$ and $\tr A_i A_j$ have to stay bounded.

\medskip

There is one part of the study of the components of the hyperbolicity locus $\cH$
which is only briefly mentioned in this paper, and deserves further work:
this is the group vs monoid question.
In the full shift case, a parameter $(A_1,\ldots,A_N)$ is hyperbolic if and only if
matrices in the \emph{monoid} generated by $A_1$, \ldots, $A_N$ grow exponentially
with word length.
For certain components of $\cH$, but not all,
it actually implies that the matrices in the (free) group generated by $A_1$, \ldots, $A_N$
grow exponentially with word length.
For instance, for the full-shift on two symbols,
this is true for non-principal components,
but not true for principal components.
In a further paper we plan to characterize which components have this property
for the full-shift on $3$ or more symbols.

\bigskip
\begin{ack}
During the long preparation of this paper, the authors benefited from support from
CNPq (Brazil), CAPES (Brazil), CNRS (France), the Franco-–Brazilian cooperation agreement in Mathematics.
This research was partially conducted
during the period A.A.\  served as a Clay Research Fellow.
J.B.\ is partially supported by a CNPq research grant.
\end{ack}

%% file: aby_multicones.tex
\section{Multicones}
%%%%%%%%%%%%%%%%%%%%%%%%%%%%%%%%%%%%%%%%%%%%%%%%%%%%%%%%%%%

%\margem{INTRODUCE NOTATION THAT WAS IN THE PREVIOUS INTRODUCTION}

We recall the following result from~\cite{Yoccoz_SL2R},
that says that uniform exponential growth of the products in \eqref{e.cocycle}
guarantees uniform hyperbolicity:

\begin{prop}\label{p.unif hyp}
If $f:X \to X$ is a homeomorphism of a compact space and $A: X \to \SL(2,\R)$
is a continuous map, then
the cocycle $(T,A)$ is uniformly hyperbolic
iff there exist $c>0$ and $\lambda>1$ such that
$\left\|A^n(x)\right\| \ge c \lambda^n$ for all $x\in \Sigma$, $n \ge 0$.
\end{prop}

%**Remark: If $\|\mathord{\cdot}\|$ is an operator norm and the shift is full,
%we can take $c=1$ in the proposition above.**

As explained in the Introduction,
we consider a general transitive subshift of finite type $\Sigma \subset N^\Z$,
where $N \ge 2$.
Given $A_1, A_2, \ldots, A_N \in \SL(2,\R)$,
we consider the map
$(x_i)_{i \in \Z}  \in \Sigma \mapsto A_{x_0} \in \SL(2,\R)$.
If the associated cocycle is uniformly hyperbolic then we say that
\emph{the $N$-tuple $(A_1, \ldots, A_N)$ is uniformly hyperbolic with respect to the subshift $\Sigma$.}

If $A \in \G$, we also indicate by $A$ the induced map $\P^1 \to \P^1$,
where $\P^1$ is the projective space of $\R^2$.

Next we describe a geometric condition which is equivalent to uniform hyperbolicity of a $N$-tuple.
Let us begin with full shifts:
\begin{thm}\label{t.multicone full}
An $N$-tuple $(A_1, \ldots, A_N)$ is uniformly hyperbolic w.r.t.~the
full shift $\Sigma = N^\Z$ iff
there exists a nonempty open subset $M \subset\P^1$ with $\overline{M} \neq \P^1$ such that\footnote{$X\Subset Y$ means that the closure $\overline{X}$ of $X$ is contained in the interior of $Y$.}
$A_\alpha(M) \Subset M$ for every $\alpha \in \{1,\ldots,N\}$.
We can take $M$ with finitely many connected components,
and those components with disjoint closures.
\end{thm}

A set $M$ satisfying all the conditions in the theorem is
called a \emph{multicone} for $(A_\alpha)$.

\medskip

Now let $\Sigma$ be any subshift of finite type.
If $\alpha$ and $\beta$ are symbols in the alphabet $\{1,\ldots,N\}$,
we write $\alpha \to \beta$ to indicate that
the symbol $\alpha$ can be followed by the symbol $\beta$.
The generalization of Theorem~\ref{t.multicone full} is:

\begin{thm}\label{t.multicone sub}
An $N$-tuple $(A_1, \ldots, A_N)$ is uniformly hyperbolic w.r.t.~$\Sigma$ iff
there are non-empty open sets $M_\alpha \subset \P^1$,
one for each symbol $\alpha$,
with $\overline{M_\alpha} \neq \P^1$, and such that
$$
\alpha \to \beta \quad \text{implies} \quad A_\beta (M_\alpha) \Subset M_\beta.
$$
We can take each $M_\alpha$ with finitely many connected components,
and those components with disjoint closures.
\end{thm}

A family of sets $(M_\alpha)$ satisfying all the conditions in the theorem
is called a \emph{family of multicones} for the $N$-tuple $(A_\alpha)$.

For any subshift of finite type $\Sigma \subset N^\Z$,
we can define the \emph{dual subshift} $\Sigma^* \subset N^\Z$ as follows:
if $\alpha \to \beta$ are the allowed transitions for $\Sigma$,
then the allowed transitions for $\Sigma^*$ are $\beta \stackrel{*}{\to} \alpha$.
If $(A_\alpha)$ is a uniformly hyperbolic $N$-tuple w.r.t.~$\Sigma$, with
a family of multicones $(M_\alpha)$, then
the $N$-tuple $(A_\alpha^{-1})$ is uniformly hyperbolic  w.r.t.~$\Sigma^*$,
with family of multicones $(M_\alpha') = (\P^1 \setminus A_\alpha^{-1}(\overline{M_\alpha}))$.

\medskip

Let us see that Theorem~\ref{t.multicone full} is a corollary of
Theorem~\ref{t.multicone sub}:
If $(A_\alpha)$ is uniformly hyperbolic, and $M_\alpha$'s are given by Theorem~\ref{t.multicone sub},
let $M = \bigcup_\alpha M_\alpha$.
Since $A_\alpha \overline{M} \subset \overline{M_\alpha} \neq \P^1$,
we have $\overline{M} \neq \P^1$.
Conversely, given a multicone $M$ we simply take $M_\alpha = M$ for all $\alpha$.

\subsection{Examples}\label{ss.examples}

Let $\Sigma = N^\Z$ be the full shift on $N$ symbols.
If the matrices $A_1$, \ldots, $A_N$ have a common strictly invariant interval,
then by Theorem~\ref{t.multicone full} $(A_1,\ldots,A_N)$ is uniformly hyperbolic.
Consider the set of such $N$-tuples; its connected components
are the \emph{principal components} of the hyperbolic locus $\cH$.
By Proposition~3 from \cite{Yoccoz_SL2R},
such a component must contain some $N$-tuple of the form
$(\pm A_*, \ldots, \pm A_*)$, where $\tr A_* >2$.
Hence there are $2^N$ principal components.

\medskip

Let $\Sigma = 2^\Z$ be the full shift on $2$ symbols.
For any $m \ge 2$, let us show that there is a uniformly hyperbolic pair $(A,B)$
which has a multicone $M$ with $m$ components, but no multicone with $m-1$ components.
Take any hyperbolic matrix $A$.
Choose $u$, $s \in \P^1$ such that
$$
s_A < u \le A^{m-2} u < s < A^{m-1} u < u_A < s_A
$$
(for some cyclical order on the circle $P^1$).
Take a hyperbolic matrix $B$ with $u_B = u$, $s_B = s$.
If the spectral radius of $B$ is large enough,
it is easy to see that $(A,B)$ has a multicone $M$ with $m$ components
containing respectively the points
$u_B$, $A(u_B)$, \ldots, $A^{m-2}(u_B)$, $u_A$.
%such that  %inverti a ordem!
%$$
%A(I_1) \Subset I_1, \quad A(I_j) \Subset I_{j-1} \ \forall j>1, \quad
%B(I_j) \Subset I_m \ \forall j.
%$$
Figure~\ref{f.multicone 1/4} illustrates the case $m=4$.
%(The connected components of $\cH$
%corresponding to cases $m=2$ and $m=3$ are studied respectively in {\S}7 and {\S}8 of \cite{Yoccoz_SL2R}.)
\begin{figure}[!hbt]
\begin{center}
\psfrag{1}{{\tiny $BA^3$}}
\psfrag{2}{{\tiny $BA^2$}}
\psfrag{3}[r][r]{{\tiny $ABA^2$}}
\psfrag{4}[r][r]{{\tiny $ABA$}}
\psfrag{5}[r][r]{{\tiny $A^2 BA$}}
\psfrag{6}{{\tiny $A^2 B$}}
\psfrag{7}{{\tiny $A^3 B$}}
\psfrag{A}[r][r]{{\tiny $A$}}
\psfrag{B}[l][l]{{\tiny $B$}}
\includegraphics[height=6.2cm]{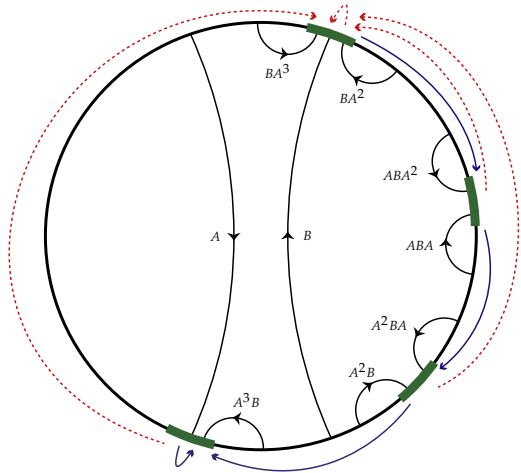}
\caption{{\small Example of a uniformly hyperbolic pair $(A,B)$ and a multicone.
Outer arrows indicate the action of $A$ and $B$ in the components of the multicone.
Inner arrows indicate stable and unstable directions of $A$, $B$, and some of their products.}}
\label{f.multicone 1/4}
\end{center}
\end{figure}

The examples just described do not exhaust the possibilities for the full $2$-shift.
See Figure~\ref{f.multicone 2/5} for a more complicate example.
We postpone the description of this and all other possible examples
for $\Sigma =2^\Z$ to Section~\ref{s.full 2}.
\begin{figure}[!hbt]
\begin{center}
\psfrag{A}[l][l]{{\tiny $A$}}
\psfrag{B}[r][r]{{\tiny $B$}}
\psfrag{0}[c][c]{{\tiny $BABAA$}}
\psfrag{1}[r][r]{{\tiny $BA$}}
\psfrag{2}[r][r]{{\tiny $ABABA$}}
\psfrag{3}[r][r]{{\tiny $AB$}}
\psfrag{4}[l][l]{{\tiny $AABAB$}}
\psfrag{5}[l][l]{{\tiny $AAB$}}
\psfrag{6}[l][l]{{\tiny $ABAAB$}}
\psfrag{7}[l][l]{{\tiny $ABA$}}
\psfrag{8}[l][l]{{\tiny $BAABA$}}
\psfrag{9}[l][l]{{\tiny $BAA$}}
\includegraphics[height=6.2cm]{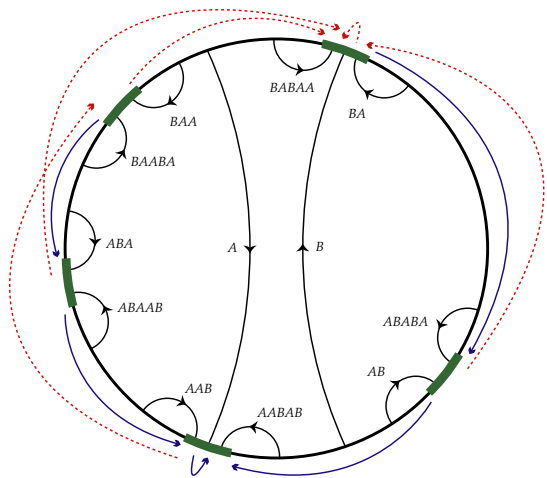}
\caption{{\small Another example of a uniformly hyperbolic pair $(A,B)$.}}
\label{f.multicone 2/5}
\end{center}
\end{figure}

Some examples of uniformly hyperbolic $3$-tuples
are indicated in Figure~\ref{f.multicone ABC}.
\begin{figure}[!hbt]
\begin{center}
\psfrag{1}[r][r]{{\tiny $ACB$}}
\psfrag{2}[r][r]{{\tiny $ABC$}}
\psfrag{3}[l][l]{{\tiny $BAC$}}
\psfrag{4}[l][l]{{\tiny $BCA$}}
\psfrag{5}[l][l]{{\tiny $CBA$}}
\psfrag{6}[l][l]{{\tiny $CAB$}}
\psfrag{A}{{\tiny $A$}}
\psfrag{B}{{\tiny $B$}}
\psfrag{C}{{\tiny $C$}}
\includegraphics[width=4.5cm]{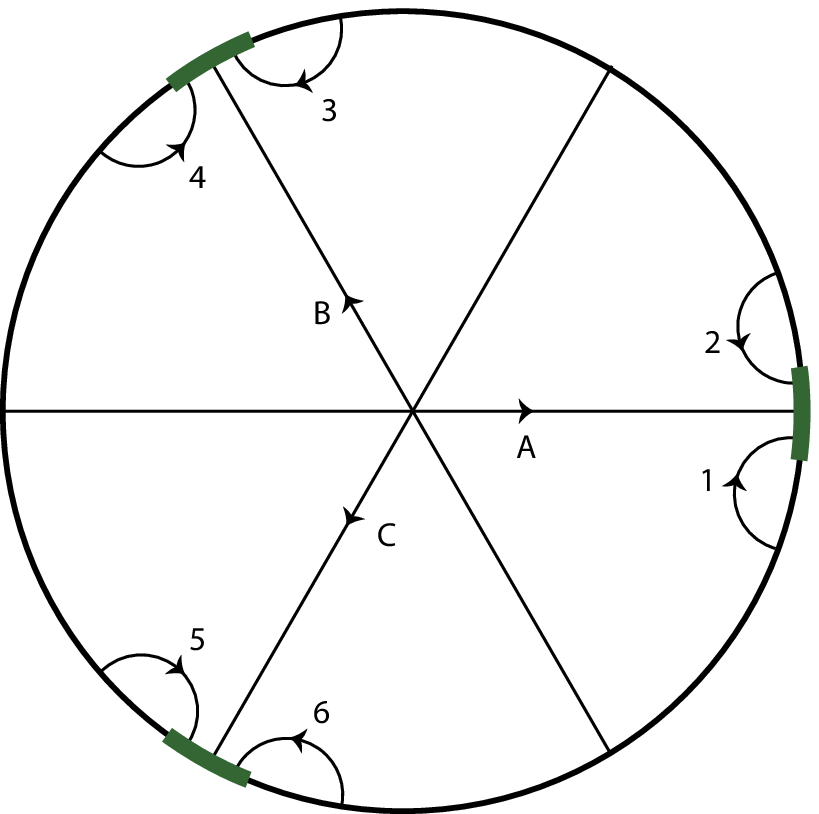} \hspace{2cm}
\psfrag{1}[r][r]{{\tiny $ACB$}}
\psfrag{2}[l][l]{{\tiny $BAC$}}
\psfrag{3}[l][l]{{\tiny $CBA$}}
\psfrag{A}[r][r]{{\tiny $A$}}
\psfrag{B}[c][c]{{\tiny $B$}}
\psfrag{C}[l][l]{{\tiny $C$}}
\includegraphics[width=4.5cm]{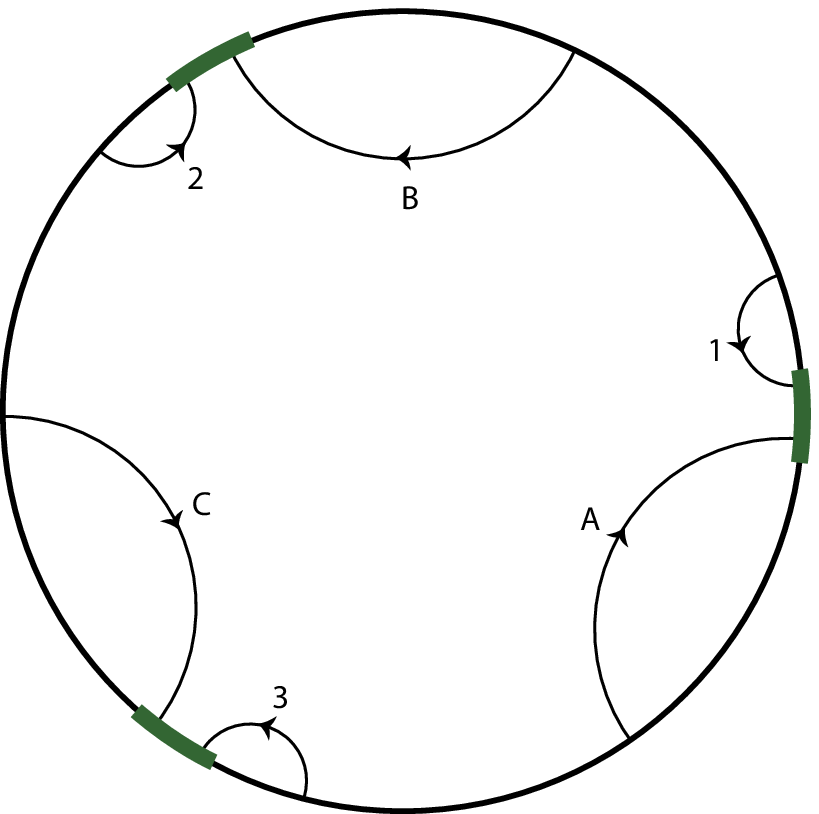}
\caption{{\small Two examples of uniformly hyperbolic $3$-tuples $(A,B,C)$.}}
\label{f.multicone ABC}
\end{center}
\end{figure}

An example illustrating the situation of Theorem~\ref{t.multicone sub},
is indicated in Figure~\ref{f.multicone family}.
(For another example, see \S\ref{ss.ping pong}, specially Fig.~\ref{f.pingpong}.)
\begin{figure}[!hbt]
\begin{center}
\psfrag{A}[r][r]{{\tiny $A_1$}}
\psfrag{B}[c][c]{{\tiny $A_2$}}
\psfrag{C}[l][l]{{\tiny $A_3$}}
\psfrag{1}[l][l]{{\tiny $M_1$}}
\psfrag{2}[r][r]{{\tiny $M_2$}}
\psfrag{3}[c][c]{{\tiny $M_3$}}
\includegraphics[width=4.5cm]{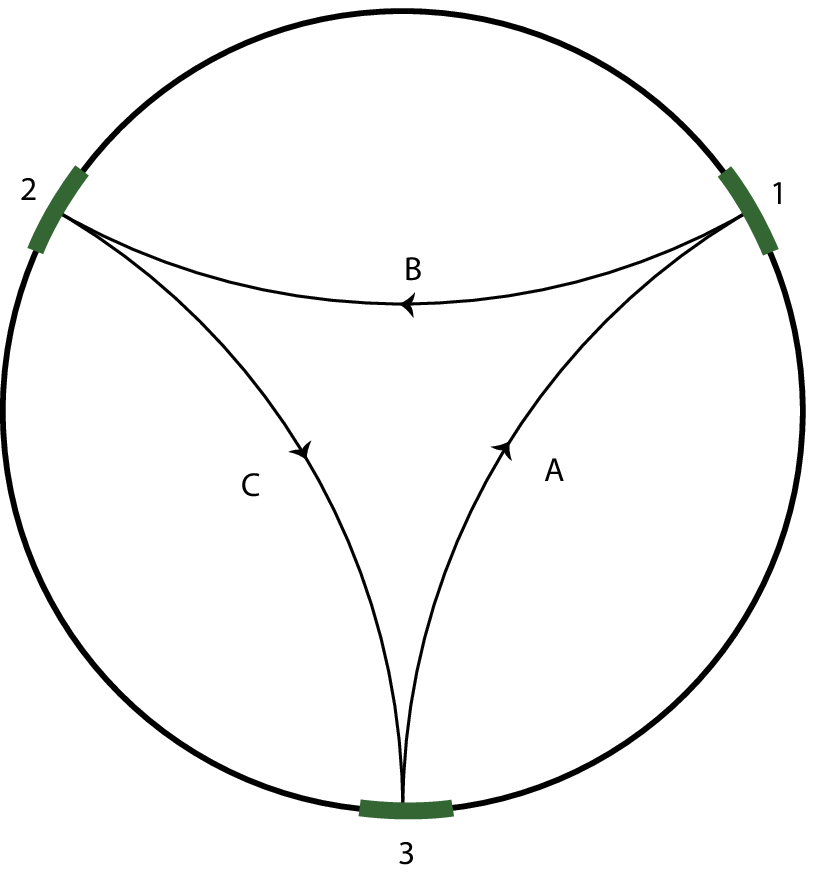}
\caption{{\small An example of a $3$-tuple $(A_1, A_2, A_3)$
that is uniformly hyperbolic with respect to the subshift on the symbols $1$, $2$, $3$
whose only forbidden transitions are $1 \to 2$, $2\to 3$, and $3 \to 1$.
The intervals $M_1$, $M_2$, $M_3$ form a family of multicones.}}
\label{f.multicone family}
\end{center}
\end{figure}

%\clearpage  %comando bruto que obriga o Latex a botar todas as figuras antes de prosseguir

\subsection{Proof of the ``If'' Part of Theorem~\ref{t.multicone sub}}\label{ss.if}

Let us first establish some notation to be used from now on:

Given an ordered basis $\cB = \{v_1, v_2\}$ of $\R^2$,
we define a bijection $P_\cB: \P^1 \to \R \cup \{\infty\}$
by $P_\cB^{-1}(t) = v_1 + t v_2$, $P_\cB^{-1}(\infty) = v_2$.
The map $P_\cB$ is called a \emph{projective chart.}

If $a$, $b$, $c$, $d$ are four distinct points in the extended real line $\R \cup \{\infty\}$
then we define their \emph{cross-ratio}
\begin{equation}\label{e.def cross ratio}
[a,b,c,d] = \frac{c-a}{b-a} \cdot \frac{d-b}{d-c} \in \R  \, .
\end{equation}

If $x$, $y$, $z$, $w$ are distinct points in the circle $\P^1$, we take any
projective chart $P: \P^1 \to \R \cup \{\infty\}$ and
define the cross-ratio $[x,y,z,w] = [P(x),P(y),P(z),P(w)]$.
The definition is good because \eqref{e.def cross ratio}
is invariant under M\"obius transformations.
Of course, for any $A \in \G$ we have $[x,y,z,w] = [A(x),A(y),A(z),A(w)]$.

A set $I \subset \P^1$ is called an \emph{open interval}
if it is non-empty, open, connected, and its complement contains more than one point.
A set $I \subset \P^1$ is called a \emph{closed interval}
if either it consists of one point or is the complement of an open interval.

%A \emph{non-degenerate interval} $I \subset \P^1$ is \margem{tirar o non-degenerate?}
%a non-empty connected open set whose complement consists of at least
%two points.

An open interval $I$ can be endowed with the
\emph{Hilbert metric} $d_I$, defined as follows:
If $a$, $b$ are the endpoints of $I$ then
$$
d_I (x,y) = \big|\log  [a,x,y,b] \big| \quad \text{for all distinct } x, \  y \in I.
$$

Recall the following properties of the Hilbert metric:
%It is a Riemannian metric. %(trivial)
If $A \in \G$ satisfies $A(I)=J$ then $A$ takes $d_I$ to $d_J$.
If $J \subsetneqq I$ are open intervals then
the metric of $J$ is greater than the metric of $I$.
If, in addition, $J \Subset I$ then the metric of $J$ is greater than the metric of $I$
by a factor at least $\lambda(I,J)>1$.

%\begin{lemma}\label{l.if part}
%Given $(A_1, \ldots, A_N) \in \G^N$, if there is a multicone $(M_\alpha)$
%then the associated cocycle is uniformly hyperbolic.
%\end{lemma}

\begin{proof}[Proof of the ``if'' part of Theorem~\ref{t.multicone sub}]
For each symbol $\alpha$,
let $d_\alpha$ be the Riemannian metric on $M_\alpha$ which coincides with the
Hilbert metric in each of its components.
Let $K_\alpha$ be the closure of the union of the sets
$A_\alpha M_\gamma$, where $\gamma \to \alpha$.
We can assume that $K_\alpha$ intersects each connected component of $M_\alpha$,
because otherwise we can take a smaller $M_\alpha$.
Let $L_\alpha \Subset M_\alpha$ be an open set
containing $K_\alpha$ and with the same number of connected components as $M_\alpha$.
Then each component $M_{\alpha,i}$ of $M_\alpha$ contains
a unique component $L_{\alpha,i}$ of $L_\alpha$.
Let $\lambda = \min_{\alpha,i} \lambda(M_{\alpha,i}, L_{\alpha,i})$.

Take an admissible sequence of symbols $\alpha_0 \to \alpha_1 \to \cdots \to \alpha_n$,
and let $A = A_{\alpha_n} \cdots A_{\alpha_1}$.
If $u$, $v$ belong to the same component of $M_{\alpha_0}$
then
$$
d_{\alpha_n} (A u, A v) \le \lambda^{-n} d_{\alpha_0} (u,v).
$$
The metrics $d_\alpha | L_\alpha$ are comparable
to the Euclidean metric $d$ on $\P^1$.
So if $u$, $v$ belong to the same component of $L_{\alpha_0}$ we get
$d (A u, A v) \le C \lambda^{-n} d(u,v)$,
where $C>0$ is some constant.
This in turn implies that $\|A\| \ge C^{-1/2} \lambda^{n/2}$.
By Proposition~\ref{p.unif hyp}, we are done.
\end{proof}

\subsection{Proof of the ``Only If'' Part of Theorem~\ref{t.multicone sub}}\label{ss.onlyif}

Assume the cocycle associated to $(A_1, \ldots, A_N)$ is uniformly hyperbolic.
This means that there are continuous functions
$e^s$, $e^u: \Sigma \to \P^1$ and constants $C>0$, $\lambda>1$ such that
for all $x\in \Sigma$:
\begin{alignat*}{3}
A(x) e^s(x)   &= e^s(\sigma x); &\qquad
\|A^n (x) v\| &\le C \lambda^{-n} \|v\| &\quad &\text{for all $v \in e^s(x)$ and $n\ge 0$;} \\
A(x) e^u(x) &= e^u(\sigma x);  &\qquad
\|A^{-n} (x) v \| &\le C \lambda^{-n}\|v\|  &\quad &\text{for all $v \in e^u(x)$ and $n\ge 0$.}
\end{alignat*}
Moreover, $e^s(x)$ and $e^u(x)$ are uniquely determined by those properties,
and $e^u(x) \neq e^s(x)$ for every $x \in \Sigma$.
Thus, for $x = (x_i)_{i\in \Z}$,
$e^u(x)$ depends only on $(\ldots, x_{-2}, x_{-1})$,
while $e^s(x)$ depends only on $(x_0, x_1, \ldots)$.
(That is, $e^u$, resp.~$e^s$,
is constant on local unstable, resp.~stable, manifolds.)

If $\alpha$ is a symbol, we define the following two compact sets:
$$
K^u_\alpha = \{e^u(x) ; \; x_{-1} = \alpha \}, \qquad
K^s_\alpha = \{e^s(x) ; \; x_0 = \alpha \}.
$$
Notice that if $\alpha \to \beta$ then
$K^u_\alpha \cap K^s_\beta = \emptyset$.
Also,
$$
K^u_\beta = \bigcup_{\alpha ; \; \alpha \to \beta} A_\beta K^u_\alpha
\quad\text{and}\quad
K^s_\alpha = \bigcup_{\beta ; \; \alpha \to \beta} A_\alpha^{-1} K^s_\beta \, .
$$
So $K^u_\alpha \cap A_\alpha K^s_\alpha = \emptyset$.

Let us now define two families of sets $U_\alpha$ and $S_\alpha$,
called the \emph{unstable} and \emph{stable families of cores of} $(A_1,\ldots,A_N)$
as follows:
\begin{itemize}
\item $U_\alpha$ is the complement of the union of the connected
components of $\P^1 \setminus K^u_\alpha$ that intersect
$A_\alpha K^s_\alpha$;
\item $S_\alpha$ is the complement of the union of the connected
components of $\P^1 \setminus K^s_\alpha$ that intersect
$A_\alpha^{-1} K^u_\alpha$.
\end{itemize}

%For fixed $\alpha$, consider the two disjoint compact sets
%$K^u_\alpha$ and
%$A_\alpha K^s_\alpha = \bigcup_{\beta \leftarrow \alpha} K^s_\beta$.
%Let $V_\alpha$ be the union of the connected components of
%$\P^1 \setminus A_\alpha K^s_\alpha$ that intersect $K^u_\alpha$.
%And let $U_\alpha$ be the complement of the union of the connected
%components of $\P^1 \setminus K^u_\alpha$ that intersect
%$A_\alpha K^s_\alpha$.
%We have
%%$K^u_\alpha \subset U_\alpha \subset V_\alpha \subset \P^1 \setminus A_\alpha K^s_\alpha$.
%$U_\alpha \subset V_\alpha$.
%Also, if $\alpha \to \beta$ then
%%$A_\beta(A_\alpha K^s_\alpha) \supset A_\beta K^s_\beta$, and therefore
%$$
%A_\beta(V_\alpha) \subset V_\beta \quad\text{and}\quad
%A_\beta(U_\alpha) \subset U_\beta .
%$$

It is straightforward to check that the families of cores satisfy the following properties:
\begin{enumerate}
\item   $U_\alpha$, $S_\alpha$ are non-empty compact sets with finitely
many connected components;
\item  $U_\alpha \cap A_\alpha S_\alpha = \emptyset$;
\item every connected component of $\P^1 \setminus A_\alpha S_\alpha$,
resp.~$\P^1 \setminus A_\alpha^{-1} U_\alpha$,
contains a unique connected component of $U_\alpha$, resp.~$S_\alpha$.
\item
${\displaystyle U_\beta \supset \bigcup_{\alpha ; \; \alpha \to \beta} A_\beta U_\alpha}$
and
${\displaystyle S_\alpha \supset \bigcup_{\beta ; \; \alpha \to \beta} A_\alpha^{-1} S_\beta}$.
\end{enumerate}
It follows from these conditions that each $U_\alpha$ has the same number $k(\alpha)$
of connected components as $S_\alpha$.
We define the \emph{rank} of the families as the integer $\sum_\alpha k(\alpha)$.
%\margem{poderia ter uma figura mostrando a posicao de $U_\alpha$ e $A_\alpha S_\alpha$ (para um $\alpha$ fixo)}

\begin{lemma}\label{l.core}
Let $(A_1, \ldots, A_N) \in \G^N$.
Assume that there exist two families of sets
$U_\alpha$ and $S_\alpha$ (where $\alpha$ runs on the symbols)
satisfying properties (i)-(iv) above, and with rank $n_0$.
Assume also that for every periodic point $x\in \Sigma$ of period $n \le n_0$,
the corresponding matrix product $A^n(x)$ is not $\pm \id$.
Then $(A_1, \ldots, A_N)$ has a family of multicones $(M_\alpha)$.
%(and, in particular, is uniformly hyperbolic).
Moreover, $U_\alpha \subset M_\alpha \Subset \P^1 \setminus A_\alpha S_\alpha$,
and each connected component of $\P^1 \setminus A_\alpha S_\alpha$ contains
a unique connected component of $M_\alpha$.
\end{lemma}

Clearly, Lemma~\ref{l.core} implies the ``only if'' part of Theorem~\ref{t.multicone sub}.
%because the unstable and unstable cores defined above satisfy properties
%(i)-(iv) in the lemma, as it is easily checked.
The reason why we stated Lemma~\ref{l.core} in this generality is that
it gives a criterion for uniform hyperbolicity which will be useful in some other occasions.
%it will also be employed in the proof of Theorem~\ref{t.general boundary}.
%\margem{tb eh usado em outros lugares!}

\begin{proof}[Proof of Lemma~\ref{l.core}]
Let $V_\alpha = \P^1 \setminus A_\alpha S_\alpha$.
Write each $V_\alpha$ as a disjoint union of open intervals
$V_{\alpha,1} \sqcup \cdots \sqcup V_{\alpha,k(\alpha)}$,
and write $U_\alpha = U_{\alpha,1} \sqcup \cdots \sqcup U_{\alpha,k(\alpha)}$
with $U_{\alpha,i}=U_\alpha \cap V_{\alpha,i}$.
%The rank is $n_0 = \sum_\alpha k(\alpha)$.

Define a Riemannian metric $d_\alpha$ on $V_\alpha$ by taking on each component
of $V_\alpha$ the corresponding Hilbert metric.
For $\epsilon>0$, let $U_{\alpha,i}(\epsilon)$ denote an $\epsilon$-neighborhood of $U_{\alpha,i}$
with respect to $d_\alpha$.
Also let $U_\alpha (\epsilon) = \bigcup_{i=1}^{k(\alpha)} U_{\alpha,i}(\epsilon)$.
Notice that if $\alpha \to \beta$ then
$A_\beta \cdot V_\alpha \subset V_\beta$ and hence
$A_\beta \cdot U_\alpha(\epsilon) \subset U_\beta(\epsilon)$.

Let $x \in \Sigma$ be such that $x_{-1} = x_{n-1} = \alpha$ for some $n$ with $1 \le n \le n_0$.
Assume that $A^n(x) \cdot V_{\alpha,i} \subset V_{\alpha,i}$ for some $i$
(or, equivalently, $A^n(x) \cdot U_{\alpha,i} \subset U_{\alpha,i}$).
We claim that then $A^n(x) \cdot U_{\alpha,i}(\epsilon) \Subset U_{\alpha,i}(\epsilon)$,
for any $\epsilon>0$.
Indeed, the matrix $B = A^n(x)$ is not $\pm \id$, by assumption,
nor elliptic, because it leaves the interval $V_{\alpha,i}$ invariant.
Therefore $u(B)$ and $s(B)$ are defined.
We have $u(B) \in U_{\alpha,i}$ %, because $B U_{\alpha,i} \subset U_{\alpha,i}$
and $s(B) \notin V_{\alpha,i}$, so $s(B) \notin \overline{U_{\alpha,i}(\epsilon)}$.
Therefore $B$ is hyperbolic and its restriction to $U_{\alpha,i}(\epsilon)$ strictly contracts
the metric $d_\alpha$.
This proves the claim.

\smallskip

From now on fix some arbitrary $\epsilon'>0$.
By compactness, there exists a positive $\epsilon'' < \epsilon'$ such that
if $x \in \Sigma$ and $1 \le n \le n_0$ are such that $x_{-1} = x_{n-1} = \alpha$
and $A^n(x) \cdot V_{\alpha,i} \subset V_{\alpha,i}$ for some $\alpha$ and $i$,
then $A^n(x) \cdot U_{\alpha,i}(\epsilon') \subset U_{\alpha,i}(\epsilon'')$.

For $n\ge 0$, let
$$
U^n_\alpha(\epsilon) = \bigcup_{x \in \Sigma;\; x_{n-1}=\alpha} A^n(x) \cdot U_{x_{-1}}(\epsilon).
$$
Notice that $U^k_\alpha(\delta) \subset U^n_\alpha(\epsilon)$ if
$\delta \le \epsilon$ and $k \ge n$,
and also that $A_\beta U^n_\alpha(\epsilon) \subset U^{n+1}_\beta(\epsilon)$
if $\alpha \to \beta$.

We claim that $U^{n_0}_\alpha(\epsilon') \subset U_\alpha(\epsilon'')$ for any $\alpha$.
Indeed, take $x\in \Sigma$ with $x_{{n_0}-1}=\alpha$ and $v \in U_{x_{-1}}(\epsilon')$.
By the definition of the rank $n_0$, %the pigeon-hole principle
there exist
$0 \le k < \ell \le {n_0}$ such that
$x_{k-1} = x_{\ell-1}$ and moreover
$A^k(x) \cdot v$ and $A^\ell(x) \cdot v$ belong to the same
connected component of $U_{x_{k-1}}(\epsilon')$,
say $U_{x_{k-1}, i}(\epsilon')$.
Then
$$
A^\ell(x) \cdot v \in
A^{\ell-k}(\sigma^k x) \cdot U_{x_{k-1}, i}(\epsilon') \subset U_{x_{\ell-1}, i}(\epsilon''),
$$
and so $A^{n_0}(x) \cdot v \in U_\alpha(\epsilon'')$,
proving the claim.

At last, take a sequence
$\epsilon'' = \epsilon_0 < \epsilon_1 < \cdots < \epsilon_{n_0} = \epsilon'$
and let
$$
M_\alpha = \bigcup_{n=0}^{{n_0}-1} U^n_\alpha(\epsilon_{n+1}),
$$
for each $\alpha$.
If $\alpha \to \beta$ then
$$
A_\beta M_\alpha \subset
\bigcup_{n=0}^{{n_0}-1} U^{n+1}_\beta(\epsilon_{n+1}) \subset
\bigcup_{n=0}^{{n_0}-1} U^n_\beta(\epsilon_n) \Subset M_\beta \, .
$$
So the family of sets $M_\alpha$ has the required properties.
\end{proof}

%%%%%%%%%%%%%%%%%%%%%%%%%%%%%%%%%%%%%%%%%%%%%%%
\subsection{The Case of Full Shifts}

Here we will give some additional information about multicones
in the specific case of the full shift $\Sigma = N^\Z$, which interests us most.
In that case, a characterization of uniform hyperbolicity becomes simpler,
involving a single multicone (cf.\ Theorem~\ref{t.multicone full}),
instead of a family of multicones (cf.\ Theorem~\ref{t.multicone sub}).

\subsubsection{Multicones}

Given a uniformly hyperbolic $N$-tuple $(A_1,\ldots,A_N)$,
let $e^u$, $e^s : N^\Z \to \P^1$ be the same maps as in \S\ref{ss.onlyif},
and let $K^u$, $K^s \subset \P^1$ be their respective images.
Notice that these sets are disjoint, $K^u = \bigcup_\alpha A_\alpha(K^u)$,
and $K^s = \bigcup_\alpha A_\alpha^{-1}(K^s)$.

These sets relate with multicones as follows:
If $M$ is any multicone for $(A_1,\ldots, A_N)$ then
$$
K^u = \bigcap_{n=0}^\infty \bigcup_{i_1, \ldots, i_n} A_{i_n} \cdots A_{i_1}(M) \, , \qquad
K^s = \bigcap_{n=0}^\infty \bigcup_{i_1, \ldots, i_n}
(A_{i_n} \cdots A_{i_1})^{-1} \big(\P^1 \setminus \overline {M} \big) \, .
$$
The proof is left to the reader.
%\margem{We can delete this and only mention $K^u \subset M$, $K^s \subset \P^1 \setminus \overline {M}$.}

Another fact that is worth to mention is:

\begin{prop}\label{p.combin contraction}
Let $M$ be a multicone for a uniformly hyperbolic $N$-tuple $(A_1,\ldots,A_N)$.
Then there exists $k$ such that every product of $A_i$'s of length $\ge k$
sends $M$ into a single connected component of $M$.
\end{prop}

\begin{proof}
Fix a multicone $M$ for $(A_1,\ldots,A_N)$.
We have $K^u \subset M$ and $K^s \subset \P^1 \setminus \overline{M}$.
In particular, there is $\eps>0$ such that the $2\eps$-neighborhood of $K^u$ (resp.\ $K^s$)
is contained in $M$ (resp.\ $\P^1 \setminus \overline{M}$).
There is $c = c(\eps) >1$ such that if $B \in \G$ is hyperbolic, the distance between $u_B$ and $s_B$ is
at least $4 \eps$, and $\|B\|>c$ then $B$ sends the complement of the $\eps$-neighborhood of $s_B$ into
the $\eps$-neighborhood of $u_B$.
Let $k$ be such that every product of $A_i$'s of length $\ge k$ has norm at least $c$.
Then we are done.
\end{proof}

\subsubsection{Cores}\label{sss.full core}

As already mentioned, Theorem~\ref{t.multicone full} is a corollary
of Theorem~\ref{t.multicone sub}.
Nevertheless, it is worthwhile to see how the proof in \S\ref{ss.onlyif}
could be simplified.

\smallskip

Given the hyperbolic $N$-tuple $(A_1, \ldots, A_N)$, let $K^u$, $K^s \subset \P^1$ be
as above.
Define other sets $U$ and $S$ as follows:
\begin{itemize}
\item $U$ is the complement of the union of the connected components of $\P^1 \setminus K^u$
that intersect $K^s$;
\item $S$ is the complement of the union of the connected components of $\P^1 \setminus K^s$
that intersect $K^u$.
\end{itemize}
The set $U$, resp.~$S$, is called the \emph{unstable}, resp.~\emph{stable},  \emph{core
of $(A_1,\ldots,A_N)$.}
The following properties are easily checked:
\begin{enumerate}
\item $U$, $S$ are non-empty compact sets with finitely many components;
\item $U$ and $S$ are disjoint, and moreover each connected component of
$\P^1 \setminus S$, resp.\ $\P^1 \setminus U$, contains a unique connected component
of $U$, resp.~$S$.
\item $A_i(U) \subset U$ and $A_i^{-1}(S) \subset S$ for every symbol $i$.
\end{enumerate}
It follows from these conditions that
the sets $U$ and $S$ have the same number of connected components;
call this number the \emph{rank} of the sets.

\begin{rem}
The relation between the cores $U$, $S$ and the families of cores $U_\alpha$, $S_\alpha$ considered before
is simple: $\P^1\setminus U$ is the union of the connected components of
$\P^1 \setminus \bigcup U_\alpha$ that meet $\bigcup S_\alpha$, and analogously for $S$.
In particular, $U$ contains $\bigcup U_\alpha$ and that
$\partial U$ is contained in $\bigcup \partial U_\alpha$.
\end{rem}

The following is a criterium for uniform hyperbolicity
(specific for the the full shift):

\begin{lemma}\label{l.core full}
Let $(A_1, \ldots, A_N) \in \G^N$.
Assume that there exists sets $U$, $S \subset \P^1$ satisfying properties (i)-(iii) above.
Assume also that for every string of $A_i$'s of length less of equal to the rank of the sets,
the product is different from $\pm \id$.
Then $(A_1, \ldots, A_N)$ has a multicone $M$.
Moreover, $U \subset M \Subset \P^1 \setminus S$,
each connected component of $\P^1 \setminus S$ contains a unique connected component of~$M$.
\end{lemma}

The proof of Lemma~\ref{l.core full} is merely a simplification
of the proof of Lemma~\ref{l.core}, and will be left to the reader.
Of course, using Lemma~\ref{l.core full} one can give a direct proof
of the ``only if'' part of Theorem~\ref{t.multicone full}.

\subsubsection{Tightness} \label{sss.tightness}

%We keep restricting ourselves the full shift in $N$ symbols.

A multicone $M$ for the $N$-tuple $(A_1, \ldots A_N)$ will be called \emph{tight}
if the following two conditions hold:
\begin{itemize}
\item the set $\bigcup_i A_i(M)$ intersects every connected component of $M$;
\item the set $\bigcup_i A_i^{-1}\big(\P^1 \setminus \overline{M} \big)$
intersects every connected component of \mbox{$\P^1 \setminus \overline{M}$}.
\end{itemize}
(Notice no condition implies the other.)
%Por exemplo, para um par free podemos achar um multicone M consistindo em 4 intervalos
%satisfazendo a primeira condicao (mas nao a segunda, claro).

Tightness has a simple reformulation in terms of the cores:

\begin{prop}\label{p.tight}
A multicone $M$ is tight iff
every connected component of $M$ contains a unique connected component of $U$ and
every connected component of $\P^1 \setminus \overline{M}$ contains a unique connected component of~$S$.
\end{prop}

\begin{proof}
Fixed a uniformly hyperbolic  $N$-tuple, let $K^u$, $K^s$, $U$, $S$ be as before.
Let $M$ be a multicone, and let $M^* = \P^1 \setminus \overline M$.

First, let us prove the ``if'' part:
Assume every connected component of $M$ (resp.\ $M^*$) intersects $U$ (resp.\ $S$).
Since $K^u \subset U$, each component of $M$ intersects $K^u$.
Now, each point in $K^u$ is the image of another point in $K^u$ (and hence in $M$) by some $A_i$.
So each component of $M$ intersects some $A_i(M)$.
With a symmetric argument for $M^*$ and $S$ we conclude that $M$ is tight.

\smallskip

Now let us prove the ``only if'' part of the proposition.
Assume that the multicone $M$ is tight.
To conclude, it is sufficient to show that every connected component of $M$ intersects $K^u$,
and that every connected component of $M^*$ intersects $K^s$.
In fact, by symmetry, we only need to prove the first  claim.

Fix a connected component of $M$, say, $M_0$.
By the first condition in the definition of tightness,
there exists a connected component $M_1$ of $M$ such that
$A_{i_1} (M_1) \subset M_0$ for some $i_1$.
Continuing by induction, define components $M_n$ and indices $i_n$ for all $n \ge 1$ so that
$A_{i_{n+1}} (M_{n+1}) \subset M_n$.
The number of connected components is finite,
so let $k \ge 1$ be the least index such that $M_k = M_\ell$ for some $\ell < k$.
The interval $M_\ell$ is forward-invariant by $A_{i_{\ell+1}} \cdots A_{i_{k-1}} A_{i_k}$,
so it contains the unstable direction of that product.
So $M_\ell$ intersects $K^u$.
The interval $M_0$ contains $A_{i_1} A_{i_2} \cdots A_{i_\ell}(M_\ell)$,
hence it intersects $K^u$ as well.
This concludes the proof.
\end{proof}

\begin{rem}
It follows from Proposition~\ref{p.tight} that a
multicone for a uniformly hyperbolic $N$-tuple $(A_1, \ldots, A_N)$
is tight iff there is no multicone
with a smaller number of connected components.
\end{rem}

%% file: aby_full_two.tex
%%%%%%%%%%%%%%%%%%%%%%%%%%%%%%%%%%%%%%%%%%%%%%%%%%%%%%%%
\section{The Full $2$-Shift Case}\label{s.full 2}
%%%%%%%%%%%%%%%%%%%%%%%%%%%%%%%%%%%%%%%%%%%%%%%%%%%%%%%%

\subsection{Statements}\label{ss.full 2 statements}

Before going into other general results, we study the simplest case:
the full shift on two symbols.
So in this section we let $\Sigma=2^\Z$ and
let $\cH \subset \G^2$ denote the associated hyperbolicity locus.

%Since the subset of hyperbolic matrices in $\G$
%consists on two connected components,
%the set $\cH \subset \G^2$ has at least four obvious connected components.
%Namely, those which contain
%$(A_0,A_0)$, $(A_0,-A_0)$, $(-A_0,A_0)$, and $(-A_0,-A_0)$,
%where $A_0 \in \G$ is any fixed matrix with $\tr A_0>2$.
%These are called the \emph{principal components}.

%The following is an immediate consequence of Proposition~3 from~\cite{Yoccoz_SL2R}
%\begin{prop}
%A pair $(A,B)$ belongs to a principal component iff there exists an open interval
%$I \subset \P^1$ with $\overline{I} \neq \P^1$ such that
%$A(I) \Subset I$ and $B(I) \Subset I$.
%\end{prop}

%That is, the principal components correspond to the simplest case
%in Theorem~\ref{t.multicone full},
%where $M$ is an open interval.

By definition, a connected component of $\cH$ is called principal
if every pair in it has a multicone consisting of a single interval.
Recall from \S\ref{ss.examples} that there are four such components.
Let $H_0$ indicate their union.

The next simplest case is when a tight multicone consists on two intervals.
So let $H_\id \subset \G^2$ denote the (open) set of
pairs $(A,B)$ that do not belong to a principal component,
and have a multicone $M$ which is a union of two intervals.

(See Figure~\ref{f.pingpong} for an example of $(A,B)\in H_\id$; $M = I_1 \cup I_2$ is a multicone.)

In fact (see Proposition~\ref{p.free}), we have
$$
H_\id = \{(A,B) \in \G^2; \; |\tr A|>2, |\tr B|>2, |\tr AB|>2, \tr A \tr B \tr AB < 0 \},
$$
and moreover, $H_\id$ has eight connected components.
Let us call these as the \emph{free} components of~$\cH$.

\medskip

Define mappings $F_+, F_- : \G^2 \to \G^2$ by
$$
F_+ (A,B) = (A,AB) \quad \text{and} \quad
F_- (A,B) = (BA,B) \, .
$$
These are diffeomorphisms of $\G^2$.
Let $\cM$ be the monoid\footnote{semigroup with identity} generated by $F_+$ and $F_-$.

\begin{thm}[Connected components of $\cH$] \label{t.components}
Every connected component of $\cH$ is one of the following:
\begin{itemize}
\item either a principal component;
\item or $F^{-1} (H)$ for some free component $H \subset H_\id \subset \cH$ and some $F \in \cM$.
\end{itemize}
Moreover, such components are distinct.
\end{thm}

\begin{thm}[Boundary of $\cH$]\label{t.full2 boundary}
A compact subset of $\G^2$ intersects only finitely many components of $\cH$.

The boundary of $\cH$ is the disjoint union of the boundaries of its components.

Moreover, if $(A,B) \in \partial \cH$ then (at least) one of the following holds:
\begin{enumerate}
\item  There is a product of $A$'s and $B$'s which is parabolic;
\item  or $u_A= s_B$ or $u_B = s_A$.
\end{enumerate}
The second possibility can only occur if $(A,B)$ belongs to
the boundary of a principal component.
\end{thm}

Let $\cE \subset \G^2$ be the set of pairs $(A,B)$ such that
there exists a product of $A$'s and $B$'s which is elliptic.
Of course, $\cE$ is an open set, disjoint from $\cH$.
In fact, $\cE$ is the complement of $\overline{\cH}$,
as a consequence of the following result:

\begin{thm}[Relation between $\cH$ and $\cE$]\label{t.H and E}
$\partial \cH = \partial \cE = (\cH \sqcup \cE)^c$.
\end{thm}

We are also able to give a precise description of the multicones
for all components of $\cH$, see \S\ref{ss.full2 combin}.

The results above answer all questions of \cite{Yoccoz_SL2R} for the full $2$-shift.
(Namely, the answers are 1: yes, 1': no, 2: no, 3, 3', 4: yes.)
The solution of Problem~1 can also be given using the description of \S\ref{ss.full2 combin}.

The proofs of Theorems~\ref{t.components}, \ref{t.full2 boundary}, and \ref{t.H and E}
occupy the following subsections.
%In the proofs we do not use the results stated in \S\ref{ss.intro general}
%(and proven in Section~\ref{s.proofs general}),
%with the exception of the (easy) ``if'' part of Theorem~\ref{t.multicone sub}.

%%%%%%%%%%%%%%%%%%%%%%%%%%%%%%%%%%%%%%%%%%%%%%%%%%%%%%%%
\subsection{Plan of Proof} %\margem{talvez melhor incorporar na subsecao anterior}

First, let us prove the assertions already made about $H_\id$:

\begin{prop}\label{p.free}
We have
\begin{equation}\label{e.free}
H_\id = \{(A,B) \in \G^2; \; |\tr A|>2, |\tr B|>2, |\tr AB|>2, \tr A \tr B \tr AB < 0 \}.
\end{equation}
The set $H_\id$ has eight connected components,
and these components have disjoint boundaries.

The subset of $H_\id$ given by
\begin{equation}\label{e.two free set}
\{(A,B); \; \tr A>2, \tr B>2, \tr AB <-2\}
\end{equation}
has two connected components,
which are conjugated by an orientation-reversing automorphism of $\P^1$.
Fixed a cyclical order on $\P^1$, we have in one of the two components that
\begin{equation} \label{e.positive free}
u_B < u_{BA} < s_{BA} < s_A < u_A < u_{AB} < s_{AB} < s_B < u_B.
\end{equation}
\end{prop}

The component of the set in \eqref{e.two free set} where \eqref{e.positive free} holds
is called the \emph{positive free component}.
(Of course this definition depends on the choice of an orientation in $\P^1$.)

\begin{proof}
If $(A,B) \in H_\id$ then modulo sign changes
(which do not affect being in either side of~\eqref{e.free})
we can assume that $\tr A$, $\tr B >2$.
The fact that $(A,B)$ does not belong to a principal component
implies that
$u_B < s_A < u_A < s_B < u_B$ for some cyclical order on $\P^1$.
Let $M$ be the multicone for the pair $(A,B)$;
write it as union of two intervals $M = I \cup J$.
Then one of the intervals, say $I$, must contain $u_A$ and the other, $u_B$.
So $u_{AB}$ is contained in $I$, and, as it is easy to see,
the associated eigenvalue of $AB$ is negative.
This shows that $\tr AB <-2$ so $(A,B)$ belongs to
the right-hand side of~\eqref{e.free}.

On the other hand, Proposition~5 in~\cite{Yoccoz_SL2R} and its proof
show that the set in~\eqref{e.two free set} has two connected components
with the stated properties.
The proof also shows that pairs $(A,B)$ in that set have a multicone
consisting in two intervals.
Of course, if $u_B < s_A < u_A < s_B < u_B$ for some cyclical order on $\P^1$
then $(A,B)$ cannot be in a principal component of $\cH$.
So the set in~\eqref{e.two free set} is contained in $H_\id$.
We conclude that the set in the right-hand side of~\eqref{e.free} is also
contained in $H_\id$ and has eight connected components.

To prove that the connected components of $H_\id$ have disjoint boundaries,
it suffices to see that the two components of the set~\eqref{e.two free set} have disjoint boundaries.
So assume $(A,B)$ is a boundary point of both components.
Then $u_B = s_A = u_A = s_B$.
So $\tr A = \tr B = 2$, and this implies $\tr AB = 2$, a contradiction.
\end{proof}

%Recall the connected components of $H_0$ are called principal,
%while the connected components of $H_\id$ are called free.

Given $F \in \cM$, let us denote $H_F = F^{-1} (H_\id)$.
Our plan to prove the main results is as follows.
In~\S\ref{ss.ping pong}--\ref{ss.length} we will show:
\begin{prop}\label{p.is hyperbolic}
For any $F \in \cM$, $H_F \subset \cH$.
\end{prop}

Then in~\S\ref{ss.twisted pairs}--\ref{ss.dynamics} we will prove:
\begin{prop} \label{p.disjoint}
$\G^2$ is the disjoint union of $\cE$, $\overline{H_0}$,
and $\bigsqcup_{F \in \cM} \overline{H_F}$.
Moreover, a compact set in $\G^2$ intersects only finitely many
of the sets $\overline{H_F}$.
%We have the following disjoint union:
%$$
%\G^2 = \cE \sqcup \overline{H_0} \sqcup \bigsqcup_{F \in \cM} \overline{H_F} \, .
%$$
\end{prop}

Putting things together,
we will prove Theorems~\ref{t.components}, \ref{t.full2 boundary}, and \ref{t.H and E}
in~\S\ref{ss.conclusion}.

In~\S\ref{ss.full2 combin} we will give an alternative proof of
Proposition~\ref{p.is hyperbolic}, by describing explicitly the multicones.

\subsection{Group-Hyperbolic Pairs}\label{ss.ping pong}

Let $(A,B) \in \G^2$ be given.
Let $\Sigma \subset 4^\Z$
be the (transitive) subshift of finite type
where the only forbidden transitions are
$1 \to 3$, $3 \to 1$, $2 \to 4$, and $4 \to 2$.
Take the
%function from the alphabet to $\G$ given by
%$a^{\pm 1} \mapsto A^{\pm 1}$,  $b^{\pm 1} \mapsto B^{\pm 1}$.
$4$-tuple $(A_1,A_2,A_3,A_4) = (A,B,A^{-1},B^{-1})$,
and consider the usual cocycle map over the subshift.
If this cocycle is uniformly hyperbolic, then we will
say the pair $(A,B)$ is \emph{group-hyperbolic}.

\begin{lemma}\label{l.free is gr hyp}
If $(A,B)$ belongs to a free component then $(A,B)$ is group-hyperbolic.
\end{lemma}

\begin{proof}
Without loss, we assume that $(A,B)$ belongs to the positive free component
(so \eqref{e.positive free} holds). %where
%$$
%u_B < u_{BA} < s_{BA} < s_A < u_A < u_{AB} < s_{AB} < s_B < u_B.
%$$
Take four disjoint (open) intervals $I_1$, $I_2$, $I_3$, $I_4$ such that
$I_1 \cup I_2$ is a multicone for $(A,B)$ (over the full $2$-shift),
$I_3 \cup I_4$ is a multicone for $(A^{-1},B^{-1})$ (over the full $2$-shift), and
$$
I_1 \supset [u_A, u_{AB}], \
I_4 \supset [s_{AB}, s_B], \
I_2 \supset [u_B, u_{BA}], \
I_3 \supset [s_{BA}, s_A].
$$

\begin{figure}[hbt]
\begin{center}
\psfrag{A}{{\tiny $A$}} \psfrag{B}{{\tiny $B$}} \psfrag{AB}{{\tiny $AB$}}
\psfrag{BA}{{\tiny $BA$}}
\psfrag{I1}{{\tiny $I_2$}}\psfrag{I2}{{\tiny $I_3$}}\psfrag{I3}{{\tiny $I_1$}}\psfrag{I4}{{\tiny $I_4$}}
\includegraphics[width=4.5cm]{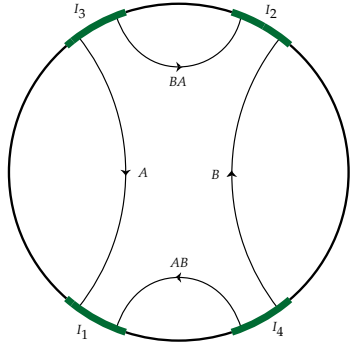}
\caption{{\small Group-hyperbolicity of the free component.}} \label{f.pingpong}
\end{center}
\end{figure}

Since $A(I_1)$, $A(I_2) \Subset I_1$, we see that $A(I_4) \Subset I_1$ as well.
In the same manner, we have:
\begin{alignat*}{2}
A     (I_1 \cup I_4 \cup I_2) &\Subset I_1,  &\qquad
B     (I_2 \cup I_3 \cup I_1) &\Subset I_2, \\
A^{-1}(I_3 \cup I_2 \cup I_4) &\Subset I_3,  &\qquad
B^{-1}(I_4 \cup I_1 \cup I_3) &\Subset I_4.
\end{alignat*}
So Theorem~\ref{t.multicone sub} applies,
and our cocycle over the subshift $\Sigma \subset 4^\Z$ is uniformly hyperbolic.
That is, $(A,B)$ is group-hyperbolic.
\end{proof}

\subsection{Length Comparison}\label{ss.length}

Let $\F_2$ be the free group in two generators $a$, $b$.
Let $|\mathord{\cdot}|$ be the usual length function on $\F_2$,
relative to the generators $a$, $b$.
Let $f_+$, $f_-$ be the homomorphisms of $\F_2$ such that
$f_+(a) = a$, $f_+(b) = ab$, $f_-(a) = ba$, $f_-(b) = b$.
Notice $|f_{\pm}(\omega)| \le 2 |\omega|$ for all $\omega \in \F_2$.
Since $f_+$ and $f_-$ are in fact automorphisms, it follows that
$|f_{\pm}^{-1}(\omega)| \ge \frac{1}{2} |\omega|$ for all $\omega \in \F_2$.

Given $(A,B) \in \G^2$, there is a unique homomorphism
$\langle \mathord{\cdot}, (A,B) \rangle : \F_2 \to \G$ such that
$\langle a, (A,B) \rangle = A$ and $\langle b, (A,B) \rangle = B$.
In fact, this gives a bijection between $\G^2$ and the
set of homomorphisms $\F_2 \to \G$.

If $f: \F_2 \to \F_2$ is a homomorphism then there is a unique map
$f^*: \G^2 \to \G^2$ such that
$\langle f(\omega), (A,B) \rangle = \langle \omega, f^*(A,B) \rangle$.
%namely, $f^*(A,B) = (\langle f(a), (A,B) \rangle, \langle f(b), (A,B) \rangle)$.
The functorial properties $\id^* = \id$ and $(g \circ f)^* = f^* \circ g^*$ hold.
Also notice that $f_+^* = F_+$ and $f_-^* = F_-${\,}.

\begin{proof}[Proof of Proposition~\ref{p.is hyperbolic}]
Let $(A, B) \in H_F$,
where $F = F_{\epsilon_k} \circ \cdots \circ F_{\epsilon_1}$, $\epsilon_i \in\{+,-\}$.
Let $(A_0, B_0) = F(A,B) \in H_\id$.
By Lemma~\ref{l.free is gr hyp}, $(A_0, B_0)$ is group-hyperbolic.
This means that there exist $c$, $\tau>0$ such that
for every $\omega \in \F_2$,
$$
\| \langle \omega, (A_0,B_0) \rangle \| \ge c \exp( \tau |\omega|).
$$
Let $f = f_{\epsilon_1} \circ \cdots \circ f_{\epsilon_k}$, so $f^* = F$.
For any $\omega \in \F_2$, we have
$$
\| \langle \omega, (A, B) \rangle \|
= \| \langle f^{-1}(\omega) , (A_0, B_0) \rangle \|
\ge c \exp \left( \tau |f^{-1}(\omega)| \right)
\ge c \exp \left( 2^{-k} \tau |\omega| \right).
$$
This proves that $(A,B)$ is group-hyperbolic and,
in particular, $(A,B)$ is a uniformly hyperbolic pair w.r.t.~the full $2$-shift.
\end{proof}

%%%%%%%%%%%%%%%%%%%%%%%%%%%%%%%%%%%%%%%%%%%%%%%%%%%%%%%%
\subsection{Twisted Pairs}\label{ss.twisted pairs}

Let us say that $(A,B) \in \G^2$ is \emph{straight} if $(A,B) \in \overline{H_0}$, that is,
$(A,B)$ belongs to the closure of a principal component.

Notice that if $(A,B)$ is straight then so are $F_+(A,B)$ and $F_-(A,B)$.

It is easy to see that if there is an open interval which is
forward-invariant for both $A$ and $B$ then $(A,B)$ is straight.
The converse is not true:
for example, if $A\neq \pm \id$ is parabolic then $(A, A^{-1})$ is straight,
but there is no invariant open interval.

Let us say that a pair $(A,B)$ is \emph{twisted}
if $A$ and $B$ are not elliptic
and $(A,B)$ is not straight.

Let $A$ be non-elliptic, and $A \neq \pm \id$, so $u_A$, $s_A \in \P^1$ are defined.
Assume that an orientation is fixed in $\P^1$.
Given $p\in \P^1$, we shall write
\mbox{$p < u_A \lesssim s_A < p$}
to indicate that $p < Ap < u_A \le s_A < p$.
This means that there exist $\tilde A$ arbitrarily close (possibly equal) to $A$ such that
$p < u_{\tilde A} < s_{\tilde A} < p$.
In the case $A$ is parabolic we can define $u_A \lesssim s_A$ without mentioning a point $p$.

\begin{lemma}\label{l.twist a}
Let $A$, $B \in \G$ be non-elliptic.
Then $(A,B)$ is twisted iff
$A$, $B\neq\pm\id$ and for some cyclical order on $\P^1$ we have
\begin{equation}\label{e.twist order}
u_A < s_B \lesssim u_B < s_A \lesssim u_A
\end{equation}
\end{lemma}

\begin{proof}
If $A$ or $B$ equals $\pm\id$, then $(A,B)$ is easily seen to be straight.
So we can assume $A$, $B \neq \pm \id$.

The rest of the proof is merely a case-by-case inspection.
The following list exhausts all possible (mutually exclusive) cases,
modulo inverting the cyclical order on $\P^1$,
or interchanging $A$ and $B$,
or replacing $(A,B)$ by $(A^{-1}, B^{-1})$:
\begin{enumerate}
\item[1.] $A$ and $B$ are hyperbolic:
    \begin{enumerate}
    \item[1.1.] $u_A = u_B$ or $u_A = s_B$
    \item[1.2.] $u_A < u_B < s_B < s_A < u_A$
    \item[1.3.] $u_A < u_B < s_A < s_B < u_A$
    \item[1.4.] $u_A < s_B < u_B < s_A < u_A$ %twisted
    \end{enumerate}
\item[2.]  $A$ hyperbolic and $B$ parabolic:
    \begin{enumerate}
    \item[2.1.] $u_A = u_B$
    \item[2.2.] $u_A < u_B \lesssim s_B < s_A < u_A$
    \item[2.3.] $u_A < s_B \lesssim u_B < s_A < u_A$ %twisted
    \end{enumerate}
\item[3.] $A$ and $B$ parabolic:
    \begin{enumerate}
    \item[3.1.] $u_A = u_B$ with $u_A \lesssim s_A$ and $s_B \lesssim u_B$%caso chatinho (excecao)
    \item[3.2.] $u_A = u_B$ with $u_A \lesssim s_A$ and $u_B \lesssim s_B$
    \item[3.3.] $u_A \neq u_B$ with $u_A \lesssim s_A$ and $s_B \lesssim u_B$
    \item[3.4.] $u_A \neq u_B$ with $u_A \lesssim s_A$ and $u_B \lesssim s_B$%twisted
    \end{enumerate}
\end{enumerate}

The cases 1.1, 1.2, 1.3, 2.1, 2.2, 3.2, and 3.3 are those where
there is an invariant open interval, and hence are straight.
In the case 3.1, there is no invariant open interval, but it is straight nevertheless.
The remaining cases, 1.4, 2.3, and 3.4 are precisely
those where condition~\eqref{e.twist order} holds;
and none of them can be straight.
\end{proof}

\begin{lemma}\label{l.twist b}
Let $(A,B)$ satisfy $\tr A,\tr B \ge 2$.
Then $(A,B)$ is twisted iff
there exists a basis (called canonical basis for $(A,B)$)
where $A$, $B$ are written as
\begin{equation}\label{e.canonical}
A=\begin{pmatrix} \mu&\alpha\\ 0&\mu^{-1} \end{pmatrix}, \quad
B=\begin{pmatrix} \nu^{-1}&0\\ \beta&\nu \end{pmatrix},
\end{equation}
with $\mu \geq 1$, $\nu \geq 1$ and $\alpha \beta<0$.
Moreover, $\gamma \equiv \alpha \beta$ only depends on $(A,B)$ and not on
the choice of the canonical basis.
\end{lemma}

\begin{proof}
Let $(A,B)$ be such that $\tr A$, $\tr B > 2$.
Introduce coordinates so that $u_A = \R(1,0)$ and $u_B = \R(0,1)$.
Then $A$ and $B$ are in the form~\eqref{e.canonical}, with
$\mu$, $\nu > 1$.
Write the other eigendirections as $s_A = \R (x,1)$ and $s_B = \R (1,y)$.
We have
$$
x = \frac{-\alpha}{\mu-\mu^{-1}}, \qquad
y = \frac{-\beta}{\nu-\nu^{-1}} \, .
$$
Then~\eqref{e.twist order} holds iff $xy<0$, that is, iff $\alpha \beta <0$.

We leave the cases where $A$ or $B$ is parabolic as
exercises to the reader.

For the last remark, notice that $\alpha \beta$ is a function
of $\tr A$, $\tr B$, and $\tr AB$.
\end{proof}

Let us say that $(A,B)$ is \emph{free} if
$$
\text{$|\tr A|$, $|\tr B|$, $|\tr AB| \geq 2$,
and $\tr A \, \tr B \, \tr AB < 0$.}
$$

\begin{lemma}\label{l.free}
Every free pair is twisted.
A pair $(A,B)$ is free iff it belongs to $\overline H_{\id}$.
\end{lemma}

\begin{proof}
If $(A,B)$ is straight then, replacing $A$ by $-A$ or $B$ by $-B$ if necessary,
we have $\tr A$, $\tr B$, $\tr AB \ge 2$, so $(A,B)$ cannot be free.

If $(A,B)$ is free, say with $\tr A$, $\tr B \ge 2$, and $\tr AB \le -2$,
then using a canonical basis we see that there exist $(\tilde A, \tilde B)$
arbitrarily close to $(A,B)$ such that
$\tr \tilde A$, $\tr \tilde B > 2$, and $\tr \tilde A \tilde B < -2$.
\end{proof}

\begin{lemma}\label{l.iterate}
Let $(A,B)$ be twisted.
Then exactly one of the following holds:
\begin{enumerate}
\item $(A,AB)$ is twisted.
\item $(BA,B)$ is twisted.
\item $(A,B)$ is free.
\item $AB$ is elliptic.
\end{enumerate}
\end{lemma}

\begin{proof}
If (iv) holds then clearly (i), (ii), and (iii) do not hold.
It follows from Proposition~\ref{p.free} and Lemma~\ref{l.free} that
if (iii) holds then (i) and (ii) do not hold.
Thus we only have to prove that if $(A,B)$ is twisted and not free
and if $AB$ is not elliptic then either (i) or (ii) holds.

We can assume that $\tr A$, $\tr B \ge 2$.
Then $\tr AB \ge 2$.
%Putting $A$ and $B$ in the canonical basis,
%it is easy to see that $AB \neq \pm \id$.
%Denote by $u_A,u_B,u_{AB}$ non-contracting eigendirections for
%$A$, $B$, $AB$ and by $s_A$, $s_B$, $s_{AB}$ the non-repelling
%eigendirections.
By taking a canonical basis for $(A,B)$, we may assume that
the expressions~\eqref{e.canonical} hold,
where we may choose $\alpha>0$ and $\beta<0$.
Notice that with that basis $u_A$ corresponds to $(1,0)$ and $u_B$ to $(0,1)$.
Let us orient $\P^1$ so that $(1,0)<(1,y)<(0,1)$ if $y>0$.
For this cyclical order, \eqref{e.twist order} holds.
It is easy to see that $AB \neq \pm \id$.

Assume that $\tr A=2$.
Then $\tr B>2$, otherwise we would have $\tr AB < 2$. % and $u_B<s_B<s_A$.
First, let us locate the fixed points of the projective action of $AB$.
It is easy to see that there is no fixed point in $[u_B,u_A]$ .
If there were a fixed point of $AB$ in $[u_A, s_B]$
then the associated eigenvalue would be negative, contradicting $\tr AB \ge 2$.
So $u_{AB}$, $s_{AB} \in (s_B,u_B)$.
It easily follows that
$$
u_A < s_{AB} \lesssim u_{AB} < s_A \lesssim u_A,
$$
and so, by Lemma~\ref{l.twist a}, $(A,AB)$ is twisted.
We have $u_{BA} = B u_{AB}$, $s_{BA} = B s_{AB} \in (s_B,u_B)$.
Notice that $(u_{BA},u_B)$ is an invariant interval for $(BA,B)$
so that $(BA,B)$ is straight.
This shows that the lemma holds if $\tr A=2$.
The same argument gives the case $\tr B=2$.

We assume from now on that $\tr A$, $\tr B>2$.
In this case we have
$u_A<s_B<u_B<s_A<u_A$.

Let us locate the eigendirections of $AB$.
None can belong to $\{u_A,u_B,s_A,s_B\}$.
It is immediate that $AB$ cannot have a fixed point in the interval $(u_B, s_A)$.
Neither can $AB$ have a fixed point in $(u_A, s_B)$,
because otherwise the associated eigenvalue would be negative, contrary to the assumptions.
So each eigendirection of $AB$ must be in one of the intervals
$(s_A, u_A)$ and $(s_B, u_B)$.
%Moreover, $u_{AB}$ and $s_{AB}$ belong to the same of those intervals,
%as is easily verified.\margem{JC: Arrumar!}

Consider the case that $u_{AB}$ belongs to $(s_B, u_B)$.
Observe that $BA$ sends $s_B$ into the interval $(u_B, s_B)$.
It follows that $s_{AB}$ also belongs to $(s_B, u_B)$, and also
$$
u_A < s_{AB} \lesssim u_{AB} < s_A < u_A.
$$
So $(A,AB)$ is twisted, by Lemma~\ref{l.twist a}.
The points $u_{BA} = B u_{AB}$ and $s_{BA} = B s_{AB}$ also belong to $(s_B, u_B)$.
The interval $(u_{BA}, u_B)$ is invariant for $BA$ and $B$,
so $(BA,B)$ is straight.

In the case that $u_{AB}$ belongs to $(s_A, u_A)$,
then $u_{BA} = A^{-1} u_{AB}$ also belongs to the same interval.
It follows as in the last case (interchanging the roles of $A$ and $B$) that
$(BA,B)$ is twisted and $(A,AB)$ is straight.
\end{proof}

\subsection{Dynamics of the Monoid}\label{ss.dynamics}

Let $I(A,B)=(\tr A,\tr B,\tr AB)$ and let
\begin{align*}
\phi_+(x,y,z) &=(x,z,xz-y),\\
\phi_-(x,y,z) &=(z,y,yz-x),\\
j(x,y,z)   &= x^2+y^2+z^2-xyz,
\end{align*}

\begin{prop}\label{p.invariant}
We have $I \circ F_\pm = \phi_\pm \circ I$ and $j \circ \phi_\pm=j$.
\end{prop}

\begin{proof}
The first assertion follows from the identity
$\tr A^2 B = \tr A \; \tr AB - \tr B$.
The second one is straightforward.
\end{proof}

Let $J = j \circ I$.

Let $(A,B)$ be twisted with $\tr A \geq 2$ and $\tr B \geq 2$,
so that in a canonical basis
$$
A=\begin{pmatrix} \mu&\alpha\\ 0&\mu^{-1} \end{pmatrix}, \quad
B=\begin{pmatrix} \nu^{-1}&0\\ \beta&\nu \end{pmatrix},
$$
with $\mu \geq 1$, $\nu \geq 1$ and $\gamma \equiv \alpha\beta<0$.
Then $\tr AB = \mu \nu^{-1}+\mu^{-1} \nu+\gamma$.
Thus
\begin{equation}\label{e.nonincreasing}
\tr AB \leq \max (\tr A,\tr B)+\gamma < \max (\tr A,\tr B).
\end{equation}
Moreover, we have
$$
J(A,B) = 4+\gamma^2-\gamma(\mu-\mu^{-1})(\nu-\nu^{-1})>4.
$$

Let us say that $(A,B)$ is \emph{almost hyperbolic} if $F(A,B)$ is a pair of
non-elliptic matrices for every $F \in \cM$.  The following is the key fact
we need about the action of $F$:

\begin{lemma}\label{l.stop}
Let $(A,B)$ be almost hyperbolic and twisted.
Then there exists a unique $F \in \cM$ such that the pair $F(A,B)$ is free.
Moreover, the length of $F$ in terms of the generators $F_+$, $F_-$ is
$\le \frac{1}{4} \left(|\tr A| + |\tr B|\right) - 1$.
\end{lemma}

\begin{proof}
We may assume that $\tr A \ge 2$ and $\tr B \ge 2$.
Let $(A_0, B_0) = (A,B)$.
Assume that it was defined an almost hyperbolic and twisted
pair $(A_k,B_k)$, for some $k>0$.
Then, by Lemma~\ref{l.iterate}, there are 3 possibilities:
\begin{equation} \label{e.options}
\text{either $F_+(A_k,B_k)$ is twisted, or $F_-(A_k,B_k)$ is twisted,
or $(A_k, B_k)$ is free.}
\end{equation}
In the first, resp.~second, alternative we set
$\eps_k =+$, resp.~$\eps_k = -$, and
$(A_{k+1}, B_{k+1})= F_{\eps_k}(A_k,B_k)$.

\medskip

We claim that the third alternative in~\eqref{e.options}
holds for some $k>0$.
If not, we have an (infinite) sequence
of twisted pairs $(A_k, B_k)$.
Then $\tr A_k \ge 2$ and $\tr B_k \ge 2$ for all $k \geq 0$.
In a canonical basis we have
$$
A_k=\begin{pmatrix} \mu_k&\alpha_k\\ 0&\mu^{-1}_k \end{pmatrix}, \quad
B_k=\begin{pmatrix} \nu^{-1}_k&0 \\ \beta_k&\nu_k \end{pmatrix}.
$$
Define sequences
$$
M_k = \max( \tr A_k,\tr B_k), \quad
m_k = \min( \tr A_k,\tr B_k), \quad \text{and} \quad
t_k=\tr A_k+\tr B_k
$$
Since $(A_k,B_k)$ is twisted, $\gamma_k=\alpha_k \beta_k<0$.
So, by~\eqref{e.nonincreasing}, $\{M_k\}$ is non-increasing.

Let also
$$
\Delta_k = t_{k+1} -2t_k + t_{k-1}, \quad k>0.
$$
Using Proposition~\ref{p.invariant}, one easily checks that
\begin{equation}\label{e.laplacian}
\Delta_k =
\begin{cases}
(\tr A_k - 2) \tr B_k           &\text{if $(\epsilon_k, \epsilon_{k+1}) = (+,+)$,}\\
(\tr B_k - 2) \tr A_k           &\text{if $(\epsilon_k, \epsilon_{k+1}) = (-,-)$,}\\
\tr A_k \tr B_k-\tr A_k-\tr B_k &\text{if $(\epsilon_k, \epsilon_{k+1}) = (-,+)$ or $(+,-)$.}
\end{cases}
\end{equation}
In particular, $\Delta_k \ge (m_k - 2) M_k \ge 0$,
so the function $k \mapsto t_k$ is convex.
Since $4 \le t_k \le 2 M_0$, we conclude that
$\{t_k\}$ is non-increasing and $\Delta_k \to 0$
(indeed $\sum \Delta_k<\infty$).
It follows that $\lim m_k = 2$.
The proof now splits in two cases:

First case: $\lim M_k > 2$.
Assume $\lim \tr A_k =2$ and $\lim \tr B_k>2$ (the other possibility being analogous).
We get from~\eqref{e.laplacian} that $\epsilon_k=+$ for all $k$ big enough.
Thus $A_{k+1}=A_k$ for all big $k$ and $\tr A_k=2$ for big $k$.
So $\Delta_k=0$ for big $k$.
Since $\{t_k\}$ is bounded we have, for all big $k$, that
$t_{k+1}=t_k$ and hence $\tr B_{k+1}=\tr B_k$.
But $\tr B_{k+1}=\tr B_k+\gamma_k<\tr B_k$ for big $k$, contradiction.

Second case: $\lim M_k = 2$.
Then $\tr A_k$, $\tr B_k$, $\tr A_k B_k \to 2$, so
$J(A_k,B_k) \to 4$.
This contradicts $J(A_k,B_k)=J(A,B)>4$.

\medskip

We conclude that the third alternative in~\eqref{e.options} holds
for some $k=N$, say.
That is, if $F = F_{\eps_{N-1}} \circ \cdots F_{\eps_0}$
then $F(A,B)$ is free.
Such $F \in \cM$ is unique.
Indeed, if $0\le j <N$ and $\delta \neq \eps_j$
then $F_{\delta}\circ F_{\eps_{j - 1}} \cdots F_{\eps_0} (A,B)$
is straight.
(This follows from uniqueness in Lemma~\ref{l.iterate}.)
And $F_+ (F(A,B))$ and $F_-(F(A,B))$ are also straight.

\medskip

To complete the proof, we have to bound $N$.
Since $\tr A_N$, $\tr B_N \ge 2$, and $\tr A_N B_N \le -2$,
we have $t_{N+1} - t_N \le -4$.
For $1 \le k \le N$  we have $\Delta_k \ge 0$ and so $t_k - t_{k-1} \le -4$.
Thus $t_0 \ge 4N + t_N \ge 4N+4$, so
$N \le \frac{1}{4}t_0 - 1$, as claimed.
\end{proof}

Now we can give the:
\begin{proof}[Proof of Proposition~\ref{p.disjoint}]
First, $\overline{H_0} \cap \overline{H_\id} = \emptyset$,
and since $F\left(\overline{H_0}\right) \subset \overline{H_0}$,
we have $\overline{H_0} \cap \overline{H_F} = \emptyset$ for any $F\in \cM$.
By Proposition~\ref{p.is hyperbolic}, we have
$$
\overline{H_0} \sqcup \bigcup_{F \in \cM} \overline{H_F}
\ \subset \  \overline{\cH} \  \subset \  \cE^c.
$$
On the other hand, let $(A,B) \in \cE^c$.
If the pair $(A,B)$ is straight, then it belongs to $\overline{H_0}$.
If it is not, then it is twisted and almost hyperbolic.
So Lemma~\ref{l.stop} gives that there exists $F \in \cM$ such that
$(A,B) \in \overline{H_F}$.
Moreover, $F$ is unique.
This shows that the sets $\overline{H_F}$ are disjoint,
so the first assertion in the proposition is proved.
The second one follows from the length estimate in Lemma~\ref{l.stop}.
\end{proof}

\subsection{Conclusion of the Proofs}\label{ss.conclusion}

\begin{proof}[Proof of Theorems \ref{t.components}, \ref{t.full2 boundary},
and \ref{t.H and E}]
First let us see that
\begin{equation}\label{e.H}
\cH = H_0 \sqcup \bigsqcup_{F \in \cM} H_F \, .
\end{equation}
The $\supset$ inclusion follows from Proposition~\ref{p.is hyperbolic}.
To show the other inclusion, it suffices, by Proposition~\ref{p.disjoint},
to show that $\partial H_0$, $\partial H_F \subset \cH^c$ for all $F\in\cM$.

The boundary of $H_0$ is described by Proposition~4 in \cite{Yoccoz_SL2R}:
if $(A,B)$ belongs to it then
either $A$ is parabolic or $B$ is parabolic or $u_A= s_B$ or $u_B = s_A$.
In any case, $(A,B) \in \cH^c$.

By definition of $H_\id$, if $(A,B)$ belongs to its boundary then
at least one of $A$, $B$, or $AB$ is parabolic.
It follows that if $(A,B) \in \partial H_F$ then there is
a product of $A$'s and $B$'s which is parabolic.
In particular,  $(A,B) \in \cH^c$.

We have proved equality~\eqref{e.H} and hence Theorem~\ref{t.components}.

Notice that the four principal components have disjoint boundaries,
and so do the eight free components (this follows easily from
Proposition~\ref{p.free}.)
So, by Proposition~\ref{p.disjoint},
the boundaries of the components of $\cH$ are disjoint,
and a compact set in $\G^2$ intersects only
a finite number of components.
It follows that the union of those boundaries gives all of $\partial \cH$.
This completes the proof of Theorem~\ref{t.full2 boundary}.

We have also shown that $\G^2 = \cE \sqcup \overline{\cH}$.
To complete the proof of Theorem~\ref{t.H and E},
it suffices to show that $\cH^c \subset \overline{\cE}$.
That is an immediate consequence of Lemma~2 from \cite{Yoccoz_SL2R}.
\end{proof}

\begin{rem}
Our proof of Theorem~\ref{t.components} also gave an algorithm to decide
whether a pair $(A,B) \in \G^2$ is uniformly hyperbolic or not (w.r.t.~the full $2$-shift).
Namely:
first, check if both $A$ and $B$ are hyperbolic;
second, compute eigendirections of $A$, $B$ to see if the pair belongs to a principal component;
third, repeat the first step for all pairs
$F_{\epsilon_k} \circ \cdots \circ F_{\epsilon_1}(A,B)$, with
$k \le \frac{1}{2} \max\{|\tr A|, |\tr B|\} - 1$.
(By the way, this third step can be done without actually computing matrix products,
if we use Proposition~\ref{p.invariant} instead.)
The algorithm ends in ``finite time'';
moreover, given an upper bound for the size of the matrices,
an upper bound for the ``running time'' of the algorithm can be given explicitly.
An example of \S\ref{ss.bifurcation} (see Proposition~\ref{p.bifurcation})
shows that the situation for the full $3$-shift is much more complicated.
\end{rem}

%** Smale decidability/computability ? **
%
%\begin{question}
%Given an $N$-tuple $A$ of matrices and a fixed SFT $\Sigma$, \margem{nao entendo bem o q estou falando!}
%is there an algorithm which decides
%in finite time whether $A$ belongs to $\cH$ or not?
%If $C$ is an upper bound for the norms of the matrices,
%is the running time of the algorithm time bounded by some function $f(C)$?
%\end{question}

%%%%%%%%%%%%%%%%%%%%%%%%%%%%%%%%%%%%%%%%%%%%%%%%%%%%%%%%%%%%%%%%%%%%%%%%%%%%%%%%%%%%
\subsection{Description of the Multicones}\label{ss.full2 combin}

Here we will give another proof of Proposition~\ref{p.is hyperbolic},
and also obtain an explicit description of the multicones for the twisted hyperbolic components.

\subsubsection{}
Let $\cM^*$ be the monoid on the generators $F_+$, $F_-$ operating on words in $A$, $B$
by the substitutions
\begin{alignat*}{2}
F_+: \quad A &\mapsto A,  &\quad B &\mapsto AB \\
F_-: \quad A &\mapsto BA, &\quad B &\mapsto B.
\end{alignat*}
(The monoid $\cM^*$ is opposite to the previously introduced $\cM$.)
We identify $\cM^*$ with $\Q \cap (0,1)$ via the canonical bijection $j$:
for $F\in \cM^*$, $j(F) = \nicefrac{p}{q}$ if $F(AB)$ has length $q$
and contains $p$ times the letter $B$.
We have $j(\id_{\cM^*}) = \nicefrac{1}{2}$.

\subsubsection{}%\margem{tchau matrizes}
For $F\in\cM^*$, with $j(F) = \nicefrac{p}{q}$, denote by $O(\nicefrac{p}{q})$ the set
of words of length $q$ deduced from $F(AB)$ by cyclic permutation.
This set can also be described in the following way:
consider the map $R_{p/q} : [0,1) \to [0,1)$,
$x \mapsto x + \nicefrac{p}{q} \bmod{1}$;
set $\theta(x) = A$ if $x\in [0,1-\nicefrac{p}{q})$ and
$\theta(x) = B$ if $x\in [1-\nicefrac{p}{q},1)$;
set $\Theta(x) = (\theta(R_{\nicefrac{p}{q}}^i (x)))_{0\le i<q}$;
the image of $\Theta$ is $O(\nicefrac{p}{q})$.

In $O(\nicefrac{p}{q})$, the first word by lexicographical order is $\Theta(0)$,
the second one is $\Theta(\nicefrac{1}{q})$ and so on until
the last word $\Theta(1-\nicefrac{1}{q})$.

\subsubsection{}
Let $F \in \cM^*$, with $j(F) = \nicefrac{p}{q}$;
let $[\nicefrac{p_0}{q_0},\nicefrac{p_1}{q_1}]$ be the Farey interval with center $\nicefrac{p}{q}$.
Recall that
\begin{equation} \label{e.farey}
p_0 + p_1 = p, \quad
q_0 + q_1 = q, \quad
p_1 q_0 - p_0 q_1 = 1.
\end{equation}
Then $O(\nicefrac{p_0}{q_0})$ is the set of words deduced from $F(A)$ by cyclic permutation,
and $O(\nicefrac{p_1}{q_1})$ is similarly the set of words deduced from $F(B)$ by cyclic permutation.
Here, we extend the definition of $O(\nicefrac{p}{q})$ setting
$O(\nicefrac{0}{1})=\{A\}$ and $O(\nicefrac{1}{1})=\{B\}$.

It follows from \eqref{e.farey} that 
$R_{p/q}^{q_1}(0) = R_{p/q}^{-q_0}(0) =  1-\nicefrac{1}{q}$.
Set
\begin{align*}
O_1(\nicefrac{p}{q})
&= \{\Theta(R^i_{p/q}(0)); \; 0< i < q_1\}, \\
O_0(\nicefrac{p}{q})
&= \{\Theta(R^{-i}_{p/q}(0); \; 0 < i < q_0\};
\end{align*}
we have thus defined a partition of $O(\nicefrac{p}{q}) \setminus \{\Theta(0), \Theta(1-\nicefrac{1}{q})\}$.

\subsubsection{}
Let $F$, $\nicefrac{p}{q}$, $\nicefrac{p_0}{q_0}$, $\nicefrac{p_1}{q_1}$ be as above.
We define a cyclical order on
$O(\nicefrac{p}{q}) \sqcup O(\nicefrac{p_0}{q_0}) \sqcup O(\nicefrac{p_1}{q_1})$.

For this cyclical order, the two sets
$O(\nicefrac{p}{q})$ and $O(\nicefrac{p_0}{q_0}) \sqcup O(\nicefrac{p_1}{q_1})$,
both of cardinality $q$, alternate.
The two intervals bounded by
$\Theta(0)$ and $\Theta(1-\nicefrac{1}{q})$
are
$O(\nicefrac{p_1}{q_1}) \sqcup O_1(\nicefrac{p}{q})$
and
$O(\nicefrac{p_0}{q_0}) \sqcup O_0(\nicefrac{p}{q})$;
morevover the element that succeds $\Theta(0)$ is in the former interval.
The order induced on $O(\nicefrac{p_1}{q_1})$ or $O_1(\nicefrac{p}{q})$
is the lexicographical order, while
the order induced on $O(\nicefrac{p_0}{q_0})$ or $O_o(\nicefrac{p}{q})$
in the antilexicographical order.
See Figure~\ref{f.order} with $\nicefrac{p}{q} = \nicefrac{2}{5}$.

\begin{figure}[hbt]
\begin{center}
\psfrag{0}{{\tiny $BABAA$}}
\psfrag{1}{{\tiny $BA$}}
\psfrag{2}{{\tiny $ABABA$}}
\psfrag{3}{{\tiny $AB$}}
\psfrag{4}{{\tiny $AABAB$}}
\psfrag{5}{{\tiny $AAB$}}
\psfrag{6}{{\tiny $ABAAB$}}
\psfrag{7}{{\tiny $ABA$}}
\psfrag{8}{{\tiny $BAABA$}}
\psfrag{9}{{\tiny $BAA$}}
\psfrag{A}{{\tiny $A$}}
\psfrag{B}{{\tiny $B$}}
\includegraphics[width=4.5cm]{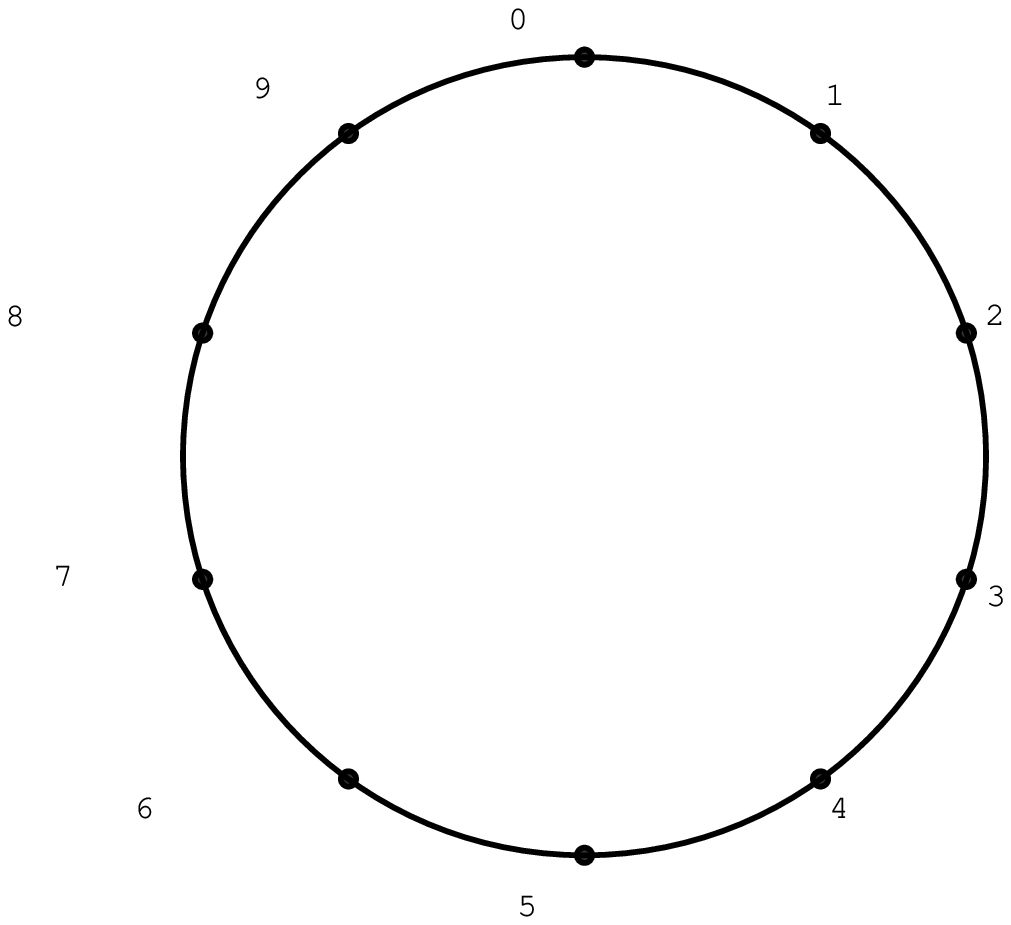}
\caption{{\small Order on $O(\nicefrac{2}{5}) \sqcup O(\nicefrac{1}{3}) \sqcup O(\nicefrac{1}{2})$.}}
\label{f.order}
\end{center}
\end{figure}

Let us give a more explicit description of this cyclical order:
\begin{lemma}
Let $\omega$ be an element in $O(\nicefrac{p}{q})$,
and denote by $\omega^-$, $\omega^+$ the elements
(in $O(\nicefrac{p_0}{q_0}) \sqcup O(\nicefrac{p_1}{q_1})$)
which are immediately before and after $\omega$ for the cyclical order.
Denote by $\Theta_0$, $\Theta_1$ the maps defined as $\Theta$ with respect to
$\nicefrac{p_0}{q_0}$, $\nicefrac{p_1}{q_1}$.
Then the following holds:
\begin{itemize}
\item If $\omega = \Theta(R^i_{p/q}(0))$ with $0\le i < q_1$ then
$\omega^+ = \Theta_1(R^i_{p_1/q_1}(0))$;

\item if $\omega = \Theta(R^i_{p/q}(1-\nicefrac{1}{q}))$ with $0\le i < q_0$ then
$\omega^+ = \Theta_0(R^i_{p_0/q_0}(1-\nicefrac{1}{q_0}))$;

\item if $\omega = \Theta(R^{-i}_{p/q}(0))$ with $0\le i < q_0$ then
$\omega^- = \Theta_0(R^{-i}_{p_0/q_0}(0))$;

\item if $\omega = \Theta(R^{-i}_{p/q}(1-\nicefrac{1}{q}))$ with $0\le i < q_1$ then
$\omega^- = \Theta_1(R^{-i}_{p_1/q_1}(1-\nicefrac{1}{q_1}))$.
\end{itemize}
\end{lemma}

\begin{proof}
From~\eqref{e.farey} we obtain $\nicefrac{p_1}{q_1} - \nicefrac{p}{q} = \nicefrac{1}{q_1 q}$.
It follows that given $i$, $j$ with $0 \le i, j < q_1$,
the point $R^i_{p/q}(0)$ is before $R^j_{p/q}(0)$ (for the usual order in $[0,1)$)
if and only if the point $R^i_{p_1/q_1}(0)$ is before $R^j_{p_1/q_1}(0)$.
Therefore the first assertion of the lemma holds.
The others are proven similarly.
\end{proof}

Define some special words
$$
\omega_A   = \Theta (\nicefrac{p}{q})       \, , \quad
\omega_B   = \Theta (\nicefrac{(p-1)}{q})   \, , \quad
{}_B\omega = \Theta (1-\nicefrac{p}{q})     \, , \quad
{}_A\omega = \Theta (1-\nicefrac{(p+1)}{q})     \, .
$$
From the description of the cyclical order, we see that
the words respectively
starting with $A$,
starting with $B$,
ending with $A$,
ending with $B$
form the intervals
$$
{}^A O = [{}_B \omega^+ , {}_A \omega]  \, , \quad
{}^B O = [{}_A \omega^+ , {}_B \omega]  \, , \quad
O^A = [\omega_A, \omega_B^-]\, , \quad
O^B = [\omega_B, \omega_A^-]\, .
$$
Observe that for $0< \nicefrac{p}{q} < \nicefrac{1}{2}$, the
union of $O^A$ and ${}^A O$ is the full set
$O(\nicefrac{p}{q}) \sqcup O(\nicefrac{p_0}{q_0}) \sqcup O(\nicefrac{p_1}{q_1})$,
and these intervals intersect at both ends.

\subsubsection{}
We assume now that $\nicefrac{p}{q} \neq \nicefrac{1}{2}$.
If $\nicefrac{p}{q} < \nicefrac{1}{2}$ (resp.\ $\nicefrac{p}{q} > \nicefrac{1}{2}$) then
we can write $F = F_+ F'$ (resp.\ $F_- F'$), with $F' \in \cM^*$,
$j(F') = \nicefrac{p}{(q-p)}$ (resp.\ $j(F')= \nicefrac{(2p-q)}{p}$\,).

Assume for instance that $\nicefrac{p}{q} < \nicefrac{1}{2}$.
Write $\nicefrac{p'}{q'} = \nicefrac{p}{(q-p)}$,
and let $[\nicefrac{p'_0}{q'_0}, \nicefrac{p'_1}{q'_1}]$
be the Farey interval which has $\nicefrac{p'}{q'}$ as center;
we have
$$
\frac{p'_0}{q'_0} = \frac{p_0}{q_0-p_0} \, , \qquad
\frac{p'_1}{q'_1} = \frac{p_1}{q_1-p_1} \, .
$$

\begin{lemma}\label{l.1}
The image of $O(\nicefrac{p'}{q'}) \sqcup O(\nicefrac{p'_0}{q'_0}) \sqcup O(\nicefrac{p'_1}{q'_1})$
under $F_+$ is exactly the interval ${}^A O$;
moreover $F_+$ preserves the cyclical orders.
\end{lemma}

\begin{proof}
Consider the map induced by $R_{p/q}$ on $[0, 1-\nicefrac{p}{q})$;
it is equal to
\begin{alignat*}{2}
x &\mapsto x + \nicefrac{p}{q}      &\quad &\text{if } 0 \le x < 1 - \nicefrac{2p}{q} \, ,\\
x &\mapsto x + \nicefrac{2p}{q} - 1 &\quad &\text{if } 1 - \nicefrac{2p}{q}  \le x < 1 - \nicefrac{p}{q} \, .
\end{alignat*}
%$$
%x \mapsto
%\begin{cases}
%x + {\textstyle \nicefrac{p}{q}}
%&\text{ if } 0 \le x < 1 - {\textstyle \nicefrac{2p}{q}}\, ,\\
%x + {\textstyle \nicefrac{2p}{q}} - 1
%&\text{ if } 1 - {\textstyle \nicefrac{2p}{q}} \le x < 1 - {\textstyle \nicefrac{p}{q}} \, .
%\end{cases}
%$$
Conjugating by the homothety of ratio $\nicefrac{(q-p)}{q}$,
we obtain $R_{p'/q'}$ on $[0,1)$.
This shows that the image of $O(\nicefrac{p'}{q'})$ under $F_+$
is the interval of $O(\nicefrac{p}{q})$ formed by the words
$\Theta(R_{p/q}^i(0))$ such that $R_{p/q}^i(0) \in [0,1-\nicefrac{p}{q})$,
i.e.~the words that start with $A$.
The other conclusions of the lemma are proved similarly.
One should observe that for $\epsilon=0,1$,
$F_+(O_\epsilon(\nicefrac{p'}{q'}))$ is the intersection of
$F_+(O(\nicefrac{p'}{q'}))$ with $O_\epsilon(\nicefrac{p}{q})$.
\end{proof}

\subsubsection{} %\margem{matrizes voltam}
For $F \in \cM^*$, denote by $H_F^+$ the set of
$(A,B) \in \G^2$ such that $(F(A), F(B))$ belongs to the positive free component 
(which is described by  Proposition~\ref{p.free}).

\begin{prop}\label{p.us rules}
Let $(A,B) \in H_F^+$.
For any $\omega \in O(\nicefrac{p}{q}) \sqcup O(\nicefrac{p_0}{q_0}) \sqcup O(\nicefrac{p_1}{q_1})$,
the corresponding matrix is hyperbolic.
Moreover, the stable directions $s(\omega)$ and unstable directions $u(\omega)$
are all distinct and
are positioned according to the following rules:
\begin{itemize}
\item for any $\omega \in O(\nicefrac{p}{q})$, $s(\omega)$ is immediately after $u(\omega)$;
\item for any $\omega \in O(\nicefrac{p_0}{q_0}) \sqcup O(\nicefrac{p_1}{q_1})$,
$s(\omega)$ is immediately before $u(\omega)$;
\item the restriction of the cyclical order to the $u(\omega)$ is the
cyclical order considered above.
\end{itemize}
(It follows from these three rules that the same is true for the
restriction to the $s(\omega)$.)
\end{prop}

\begin{proof}
The first assertion is clear.
If $j(F) = \nicefrac{1}{2}$, the cyclical order is the one described above.
Assume $j(F) = \nicefrac{p}{q} \neq \nicefrac{1}{2}$,
for instance $\nicefrac{p}{q} < \nicefrac{1}{2}$.
We write $F= F_+ F'$, $\nicefrac{p'}{q'} = \nicefrac{p}{(q-p)}$ as above.
Let $A' = A$, $B' = AB$.
We prove the proposition by induction, thus we may assume that the conclusions are satisfied
for $(A',B') \in H^+_{F'}${\,}.
This means that
the points $\{u(\omega), s(\omega) ; \; \omega \in {}^A O\}$ are all distinct and
the restriction of the cyclical order to this set
is in accordance with the proposition.
Let $a: O^A \to {}^A O$ be the bijection which takes the final letter $A$ into
first position;
this map corresponds to $A$ in the sense that
$$
A u(\omega) = u(a\omega), \quad
A s(\omega) = s(a\omega), \quad \omega\in O^A
$$
and therefore the restriction of the cyclical order to the set
$\{ {u(\omega), s(\omega);} \; {\omega \in O^A}\}$ is also in accordance with the proposition.
As ${}^A O$, $O^A$ are intervals which cover
$O(\nicefrac{p}{q}) \sqcup O(\nicefrac{p_0}{q_0}) \sqcup O(\nicefrac{p_1}{q_1})$
and have non-empty intersection at both ends,
the points
$\{u(\omega), s(\omega) ; \; \omega \in O(\nicefrac{p}{q}) \sqcup O(\nicefrac{p_0}{q_0}) \sqcup O(\nicefrac{p_1}{q_1})\}$ are all distinct and
there is only one cyclical order with the given restrictions,
which is the one described in the proposition.
\end{proof}

\subsubsection{}\label{sss.full2 combin end}
Now we give the other proof of Proposition~\ref{p.is hyperbolic}.
It is sufficient to show that any $(A,B) \in H_F^+$ is uniformly hyperbolic.
We will apply Lemma~\ref{l.core full}
and therefore we will define sets $U$ and $S$ satisfying the required conditions.

For $\omega \in O(\nicefrac{p}{q})$, we define intervals
$I^u_\omega = [u(\omega^-), u(\omega)]$,
$I^s_\omega = [s(\omega), s(\omega^+)]$.
Let
$U = \bigcup_{\omega \in O(\nicefrac{p}{q})} I^u_\omega$,
$S = \bigcup_{\omega \in O(\nicefrac{p}{q})} I^s_\omega$.
Then $U$, $S$ are disjoint compact subsets with finitely many components which alternate.
To apply Lemma~\ref{l.core full}, we need to check that
$AU \cup BU \subset U$,
$A^{-1}S \cup B^{-1}S \subset S$.
Indeed, we have:
\begin{itemize}
\item $A(I^u_\omega) = I^u_{a\omega}$                          for $\omega_A < \omega < \omega_B$;
\item $A(I^u_\omega) \subset I^u_{\Theta(0)}$                  for $\omega_B \le \omega \le \omega_A$;
\item $B(I^u_\omega) = I^u_{b\omega}$                          for $\omega_B < \omega < \omega_A$;
\item $B(I^u_\omega) \subset I^u_{\Theta(1- \nicefrac{p}{q})}$ for $\omega_A \le \omega \le \omega_B$.
\end{itemize}
(The map $b: O^B \to {}^B O$ is defined analogously as $a$,
by switching a letter $B$ from the last to the first place.)
This proves that $AU$ and $BU$ are disjoint and contained in $U$;
it also follows that no non-trivial product of $A$, $B$
is equal to $\pm \id$.
Similar formulas hold for $A^{-1}$, $B^{-1}$ and the intervals $I^s_\omega$.
Thus we can apply Lemma~\ref{l.core full} and conclude that
$(A,B)$ is uniformly hyperbolic.
The sets $U$ and $S$ are of course the unstable and stable cores,
and the formulas above give the action of $A$, $B$ on the components of the associated multicone.
Both $U$ and $S$ have $q$ components,
and the set $O(\nicefrac{p}{q}) \sqcup O(\nicefrac{p_0}{q_0}) \sqcup O(\nicefrac{p_1}{q_1})$
is in canonical correspondence with the connected coimponents of the complement of $U \sqcup S$:
see Figure~\ref{f.description}.

\begin{figure}[hbt]
\begin{center}
\psfrag{s0}[l][l]{{\tiny $s(BABAA)$}}
\psfrag{u0}[l][l]{{\tiny $u(BABAA)$}}
\psfrag{u1}[l][l]{{\tiny $u(BA)$}}
\psfrag{s1}[l][l]{{\tiny $s(BA)$}}
\psfrag{s2}[l][l]{{\tiny $s(ABABA)$}}
\psfrag{u2}[l][l]{{\tiny $u(ABABA)$}}
\psfrag{u3}[l][l]{{\tiny $u(AB)$}}
\psfrag{s3}[l][l]{{\tiny $s(AB)$}}
\psfrag{s4}[l][l]{{\tiny $s(AABAB)$}}
\psfrag{u4}[l][l]{{\tiny $u(AABAB)$}}
\psfrag{u5}[r][r]{{\tiny $u(AAB)$}}
\psfrag{s5}[r][r]{{\tiny $s(AAB)$}}
\psfrag{s6}[r][r]{{\tiny $s(ABAAB)$}}
\psfrag{u6}[r][r]{{\tiny $u(ABAAB)$}}
\psfrag{u7}[r][r]{{\tiny $u(ABA)$}}
\psfrag{s7}[r][r]{{\tiny $s(ABA)$}}
\psfrag{s8}[r][r]{{\tiny $s(BAABA)$}}
\psfrag{u8}[r][r]{{\tiny $u(BAABA)$}}
\psfrag{u9}[r][r]{{\tiny $u(BAA)$}}
\psfrag{s9}[r][r]{{\tiny $s(BAA)$}}
\includegraphics[width=5cm]{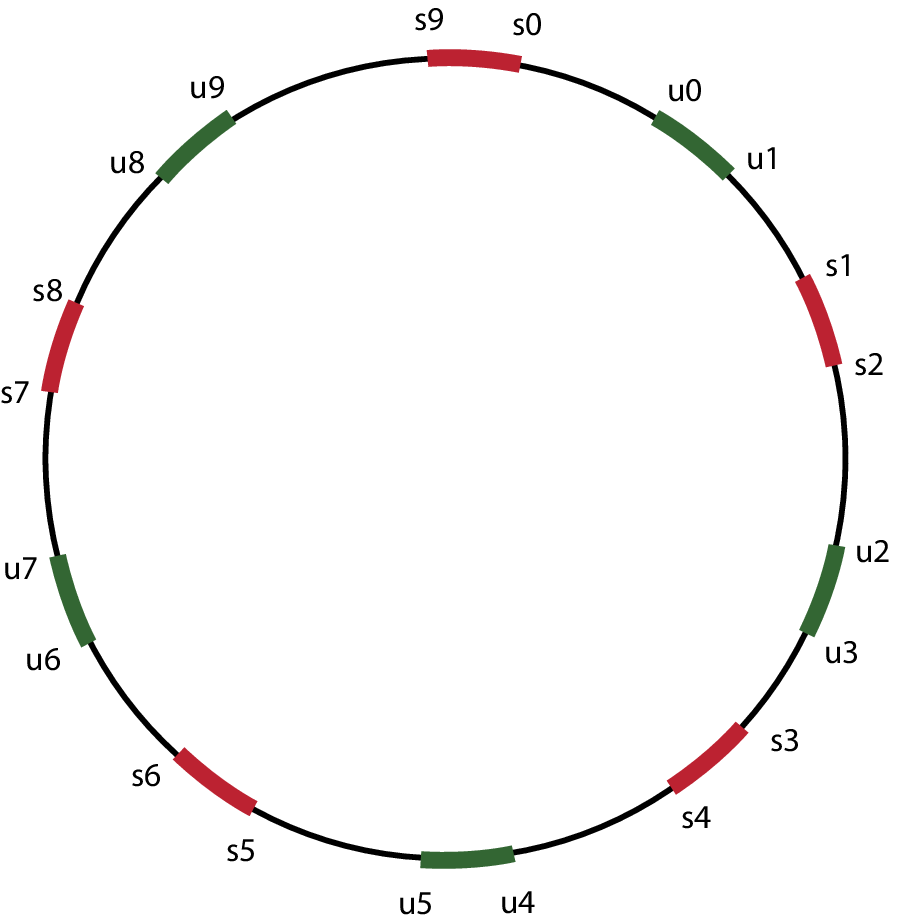}
\caption{{\small The intervals $I^u_\omega$, $I_s^\omega$ for $\nicefrac{p}{q} = \nicefrac{2}{5}$.}}
\label{f.description}
\end{center}
\end{figure}

%% file: aby_boundary.tex
%%%%%%%%%%%%%%%%%%%%%%%%%%%%%%%%%%%%%%%%%%%%%%%%%%%%%%%%%%%%%%
\section{Boundaries of the Components}\label{s.boundaries}
%%%%%%%%%%%%%%%%%%%%%%%%%%%%%%%%%%%%%%%%%%%%%%%%%%%%%%%%%%%%%%

\subsection{A General Theorem on Boundary Points}

Again, fix any subshift of finite type $\Sigma \subset N^\Z$,
and let $\cH \subset \G^N$ be the associated hyperbolicity locus.

Given $x = (x_i)_{i \in\Z} \in \Sigma$, we denote
$$
W^u_\mathrm{loc}(x) = \{(z_i) \in \Sigma; \; z_i = x_i \text{ for } i < 0\} , \quad
W^s_\mathrm{loc}(x) = \{(z_i) \in \Sigma; \; z_i = x_i \text{ for } i \ge 0\} .
$$

The next result describes the boundary points of connected components of $\cH$.

\begin{thm} \label{t.general boundary}
Let $(A_1, \ldots, A_N)$ belong to the boundary of a connected component
$H$ of $\cH$.
Then one of the following possibilities holds:
\begin{enumerate}
\item
There exists a periodic point $x\in \Sigma$ of period $k$ such that
$A^k(x) = \pm \id$.

\item (\textit{``parabolic periodic''})
There exists a periodic point $x\in \Sigma$ of period $k$ such that
$A^k(x) \neq \pm \id$ is parabolic;

\item (\textit{``heteroclinic connection''})
There exist periodic points $x$ and $y \in \Sigma$,
of respective periods $k$ and $\ell$,
such that the matrices $A^k(x)$ and $A^\ell(y)$ are hyperbolic
and there exist an integer $n\ge 0$
and a point $z \in W^u_\mathrm{loc} (x) \cap \sigma^{-n} W^s_\mathrm{loc} (y)$
such that
\begin{equation} \label{e.hetero connection}
A^n (z) \cdot u( A^k(x) ) = s ( A^\ell(y) ) \, .
\end{equation}
\end{enumerate}

Furthermore, for each component $H$, one can give uniform bounds
to the numbers $k$, $\ell$, $n$ that may appear in the alternatives above.
\end{thm}

In alternative (iii), there exists a point $z = (z_i)_{i\in \Z}$
such that $z_{-k-1} = z_{-1}$, $z_{n+\ell} = z_n$, and
$$
A_{z_{n-1}} \cdots A_{z_0} \cdot u( A_{z_{-1}} \cdots A_{z_{-k}} ) =
s( A_{z_{n+\ell-1}} \cdots A_{z_n} ) \, .
$$
That is what we call a \emph{heteroclinic connection}
(provided $A_{z_{-1}} \cdots A_{z_{-k}}$ and $A_{z_{n+\ell-1}} \cdots A_{z_n}$ are hyperbolic).
%**remark: given $z$ as above, we can take $n$ as the smallest possible.

\begin{rem}
In alternative~(iii), the periodic points $x$ and $y$ cannot belong to the
same periodic orbit.
\end{rem}

\begin{proof}
Assume the contrary, so $k = \ell$ and $x = \sigma^j(y)$ for some $j$ with $0 \le j < k$.
%Since $A^k(x) = A^j(y) \cdot A^k(y) \cdot [A^j(y)]^{-1}$,
Then
$$
s(A^k(y)) = A^n(z) \cdot u(A^k(x))
          = A^n(z) \cdot A^j(y) \cdot u(A^k(y))
          = A^{n+j}(\sigma^{-j} z) \cdot u(A^k(y)).
$$
So, writing $A = A^k(y)$ and $B = A^{n+j}(\sigma^{-j} z)$,
we have that $A$ is hyperbolic and $B \cdot u(A) = s(A)$.
A direct calculation shows that $\lim_{m \to +\infty} \tr A^m B =0$.
Therefore there is $m>0$ such that $A^m B = A^{km+n+j}(\sigma^{-j} z)$ is elliptic.
Since $z_{km+n} = z_{-j}$,
this contradicts the assumption that the $N$-tuple belongs to the boundary of~$\cH$.
\end{proof}

\begin{rem}
If $\Sigma$ is the full-shift, and $H$ is a principal component,
then by Proposition~4 in~\cite{Yoccoz_SL2R} one can take $n=0$, $k = \ell=1$ in
alternative~(iii) of Theorem~\ref{t.general boundary}.
\end{rem}

\begin{rem}\label{r.no id}
We will see later (Proposition~\ref{p.no id}) that in the case of full shifts,
alternative~(i) in Theorem~\ref{t.general boundary} is only possible
if $H$ is a principal component.
\end{rem}

Theorem~\ref{t.general boundary} has the following interesting consequence:

\begin{corol}\label{c.semialgebraic}
Every connected component of $\cH$ is a semialgebraic set.
\end{corol}

Notice $\cH$ itself is not semialgebraic, because it has
infinitely many connected components (see Theorem~2.4.5 from \cite{BochnakCosteRoy}).

\begin{proof}[Proof of the corollary]
Of course, $\G^N$ itself is a (semi) algebraic subset of $\R^{4N}$.

Let $H$ be a connected component of $\cH$.
Let $K$ be the upper bound on the numbers $k$, $\ell$, $n$ that appear in
Theorem~\ref{t.general boundary}.
Let $S_1$, $S_2$, and $S_3$ be the subsets of $\G^N$
formed by the $N$-tuples
that satisfy respectively alternatives (i), (ii), and (iii) of the theorem,
with $k$, $\ell$, $n$ not greater than $K$.

The set $S_1 \cup S_2$ is obviously semialgebraic;
let us see that $S_3$ also is.
Introduce variables $\lambda$, $\mu \in \R$, $w_1$, $w_2 \in \R^2$,
and rewrite~\eqref{e.hetero connection} as
$$
\left \{
\begin{array}{ll}
A^k(x) \cdot w_1    = \lambda w_1 & \quad \lambda^2>1 \\
A^\ell(y) \cdot w_2 = \mu w_2     & \quad -1<\mu<1 \\
A^n(z) \cdot w_1 = w_2            & \quad w_1 \neq (0,0)
\end{array}
\right.
$$
Such relations define a semialgebraic set on $\G^N \times \R^6$,
which is sent by the obvious projection onto $S_3$.
Therefore $S_3$ is semialgebraic, by the Tarski-Seidenberg principle
(see \cite{BochnakCosteRoy}, Theorem~2.2.1).

The set $S = S_1 \cup S_2 \cup S_3$ is closed, disjoint from $H$, and contains the boundary of $H$.
Thus $H$ is a connected component of the semialgebraic set
$\G^N \setminus S$,
and hence is semialgebraic, by Theorem~2.4.5 from \cite{BochnakCosteRoy}.
\end{proof}

\medskip

To prove Theorem~\ref{t.general boundary}, we first establish two lemmas.
In both of them we assume that $(A_1, \ldots, A_N)$ belongs to the hyperbolic locus,
and let $U_\alpha$, $S_\alpha$ be its unstable and stable families of cores
(see \S\ref{ss.onlyif}).

\begin{lemma}\label{l.preimage}
Let $\beta$ be a symbol, and $v\in \partial U_\beta$.
Then there exist a symbol $\alpha$ such that $\alpha \to \beta$
and $A_\beta^{-1} (v) \in \partial U_\alpha$.
\end{lemma}

\begin{proof}
Recalling the definition of $U_\beta$, we see that
the condition $v \in \partial U_\beta$ is equivalent to the following:
\begin{quote}
$v\in K^u_\beta$ and
there exist a point $w \in A_\beta K^s_\beta$ and an open interval $I \subset \P^1$
such that $\partial I = \{ v, w\}$ and $I \cap K^u_\beta = \emptyset$.
\end{quote}
Let $v$, $w$, and $I$ be as above.
Take $x = (x_i)_{i \in \Z} \in \Sigma$
such that $x_{-1} = \beta$ and $e^u(x) = v$.
Let $\alpha = x_{-2}$.
Set $v' = A_\beta^{-1} (v)$, $w' = A_\beta^{-1} (w)$,
and $I' = A_\beta^{-1} (I)$.
We have
$v' \in K^u_\alpha$,
$w' \in A_\alpha K^s_\alpha$, %(because $K^s_\beta \subset A_\alpha K^s_\alpha$);
and $I' \cap K^u_\alpha = \emptyset$. %(because $A_\beta K^u_\alpha \subset K^u_\beta$).
We conclude that $v' \in \partial U_\alpha$.
\end{proof}

\begin{lemma}\label{l.preperiodic}
Let $v \in \partial U_\alpha$.
Then there exist a periodic point $x \in \Sigma$ of period $k$,
a point $z \in W^u_\mathrm{loc} (x)$,
and an integer $m \ge 0$ such that $z_{m-1} = \alpha$ and
$$
v = A^m (z) \cdot u( A^k(x) ) \, .
$$

Analogously, if $v' \in \partial S_\alpha$
then there exist a periodic point $y \in \Sigma$ of period $\ell$,
a point $w \in W^s_\mathrm{loc} (x)$,
and an integer $p \ge 0$ such that $w_{-p}=\alpha$ and
$$
v' = A^{-p} (w) \cdot s( A^\ell(y) ) \, .
$$
Moreover, $k$, $m$, $\ell$, and $p$ are less or equal than the rank
of the families of cores.
\end{lemma}

\begin{proof}
We will prove one half of the lemma.
Take $v \in \partial U_\alpha$.
Set $\alpha_0 = \alpha$ and $v_0 = v$.
Applying repeatedly Lemma~\ref{l.preimage} we
find a sequence
$\alpha_0 \leftarrow \alpha_1 \leftarrow \alpha_2 \leftarrow \cdots$ %\margem{ordem oposta aa usual}
such that
$$
v_{n+1} = A_{\alpha_n}^{-1} \cdots A_{\alpha_0}^{-1} v_0 \in \partial U_{\alpha_{n+1}}
\quad \text{for every $n \ge 0$.}
$$
Let $n_0$ be the rank of the family $U_\alpha$.
By the pigeon-hole principle, there exist
integers $m$ and $k$  such that $0 \le m < m+k \le n_0$ and
$\alpha_m = \alpha_{m+k}$ and $v_m = v_{m+k}$.
Then $v_m$ is fixed by $A_{\alpha_m} \cdots A_{\alpha_{m+k-1}}$,
and so must be the unstable direction of this matrix product.
We also have $v_0 = A_{\alpha_0} \cdots A_{\alpha_{m-1}} \cdot v_m$.
The lemma follows.
\end{proof}

\begin{proof}[Proof of Theorem~\ref{t.general boundary}]
Observe that unstable and stable families of cores vary continuously with the $N$-tuple.
So if we restrict ourselves to $N$-tuples in~$H$,
the rank $n_0$ of the families of cores is constant.

Now take $(A_1, \ldots, A_N)$ in the boundary of $H$.
Assume that there is no periodic point $x \in \Sigma$ of period $n \le n_0$
for which $A^n(x) = \pm \id$.
We will show that then one of the alternatives (ii) or (iii) in the theorem holds.

Consider the following finite subsets of $\P^1$:
\begin{equation}\label{e.finite}
\begin{gathered}
U^*_\alpha = \{ A^m(z) \cdot u(A^k(x)) ; \;
1\le k \le n_0, \ 0 \le m \le n_0, \ x = \sigma^k x, \ z \in W^u_\mathrm{loc}(x), \ z_{m-1}=\alpha \}, \\
S^*_\beta = \{ A^{-p}(w) \cdot s(A^\ell(y)) ; \;
1\le \ell \le n_0, \ 0 \le p \le n_0, \ y = \sigma^\ell y, \ w \in W^s_\mathrm{loc}(y), \  w_{-p}=\beta \}.
\end{gathered}
\end{equation}
Notice that
\begin{equation}\label{e.invariance}
U_\beta^* \subset \bigcup_{\alpha ; \; \alpha \to \beta} A_\beta U_\alpha^*
\quad \text{and} \quad
S_\alpha^* \subset \bigcup_{\beta ; \; \alpha \to \beta} A_\alpha^{-1} S_\beta^* \, .
\end{equation}
(To see this, use for instance that
if $x=\sigma^k x$ then $u(A^k(x)) = A_{x_{-1}} \cdot u(A^k(\sigma^{-1} x))$.)

Assume that $U^*_\alpha \cap S^*_\beta \neq \emptyset$ for some $\alpha$, $\beta$ with $\alpha \to \beta$.
Then, for some $m$, $x$ etc as in \eqref{e.finite}, we have an equality
$A^m(z) \cdot u(A^k(x)) = A^{-p}(w) \cdot s(A^\ell(y))$.
Moreover, we can assume that $w = \sigma^n z$, where $n = m+p$.
Then $A^n(z) \cdot u(A^k(x)) = s(A^\ell(y))$,
with $z \in W^u_\mathrm{loc}(x) \cap \sigma^{-n} W^s_\mathrm{loc}(y)$.
If $A^k(x)$ or $A^\ell(y)$ is parabolic, we are in alternative~(ii) of the theorem.
Otherwise, both $A^k(x)$ and $A^\ell(y)$ are hyperbolic and
alternative~(iii) holds.

In order to complete the proof of the theorem,
we will assume by contradiction that
$U^*_\alpha \cap S^*_\beta = \emptyset$ for every $\alpha$, $\beta$ with $\alpha \to \beta$.
It follows from~\eqref{e.invariance} that
$U^*_\alpha \cap A_\alpha S^*_\alpha = \emptyset$ for every~$\alpha$.

Take a sequence $(A_1(i), \ldots, A_N(i))$ in $H$
converging to $(A_1, \ldots, A_N)$ as $i \to \infty$.
Let $A_\alpha(\infty) = A_\alpha$.

Define sets $U^*_\alpha(i)$, $S^*_\alpha(i)$
in the same way $U^*_\alpha$, $S^*_\alpha$ were defined,
replacing each $A_\beta$ with $A_\beta(i)$.
By continuity of the $u$ and $s$ directions for non-elliptic matrices
far from $\pm \id$,
we have that for every large $i$,
$U^*_\alpha(i)$ and $S^*_\alpha(i)$,
are close to $U^*_\alpha$ and $S^*_\alpha$, respectively.

For $i\in \N \cup\{\infty\}$, define other sets $U_\alpha(i)$, $S_\alpha(i)$
as follows:
$U_\alpha(i)$ is the complement of the union of the connected
components of $\P^1 \setminus U^*_\alpha(i)$ that intersect
$A_\alpha S^*_\alpha(i)$,
and $S_\alpha(i)$ is the complement of the union of the connected
components of $\P^1 \setminus S^*_\alpha(i)$ that intersect
$A_\alpha^{-1} U^*_\alpha(i)$.
If $I$ is large enough then $U_\alpha(i)$ and $S_\alpha(i)$
are respectively close (with respect to the Hausdorff distance)
to $U_\alpha(\infty)$ and $S_\alpha(\infty)$.

By Lemma~\ref{l.preperiodic}, if $i<\infty$ then $U_\alpha(i)$ and $S_\alpha(i)$
are precisely the unstable and stable families of cores of the $N$-tuple $(A_\alpha(i))$.
It follows from continuity that the sets
$U_\alpha = U_\alpha(\infty)$, $S_\alpha = S_\alpha(\infty)$
also satisfy properties (i)-(iv) of \S\ref{ss.onlyif}.
By Lemma~\ref{l.core}, $(A_\alpha)$ has a family of multicones,
that is, $(A_\alpha) \in \cH$.
Contradiction.
\end{proof}

\medskip

\begin{center}
\emph{From this point until the end of Section~\ref{s.abstract combinatorics},
we will be interested only in full shifts.}
\end{center}

\subsection{Non-Principal Components} \label{ss.nonprincipal}

%To give the limit cores, Proposition~\ref{p.limit cores} requires that
%no $\pm$identity products exist.
%We shall see that condition is automatically satisfied
%on the whole boundary of any non-principal component.

As mentioned in Remark~\ref{r.no id},
we will prove that no $\pm$identity products exist in the boundaries of non-principal components.

Let us begin with a lemma about pairs of matrices.
Recall that a uniformly hyperbolic pair induces maps
$e^u$, $e^s: 2^\Z \to \P^1$ (see \S\ref{ss.onlyif}).

\begin{lemma}\label{l.u interval}
For every $c>0$ there exists $\delta=\delta(c)>0$ with the following properties:
If $(A,B)$ is a uniformly hyperbolic pair with
\begin{equation}\label{e.weak}
\|A\| \le c \quad \text{and} \quad \|B \mp \id \| < \delta
\end{equation}
then $(A,B)$ belongs to a principal component.
Moreover, the images of the maps $e^u$, $e^s$
are (disjoint closed) \emph{intervals} $I_u$, $I_s \subset \P^1$.
\end{lemma}

\begin{proof}
Our study of the $N=2$ case shows that the boundary of a non-principal component
cannot contain a pair of the form $(A, \pm \id)$.
If follows that there exists $\delta=\delta(c)$ such that
every hyperbolic pair $(A,B)$ satisfying~\eqref{e.weak}
belongs to a principal component.

Let us also assume that $\delta(c)$ is small enough so that~\eqref{e.weak} implies
$$
\inf_{x\in\P^1} \left|(A^{\pm 1})'(x)\right| +
\inf_{x\in\P^1} \left|(B^{\pm 1})'(x)\right| > 1.
$$

Now, given a hyperbolic pair $(A,B)$ satisfying~\eqref{e.weak}, let
$I_u$ and $I_s$ be disjoint closed intervals
such that $\partial I_u = \{u_A, u_B\}$ and $\partial I_s = \{s_A, s_B\}$.
By the choice of $\delta>0$, we have
$|A(I_u)| + |B(I_u)| > |I_u|$ (where $|\mathord{\cdot}|$ denotes interval length).
Therefore
$$
I_u = A(I_u) \cup B(I_u).
$$

Let us write $A_1 = A$, $A_2 = B$.
Given $z_0 \in I_u$, there exists $x_{-1} \in \{1,2\}$ and $z_1 \in I_u$ such that
$A_{x_{-1}} (z_1) = z_0$.
Inductively, we find $x_{-n} \in \{1,2\}$ and $z_n \in I_u$ such that
$A_{x_{-n}} (z_n) = z_{n-1}$.
We form a sequence $x=(x_i)_{i\in\Z} \in 2^\Z$, choosing arbitrarily $x_i$ for $i \ge 0$.
Then it is easy to see that $z_0 = e^u(x)$.
This shows that $e^u(2^\Z)=I_u$.
The proof that $e^s(2^\Z)=I_s$ is analogous.
\end{proof}

Let $\cH_\mathrm{NP} \subset \G^N$ be the union of the non-principal components.

\begin{prop}\label{p.no id}
If an $N$-tuple is in $\overline{\cH_\mathrm{NP}}$
then no product of the matrices in the $N$-tuple equals~$\pm \id$.

Furthermore, for every compact subset $K$ of $\G^N$, there exists
a neighborhood $V$ of $\{\pm \id\}$ such that if
an $N$-tuple belongs to $K \cap \overline{\cH_\mathrm{NP}}$
then no product of the matrices in the $N$-tuple belongs to $V$.
\end{prop}

\begin{proof}
Given $c>1$, let $\delta = \delta(2c)$ be given by Lemma~\ref{l.u interval}.
For a compact set of the form
$K(c) = \{(A_1,\ldots,A_N) \in \G^N ; \; \|A_i\| \le c\}$,
we will take $V$ as the open neighborhood of $\{\pm \id\}$
of size $\delta$.

Fix an $N$-tuple $\xi_0 \in K(c) \cap \overline{\cH_\mathrm{NP}}$.
By contradiction, assume that there exists a product of the matrices in $\xi_0$ which is
$\delta$-close to $\pm \id$.

Take $\xi = (A_1, \ldots, A_N)\in \cH_\mathrm{NP}$ close to $\xi_0$.
If $\xi$ is close enough to $\xi_0$,
there exists a product of the $A_i$'s, say $B$, which is $\delta$-close to $\pm \id$.

Fix some cyclical order on $\P^1$.
Since $\xi$ is not in a principal component, there exist
$i,j,k,\ell \in \{1,\ldots,N\}$ such that
$$
u(A_i) < s(A_j) < u(A_k) < s(A_\ell) < u(A_i).
$$
Lemma~\ref{l.u interval} applied to the pair $(A_i,B)$
implies that there is an interval containing $u(A_i)$ and $u(B)$, and disjoint from
$\{s(A_j), s(A_\ell)\}$; in particular $u(B)$ must belong to the interval
$(s(A_\ell), s(A_j))$.
A symmetric argument gives $u(B) \in (s(A_j), s(A_\ell))$.
We reached a contradiction.
\end{proof}

Next, let us prove that connected components of cores associated to a $N$-tuple
in a non-principal component are non-degenerate intervals:

\begin{lemma}\label{l.non degenerate}
Fix a non-principal component $H \subset \G^N$, and let $K \subset \G^N$
be a compact set.
Then there exists $\delta>0$ such that
for any $\xi \in H \cap K$, each interval composing the unstable or stable cores
of $\xi$ has length at least $\delta$.
\end{lemma}

\begin{proof}
Assume that there exists $\xi \in H \cap K$ whose unstable core $U$
has a connected component $I$ which is very small.
Recalling Proposition(s)~\ref{p.combin contraction} (and \ref{p.tight}),
there exists a product $B$ of matrices in $\xi$
such that $B(U) \subset I$.
Moreover, we can give an upper bound for $\|B\|$ depending on $H$ and $K$ only.
If follows that the diameter of $U$ is small.
Consider the shortest closed interval that contains $U$.
That interval is forward-invariant by each matrix in $\xi$.
This implies that $\xi$ is in a principal component, contradiction.
\end{proof}

\subsection{Limit Cores} %\margem{poderia ser uma secao}

%\subsubsection{Cores for $N$-tuples at the Boundary of a Component}

The proof of Theorem~\ref{t.general boundary} gives
some useful information about the families of cores.
We will register that information for later use,
however we will focus on the case of full shifts, where cores are defined differently
(see \S\ref{sss.full core}).

The analogue of Lemmas~\ref{l.preimage} and \ref{l.preperiodic} for full shifts are the following:

\begin{lemma}\label{l.preimage full}
Let $(A_1, \ldots, A_N)$ be uniformly hyperbolic w.r.t.\ the full shift,
and let $U$ be the unstable core.
For any $v\in \partial U$,
then there exists a symbol $i$ such that $A_i^{-1} (v) \in \partial U$.
\end{lemma}

The proof is analogue to that of Lemma~\ref{l.preimage}, but let us give it for the reader's convenience:
\begin{proof}
Let $v \in \partial U$; then $v \in K^u$, so $v= e^u(x)$.
Let $v' = A_i^{-1} (v)$ where $i = x_{-1}$; then $v' = e^u(\sigma^{-1}(x)) \in K^u$.
Since $v \in \partial U$, there is an open interval $I$ disjoint from $K^u$ with endpoints $v$ and $w \in K^s$.
Then the open interval $I'= A_i^{-1}(I)$ is disjoint from $K^u$, has one endpoint $v'$ in $K^u$ and the
other in $K^s$.
This implies that $v' \in \partial U$.
\end{proof}

From the lemma one easily gets:

\begin{lemma}\label{l.preperiodic full}
Let $(A_1, \ldots, A_N)$ be uniformly hyperbolic w.r.t.\ the full shift,
and let $U$ and $S$ be the unstable and stable cores.
Let $v \in \partial U$.
Then
$$
v = A_{i_m} \cdots A_{i_1} \cdot u( A_{j_k} \cdots A_{j_1}) \, .
$$
for some choice of indices. ($m$ can be zero, meaning that $v = u( A_{j_k} \cdots A_{j_1})$.)
Analogously, if $v' \in \partial S$
then
$$
v' = A_{i_1'}^{-1} \cdots A_{i_p'}^{-1} \cdot s( A_{j_\ell'} \cdots A_{j_1'} ) \, .
$$
for some choice of indices. ($p$ can be zero.)
Moreover, $k$, $m$, $\ell$, and $p$ are less or equal
than the rank of~$U$.
\end{lemma}

Using the last lemma, one shows:

\begin{prop}\label{p.limit cores}
Let $H$ be a connected component of the hyperbolic locus relative to the full shift on $N$ symbols.
For each $i \in \N$, let $(A_1(i), \ldots, A_N(i)) \in H$
have unstable core $U(i)$ and stable core $S(i)$.
Suppose that $(A_1(i), \ldots, A_N(i))$ converges
to some $(A_1, \ldots A_N)$ in the boundary of $H$ as $i \to \infty$.
Also assume every product of the $A_j$'s of length less or equal than the rank of the cores
is different from $\pm \id$.
Then the sets $U(i)$ and $S(i)$ converge (with respect to the Hausdorff distance)
as $i \to \infty$, say to sets $U$ and $S$.
Moreover, the intersection $U \cap S$ is finite and non-empty.
\end{prop}

We call the sets $U$ and $S$ given by the proposition
the \emph{limit cores} of $(A_1, \ldots, A_N)$.

If $H$ is a non-principal component then, by Proposition~\ref{p.no id},
the no $\pm\id$ assumption in Proposition~\ref{p.limit cores} is satisfied;
hence the limit cores are well-defined for each point in the boundary of $H$.
Moreover, we have:

\begin{prop}\label{p.independence}
If an $N$-tuple belongs to the boundaries of two
different non-principal components,
then the respective limit cores are precisely the same.
%Moreover, no connected component  of $U$ or $S$ can be a point.
\end{prop}

However, we do not know if the boundaries of
two different components can meet.
%As a consequence of the proposition,
%the unstable and stable cores can be continuously extended
%to the union of all the closures of non-principal components.
%NAO! ESQUECEU Q COMPONENTES PODEM SE ACUMULAR?

\begin{proof}[Proof of the proposition]
Fix an $N$-tuple $(A_1,\ldots, A_N)$ in the closure of a non-principal component $H$.
Let $U$ and $S$ be the limit cores with respect to~$H$.

Let $K^u_*$ be the set of all points of the form $u_P$ or $Q(u_P)$,
where $P$ and $Q$ are products of the  $A_i$'s.
(Recall that $u_P$ is defined, by Proposition~\ref{p.no id}.)
Analogously, let $K^s_*$ be the set of all $s_P$ and $Q^{-1}(s_P)$.
Then $K^u_* \subset U$ and $K^s_* \subset S$.
Also, by Lemma~\ref{l.preperiodic full}, $\partial U \subset K^u_*$ and $\partial S \subset K^s_*$.

We claim that no point in $K^u_*$ is isolated.
Indeed, consider a point $x = Q (u_P)$.
By Lemma~\ref{l.non degenerate}, %\margem{A much weaker lemma would suffice.}
$\partial U$, and hence $K^u_*$, contains at least $4$ points.
In particular, we can find $y \in K^u_*$ different from $u_P$ and from $s_P$.
The sequence $Q P^n (y)$ is contained in $K^u_* \setminus \{x\}$ and converges to $x$.
This shows that $x$ is not isolated.
Symmetrically, no point in $K^s_*$ is isolated.

It follows from these facts that
the complement of the union of the connected components of
$\P^1 \setminus \overline{K^u_*}$ (resp.\ $\P^1 \setminus \overline{K^s_*}$)
that intersect $K^s_*$ (resp.\ $K^u_*$)
is precisely $U$ (resp.~$S$).
This procedure describes $U$ and $S$ without referring to $H$,
so the proposition follows.
\end{proof}

%\margem{Il y a une definition directe de $U$, $S$ (limit cores) qui suit de Lemma~\ref{l.preperiodic full}.}

\subsection{An Addendum for the Full $2$-Shift}

In the light of the general results about boundaries obtained so far,
let us come back to the case of the full \emph{two}-shift
and give some additional information
complementing Theorem~\ref{t.full2 boundary}:

\begin{prop}\label{p.addendum 2 shift}
Let $H$ be a \emph{non-principal}
connected component of the hyperbolic locus relative to the full shift on two symbols.
Then:
\begin{enumerate}
\item
No $\pm$identity products exist for a pair on the boundary of $H$.

\item
No heteroclinic connection occurs on the boundary of $H$.

\item
There are only three words (other than their cyclic permutations and powers)
that can become parabolic on the boundary of $H$.
\end{enumerate}
\end{prop}

\begin{proof}
Let $H$ be a twisted component.
Assertion (i) follows from Proposition~\ref{p.no id}.
For $(A_0,A_1) \in H$,
the cores $U$ and $S$ are described precisely in \S\ref{sss.full2 combin end} --
in particular, we have:
\begin{itemize}
\item[(a)] The sets $\partial U$ and $\partial S$
are respectively formed by unstable and stable directions of
certain ``special'' products of $A_0$'s and $A_1$'s.

\item[(b)] If points $v \in \partial U$ and $w \in \partial S$
are ``neighbors'' (in the sense that there is an open interval
with endpoints $v$ and $w$ that does not meet $U \cup S$)
then they are respectively the unstable and stable directions of the \emph{same}
``special'' product of $A_0$'s and $A_1$'s.

\item[(c)] There are three words in the letters $A_0$ and $A_1$ which are not powers
and that form, together with their cyclic permutations,
the full list of special words that need to be considered in (b) and (c).

\item[(d)] No connected component of $U$ intersects both $A_0(U)$ and $A_1(U)$.
%Analogously for $S$.
\end{itemize}

%Let $n_0$ be the rank, that is, the number of connected components of $U$ (or $S$),
%which is constant over $H$.

It follows from (d) and  Lemma~\ref{l.preimage full} that:
\begin{itemize}
\item[(e)] For every $v \in \partial U$ there exist a unique $i\in \{0,1\}$
such that $A_i^{-1} (v) \in \partial U$.
%Analogously for $S$.
\end{itemize}
Repeated application of (e) gives:
\begin{itemize}
\item[(f)] For any $v_0 \in \partial U$,
there exists a unique sequence $i_1$, $i_2$, \dots in $\{0,1\}$
such that $v_{j+1} =  A_{i_j}^{-1} (v_j) \in \partial U$.
\end{itemize}
Now it follows from (a) that:
\begin{itemize}
\item[(g)] For any $v_0 \in \partial U$, if $v_j$ is the sequence given by (f)
and $\ell$ is the least positive integer such that $v_\ell \in \{v_0, \ldots, v_{\ell - 1}\}$
then $v_\ell = v_0$.
\end{itemize}

Now let $(A_0,A_1)$ be in the boundary of $H$,
and let $U$ and $S$ be the limit cores given by Proposition~\ref{p.limit cores}
(which are well-defined because $H$ is not principal).
By Lemma~\ref{l.non degenerate}, $U$ and $S$ have
the same number of components as before taking the limit,
and none of these components is a point.
It follows that Properties (d) and (e) above are also true for the limit cores.
Property~(f) follows from (e).
So (g) makes sense for the limit cores, and it is true by continuity.

Any $v_0 \in\partial U$ equals $u(P)$ where $P = A_{i_1} \cdots A_{i_\ell}$
and the indices $i_j$ are as in (f) and (g).
The word $P$ is not a power, and so is one the special words alluded in (a)--(c).
Let $w_0 \in \partial S$ be the neighbor of $v_0$.
(Precisely, we define $w_0$ as $v_0$ if $v_0 \in S$,
otherwise we let $w_0 \in S$ be so that there is
an open interval with endpoints $v_0$ and $w_0$ that does not intersect $U \cup S$.)
We infer from property~(b) that $w_0 = s(P)$.
In particular, $v_0 \in S$ implies that $P$ is parabolic.

Now, suppose $v_0$ is also given by
$R u(Q)$, where $Q$ and $R$ are words in the letters $A_0$ and $A_1$,
with  $R$ allowed to be the empty word (corresponding to $\id$ product).
It follows from uniqueness in~(f) that
the infinite words $RQQQ \dots$ and $PPP\dots$ must coincide.
In particular, $Q$ is (as a word) a power of a cyclic permutation of $P$.
Therefore $Q$ is parabolic (as a matrix)  if and only if so is~$P$.

By contradiction, assume there is a heteroclinic connection $R u(Q) = s(P')$,
for some products $P'$, $Q$, $R$, of $A_0$'s and $A_1$'s.
Then $v_0  = R u(Q)$ belongs to $U\cap S$.
Therefore, as we have seen, $Q$ has to be parabolic.
This is forbidden by definition of heteroclinic connection, so assertion~(ii) of the theorem is proved.
Assertion~(iii) follows similarly.
\end{proof}

%%%%%%%%%%%%%%%%%%%%%%%%%%%%%%%%%%
%NOW, THE EXAMPLES

\subsection{An Example of Heteroclinic Connection} \label{ss.ex connection}

In this subsection, we introduce what is probably the simplest example
of heteroclinic connection for a principal component.
The base dynamics is full-shift on $3$ symbols.
The component $H$ of the hyperbolicity locus $\cH$
is the one that contains triples $(A,B,C)$ such that
$(A,B) \in H_\id^+$ (the positive free component for the full-shift on two symbols)
and $C = -AB$;
such triples are indeed obviously uniformly hyperbolic.
The associated stable and unstable cores have two components.

\begin{prop}\label{p.hetero ex}
A triple $(A,B,C)$ belongs to $H$ iff the following conditions are satisfied:
\begin{enumerate}
\item $(A,B) \in H_\id^+$;
\item $\tr C > 2$;
\item the stable and unstable directions for $C$ satisfy
$$
s_A < u_C < s_{AB} \, , \quad
u_{AB} < s_C < u_B \, , \quad
s_A < u_C < s_C < u_B \, .
$$
\item $s_A < Cu_B < u_C$.
\end{enumerate}
\end{prop}

\begin{figure}[htb]
\psfrag{A}[][c]{{\footnotesize $A$}}
\psfrag{B}[][c]{{\footnotesize $B$}}
\psfrag{C}[][c]{{\footnotesize $C$}}
\psfrag{AB}[][c]{{\footnotesize $AB$}}
\psfrag{BA}[][c]{{\footnotesize $BA$}}
\psfrag{BA}[][c]{{\footnotesize $BA$}}
\psfrag{C-1sA}[][l]{{\footnotesize $C^{-1} s_A$}}
\psfrag{CuB}[][r]{{\footnotesize $C u_B$}}
\begin{center}
\includegraphics[width=7.0cm]{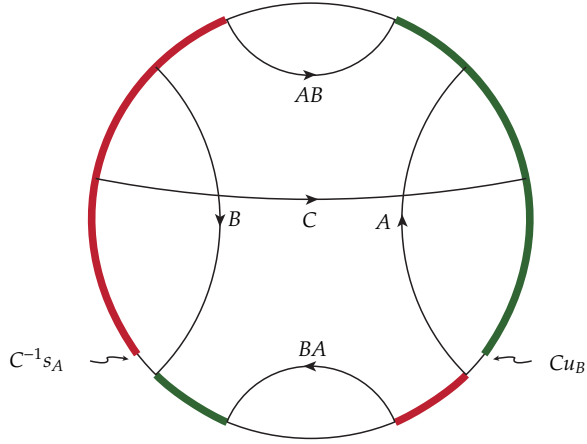}
\caption{{\small A possible situation for $(A,B,C) \in H$ in Proposition~\ref{p.hetero ex}; the cores are indicated.}}
\label{f.heteroclinic}
\end{center}
\end{figure}

\begin{proof}
Let $\hat{H}$ be the set of parameters defined by the $4$ conditions in the proposition.
Clearly, $\hat{H}$ is open in $(\SL(2,\R))^3$.
It is also clear that the boundary of $\hat{H}$ does not intersect the hyperbolicity
locus $\cH$, and that $\hat{H}$ contains any triple
$(A,B, -AB)$ with $(A,B) \in H_\id^+$.
To prove that $\hat{H} = H$, we prove that $\hat H$ is connected and contained
in $\cH$.

To see that $\hat H$ is connected, we fix $(A,B) \in H_\id^+$
and check that the set of $C$ satisfying (ii), (iii), (iv) is connected.
Indeed, the set of positions for $(u_c, s_c)$ in $\P^1 \times \P^1$
determined by (iii) is connected, and for any such position,
condition (iv) is equivalent to some condition $\tr C > k$ ($>2$).
This proves that $\hat H$ is connected.

Let $(A,B,C) \in \hat H$.
Define
\begin{alignat*}{2}
U_{AC} &= [\min(u_A, C u_B), \max(u_{AB}, u_C)], &\quad
S_A &= [\min(s_{BA}, A^{-1} s_c), s_A], \\
U_B &=[u_B, \max(u_{BA}, Bu_C)], &\quad
S_{BC} &= [\min(s_{AB}, s_C), \max(s_B, C^{-1} s_A)], \\
U &= U_{AC} \sqcup U_B, &\quad
S &= S_A \sqcup S_{BC} .
\end{alignat*}
We have then
$$
A(U) \cup C(U) \subset U_{AC}           \, , \quad B(U) \subset U_B \, , \quad
B^{-1}(S) \cup C^{-1}(S) \subset S_{BC} \, , \quad A^{-1}(S) \subset S_A \, .
$$
It follows from Lemma~\ref{l.core full}
that $(A,B,C)$ is uniformly hyperbolic (with cores $U$, $S$).
The proof is now complete.
\end{proof}

We have seen in the proof of the proposition that for fixed
$A$, $B$, $u_C$, $s_C$ satisfying (i), (ii), (iii),
the set $H$ is determined by a condition
$\tr C > k$ for some  $k = k(A,B,u_C,s_C)>2$.
If we take $C=C_0$ we still have a
triple $(A,B,C_0)$ such that  (i), (ii), (iii) are satisfied and
$C_0 u_B = s_A$.
In a neighborhood $V$ of $(A,B,C_0)$ in $(\SL(2,\R))^3$,
the equation $C u_B = s_A$ determines
a smooth hypersurface contained in the boundary of $H$.
This part of the boundary of $H$ corresponds to a heteroclinic connection.

We will investigate in the next two subsections what happens on the side
of the hypersurface not contained in $H$.
We already know from Proposition~6 in \cite{Yoccoz_SL2R}
that the other side $V \setminus \overline{H}$ intersects the
elliptic locus $\cE$
(the (open) set of triples that have an elliptic product.)
In the sequel we will construct two examples displaying different phenomena near boundary points:
\begin{itemize}
\item In one example (Proposition~\ref{p.elliptic side})
we have $V \setminus \overline{H} \subset \cE$.
\item In other example (Proposition~\ref{p.bifurcation}),
any neighborhood $V$ intersects infinitely many
hyperbolic components.
\end{itemize}
For convenience, we will assume that
$s_A < u_C < u_A$ and $s_B < s_c < u_B$
(as in Figure~\ref{f.heteroclinic}).

\subsection{Heteroclinic Connection with Elliptic Products on the Other Side} \label{ss.ex elliptic}

Let $H\subset \SL(2,\R)^3$ be the hyperbolic component introduced in \S\ref{ss.ex connection}.

\begin{prop}\label{p.elliptic side}
There there exist a point $(A_0,B_0,C_0)$ in the boundary of $H$,
and a neighborhood $V\subset \G^3$ of $(A_0,B_0,C_0)$ such that:
\begin{itemize}
\item If $(A,B,C) \in V \cap \partial H$ then $C \cdot u(B) = s (A)$.
\item If $(A,B,C) \in V \setminus \overline{H}$ then $(A,B,C) \in \cE$
(that is, there exists an elliptic product of $A$, $B$, and $C$'s).
\end{itemize}
\end{prop}

For another example with similar properties, see Proposition~7 in~\cite{Yoccoz_SL2R}.

\begin{proof}
Fix numbers $\lambda$, $\theta$, and $\nu$ such that:
\begin{equation}\label{e.lambda theta}
1 < \lambda < 1 +\sqrt{2}, \qquad
\frac{\lambda^2+1}{\lambda^2-1} < \theta < \frac{2}{\lambda - 1},
\qquad \nu > \theta \, .
\end{equation}

Define three matrices in $\G$ as follows:
$$
A_0 = \begin{pmatrix}
\lambda                       & 0 \\
-\theta(\lambda-\lambda^{-1}) & \lambda^{-1}
\end{pmatrix}, \quad
B_0 = \begin{pmatrix}
\lambda & \theta (\lambda - \lambda^{-1}) \\
0       & \lambda^{-1}
\end{pmatrix}, \quad
C_0 = \begin{pmatrix}
0 & -1 \\ 1 & \nu + \nu^{-1}
\end{pmatrix}
$$
All matrices have traces $>2$.
The stable and unstable directions are ordered as follows:
\begin{multline*}
u(B_0) = \begin{pmatrix} 1 \\ 0            \end{pmatrix} <
s(A_0) = \begin{pmatrix} 0 \\ 1            \end{pmatrix} <
u(C_0) = \begin{pmatrix} 1 \\ -\nu         \end{pmatrix} <
u(A_0) = \begin{pmatrix} 1 \\ -\theta      \end{pmatrix}<\\
<s(B_0)= \begin{pmatrix} 1 \\ -\theta^{-1} \end{pmatrix}<
s(C_0) = \begin{pmatrix} 1 \\ -\nu^{-1}    \end{pmatrix}<
u(B_0).
\end{multline*}
Also, $C_0(u(B_0)) = s(A_0)$.
Finally, due to one inequality in~\eqref{e.lambda theta} we have
$$
\tr A_0 B_0 = \lambda^2 - \theta^2 (\lambda - \lambda^{-1})^2 + \lambda^{-2} < -2.
$$
We conclude that
$(A_0,B_0,C_0)$ belongs to the boundary of the hyperbolic component $H$
described in \S\ref{ss.ex connection}.
Let $V$ be a small neighborhood of this $3$-tuple
such that %$V \cap \partial H$ is a smooth hypersurface and
$V \setminus \overline{H} = \{(A,B,C) \in V;\; u(B) < C \cdot u(B) < s(A) \}$.
To complete the proof, we will show that
this set is contained in $\cE$, provided $V$ is small enough.

For any $(A,B,C) \in V \setminus \overline{H}$,
take a basis of $\R^2$  close to the canonical basis and
formed by vectors collinear to $u(B)$, $s(A)$,
so that the matrices of $A$, $B$, and $C$ become:
$$
A = \begin{pmatrix}
\lambda_1                       & 0 \\
-\theta_1(\lambda_1-\lambda_1^{-1}) & \lambda_1^{-1}
\end{pmatrix}, \quad
B = \begin{pmatrix}
\lambda_2 & \theta_2(\lambda_2 - \lambda_2^{-1}) \\
0             & \lambda_2^{-1}
\end{pmatrix}, \quad
C = \begin{pmatrix}
t & -1 +td \\ 1 & d
\end{pmatrix},
$$
for certain numbers $\lambda_1$ and $\lambda_2$ close to $\lambda$,
$\theta_1$ and $\theta_2$ close to $\theta$,
$d$ close to $\nu + \nu^{-1}$, and
$t$ close to zero.
Since $u(B) < C(u(B)) < s(A)$,
$t$~must be positive.

We are going to look for elliptic products of the form $A^m C B^n$.
So we write
$$
A^m = \begin{pmatrix}
\lambda_1^m  & 0 \\
-\xi_1(m)    & \lambda_1^{-m}
\end{pmatrix}, \quad
B^n = \begin{pmatrix}
\lambda_2^n & \xi_2(n) \\
0           & \lambda_2^{-n}
\end{pmatrix},
\quad\text{with }
\left\{
\begin{array}{l}
\xi_1(m) = \theta_1 (\lambda_1^m - \lambda_1^{-m}), \\
\xi_2(n) = \theta_2 (\lambda_2^n - \lambda_2^{-n}).
\end{array}
\right.
$$
A computation gives
\begin{align*}
\tr A^m C B^n &=
\lambda_1^m t \lambda_2^n - \xi_1(m) t \xi_2(n) - \xi_1(m) (-1+td) \lambda_2^{-n} +
\lambda_1^{-m} \xi_2(n) + \lambda_1^{-m} d \lambda_2^{-n} \\
& = - v(m,n) t + u(m,n),
\end{align*}
where
\begin{align}
v(m,n) &= \lambda_1^m \lambda_2^n (\theta_1 \theta_2 -1)
            \left( 1+ \cO(\lambda_1^{-2m} + \lambda_2^{-2n}) \right), \label{e.v} \\
u(m,n) &= \theta_1\lambda_1^m\lambda_2^{-n} + \theta_2\lambda_1^{-m}\lambda_2^n +
            \cO(\lambda_1^{-m} \lambda_2^{-n}) \label{e.u}.
\end{align}

Choose a sequence $(m_k, n_k)$ (depending on $\lambda_1$ and $\lambda_2$ only)
starting at $(m_0,n_0)=(0,0)$,
such that for all $k$, $(m_{k+1}, n_{k+1})$ is either $(m_k + 1, n_k)$ or $(m_k, n_k + 1)$,
and
\begin{equation}\label{e.mn}
\lambda_2^{-1} \le \lambda_1^{m_k} \lambda_2^{-n_k} \le \lambda_1.
\end{equation}
Write $v_k = v(m_k, n_k)$, $u_k = u(m_k, n_k)$.
Assuming $V$ is sufficiently small,
there is some constant $k_0$ (not depending on $(A,B,C)$ in $V$)
such that $v_k > 0$ and $u_k > 2$ for every $k \ge k_0$.

Let
$$
\delta = \max \big( |\lambda_1 - \lambda|, |\lambda_2 - \lambda|,
|\theta_1 - \theta|, |\theta_2 - \theta|, | d-\nu-\nu^{-1}| \big)
$$
(Notice that $t$ does not appear above.)
Let $\cO_\delta(1)$ indicate a quantity that goes to zero as $\delta \to 0$.
It follows from \eqref{e.v}, \eqref{e.u}, and \eqref{e.mn} that
\begin{equation}\label{e.odelta}
\frac{v_{k+1}}{v_k} = \lambda + \cO_\delta(1) \, , \qquad
2\theta + \cO_\delta(1) < u_k < \theta(\lambda+\lambda^{-1}) + \cO_\delta(1) \, .
\end{equation}

For $k \ge k_0$, define intervals
$$
I_k = (\alpha_k, \beta_k) = \left( \frac{u_k-2}{v_k}, \frac{u_k+2}{v_k} \right).
$$
Each $I_k$ depends on $\lambda_1$, $\lambda_2$, $\theta_1$, $\theta_2$, and $d$,
\emph{but not on $t$}.
Also,
$$
|\tr A^{m_k} C B^{n_k} | < 2 \quad \text{iff} \quad
t \in I_k \, .
$$

We claim that if $\delta$ is sufficiently small then $I_k \cap I_{k+1} \neq \emptyset$ for all $k \ge k_0$.
Indeed, using~\eqref{e.odelta}, we get:
\begin{gather}
\frac{\alpha_k}{\beta_{k+1}} =
\frac{u_k - 2}{u_{k+1} +2} \cdot \frac{v_{k+1}}{v_k} \le
\frac{\theta (\lambda+\lambda^{-1}) - 2}{2\theta + 2} \cdot \lambda + \cO_\delta(1) \, , \label{e.babo1}\\
\frac{\alpha_{k+1}}{\beta_k} =
\frac{u_{k+1} - 2}{u_k +2} \cdot \frac{v_k}{v_{k+1}} \le
\frac{\theta (\lambda+\lambda^{-1}) - 2}{2\theta + 2} \cdot \frac{1}{\lambda} + \cO_\delta(1) \, . \label{e.babo2}
\end{gather}
From the assumption $\theta < 2/(\lambda-1)$ in \eqref{e.lambda theta},
it follows that the the right-hand side of \eqref{e.babo1}
is strictly less than $1 + \cO_\delta(1)$.
The same is true for the (smaller)   right-hand side of \eqref{e.babo2}.
Thus we have shown that if $k \ge k_0$ and $\delta$ is small enough
then $\alpha_k < \beta_{k+1}$ and $\alpha_{k+1} < \beta_k$;
in particular $I_k \cap I_{k+1} \neq \emptyset$.
Hence for small $\delta$, we have
$$
\bigcup_{k \ge k_0} I_k = \left(\liminf_{k \to \infty} \alpha_k, \, \sup_{k \ge k_0} \beta_k \right) \supset
(0, \beta_{k_0}).
$$
The number $\beta_{k_0}$ has a positive lower bound on $V$.
Therefore, reducing the neighborhood $V$ of $(A_0,B_0,C_0)$ if necessary,
we have that for any $(A,B,C) \in V \setminus \overline{H}$,
there exists some $k \ge k_0$ such that
the corresponding $t$ belongs to the corresponding $I_k$.
This means that the matrix $A^{m_k} C B^{n_k}$ is elliptic,
showing that $(A,B,C)$ belongs to the elliptic locus~$\cE$.
\end{proof}

\subsection{An Example of Accumulation of Components}\label{ss.bifurcation}

Again consider the hyperbolic component $H$ for the full shift in three symbols
that was introduced in \S\ref{ss.ex connection}.

\begin{prop}\label{p.bifurcation}
There exists a path $t \mapsto (A,B,C(t))$ with the following properties:
\begin{enumerate}
\item
$(A,B,C(t)) \in H$ for $t<0$.

\item
At the parameter $t=0$, the heteroclinic connection $C(0) \cdot u_B = s_A$ occurs;
in particular, $(A, B, C(0))$ belongs to $\partial H$.

\item
There exists a sequence of hyperbolic components $H_i$, all different,
and a sequence $t_i>0$ converging to $0$ as $i \to \infty$ such that
$(A,B,C(t_i)) \in H_i$ for all~$i$.

\item
There exist a sequence $s_i>0$ converging to $0$ as $i \to \infty$ such that
$(A,B,C(s_i))$ belongs to the elliptic locus $\cE$ for all $i$.
\end{enumerate}
\end{prop}

%If $A \in \G$ is hyperbolic, and $x$, $y \in \P^1$ are such that
%$u_A < x < y < s_A$ for some cyclical order, then
%\begin{equation} \label{e.lim cross ratio}
%[u_A, x, y, s_B] =
%\lim_{n \to + \infty} \frac{\sphericalangle (A^n x, A^n y)}{\sphericalangle (u_A, A^n x)} =
%\lim_{n \to + \infty} \frac{\sphericalangle (A^{-n} x, A^{-n} y)}{\sphericalangle (A^{-n} y, s_A)}
%\end{equation}
%This follows easily from the invariance of the cross-ratio under the action of~$A$.

\begin{proof}
Take $(A,B)$ in the positive free component of the full $2$-shift.
Assume that the order in $\P^1$ is so that
$$
u_B < s_A < u_A < s_B < u_B.
$$
Take points $p$, $q \in (u_{BA},s_{BA})$ such that
\begin{equation}\label{e.p q}
u_{BA} < BA \cdot q < p < q < s_{BA} \, .
\end{equation}
Define the following cross-ratios (recall formula \eqref{e.def cross ratio} from \S\ref{ss.if}):
\begin{equation}\label{e.cross ratios}
\alpha = [u_A, p, q, s_A] \,  , \quad
\beta = [u_B, BAq, p, s_B] \, .
\end{equation}
Then $\alpha$, $\beta>1$.
We claim that the choices of $A$, $B$, $p$, $q$ can be made so that
\begin{equation}\label{e.assumption}
(\alpha-1) (\beta-1) > 1 \, .
\end{equation}
Indeed, if $B$ is replaced with $B^T$ with $T>1$ (keeping $A$ fixed)
then $(A,B)$ remains in the free component;
moreover  \eqref{e.p q} still holds
keeping $p$, $q$ (and hence $\alpha$) fixed.
If $T$ is large enough then so is $\beta$ and \eqref{e.assumption} is satisfied.

If $\mu$, $\nu$ are the spectral radii of $A$, $B$, respectively, we also assume that
\begin{equation}\label{e.irrational}
\frac{\log \nu}{\log \mu} \not\in \Q \, .
\end{equation}

Take any smooth path $t \mapsto C(t)$ such that
$$
\tr C(t) > 2, \quad
s_{C(t)} \in (s_B, u_B), \quad
u_{C(t)} \in (s_A, u_A) \quad
\text{for all $t$},
$$
and
\begin{equation}\label{e.C}
C(0) \cdot u_B = s_A , \quad
\left. \frac{\partial}{\partial t}  C(t) \cdot u_B \right|_{t=0} < 0 \, .
\end{equation}
(In particular, $C(t) \cdot u_B$ belongs to $(s_A, u_C)$, resp.~$(s_C,s_A)$
for small negative, resp.~positive $t$.)
%$$
%C(t) \cdot u_B
%\begin{cases}
%\in (s_A, u_A) &\text{for $t<0$,}\\
%= s_A &\text{for $t=0$,}\\
%\in (s_{B A}, s_A) &\text{for $t>0$.}
%\end{cases}
%$$
By Proposition~\ref{p.hetero ex}, $(A,B, C(t))$ belongs to $H$
for all small $t<0$.
So assertions (i) and (ii) of the statement hold.

\medskip

Next define (disjoint) intervals
$$
I_n = B^n    \cdot [Bp, BAq], \qquad
J_m = A^{-m} \cdot [p,q], \quad
\text{for integers $n$, $m \ge 0$.}
$$
Define also
$$
I_n^* = [u_B, B^{n+1} Aq], \quad
\text{for $n \ge 0$.}
$$
(See Figure~\ref{f.bifurcation}.)
\begin{figure}[htb]
\psfrag{A}{{\footnotesize $A$}} \psfrag{B}{{\footnotesize $B$}}
\psfrag{AB}{{\footnotesize $AB$}} \psfrag{BA}{{\footnotesize $BA$}}
\psfrag{Cti}{{\footnotesize $C(t_i)$}}
\psfrag{I0}{{\footnotesize $I_0$}} \psfrag{In}{{\footnotesize $I_{n-1}$}}
\psfrag{Ins}{{\footnotesize $I_n^*$}} \psfrag{V}{{\footnotesize $V$}}
\psfrag{J0}{{\footnotesize $J_0$}} \psfrag{J1}{{\footnotesize $J_1$}} \psfrag{Jm}{{\footnotesize $J_m$}}
\psfrag{1}[r]{{\footnotesize $B^{n+1}Aq$}}
\psfrag{2}[r]{{\footnotesize $B^n p$}}
\psfrag{3}[c]{{\footnotesize $Bp$}}
\psfrag{4}[c]{{\footnotesize $BAq$}}
\psfrag{p}[c]{{\footnotesize $p$}}
\psfrag{q}[c]{{\footnotesize $q$}}
\psfrag{5}{{\footnotesize $A^{-m} q$}}
\psfrag{6}{{\footnotesize $C(t_i) \cdot B^n p$}}
\psfrag{7}{{\footnotesize $Aq$}}
\begin{center}
\includegraphics[width=9.0cm]{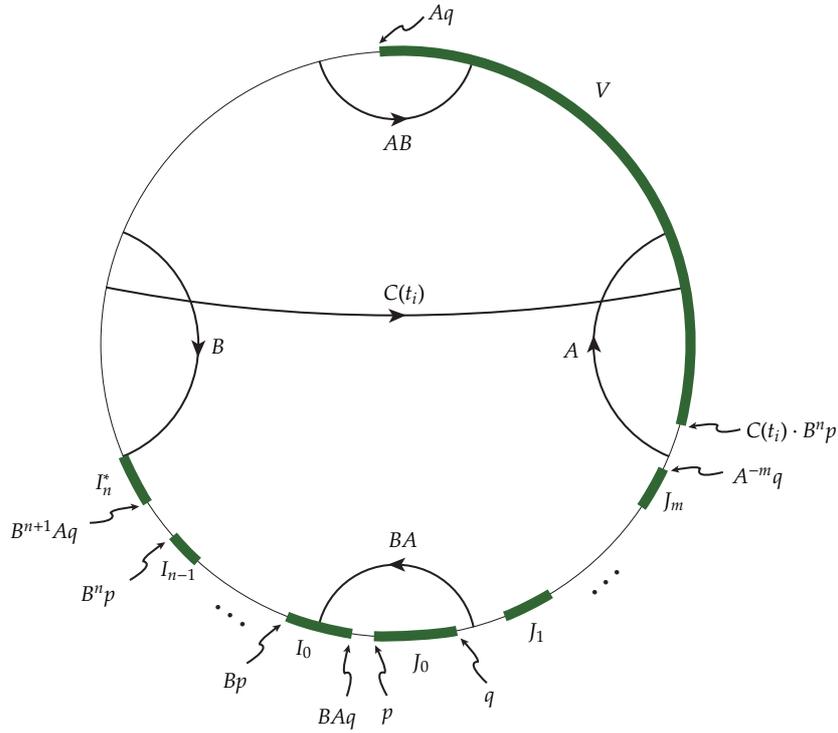}
\caption{{\small A ``non-strict'' multicone for $(A,B,C(t_i))$.}}
\label{f.bifurcation}
\end{center}
\end{figure}

In the manifold $\P^1$ we take charts using euclidian angle;
these serve to compute derivatives and speak of length of intervals.
Let $\kappa>0$ be the derivative of $C(0): \P^1 \to \P^1$ at $u_B$.
By \eqref{e.assumption}, we can find $\eps>0$ such that
\begin{equation}\label{e.choice eps}
(\alpha-1) (\beta-1) (1-2\kappa\eps) > 1  \, .
\end{equation}
We claim that
\begin{equation}\label{e.kappa}
\text{there are sequences $n_i$, $m_i \uparrow + \infty$ such that }
\kappa^{-1} - 2 \eps < \frac{|I^*_{n_i}|}{|J_{m_i}|} < \kappa^{-1} - \eps \, .
\end{equation}
Indeed, there is a projective chart (see \S\ref{ss.if})
$P : \P^1 \to \R \cup \{\infty\}$
such that $P \circ B \circ P^{-1}(t) = \nu^{-2} t$.
It follows that the limit ${\displaystyle \lim_{n \to + \infty} \nu^{2n} |I_n^*|}$ exists.
Analogously, the limit ${\displaystyle \lim_{m \to + \infty} \mu^{2m} |J_m|}$ exists.
By \eqref{e.irrational}, for any $N$ the set
$\{\mu^{2m} \nu^{-2n}; \; m,\ n > N\}$ is dense in $\R_+$.
So~\eqref{e.kappa} follows.

Define also intervals
\begin{equation}\label{e.dual intervals}
\tilde J_n   = [B^{n+1} A q , B^{n} p] \, , \qquad
\tilde I_m^* = [A^{-m} q, s_A] \, .
\end{equation}
Next we claim that
if $i$ is large enough and $t$ is sufficiently close to zero
then
\begin{align}
|C(t) \cdot I_{n_i}^*|      &< |J_{m_i}|          \, , \label{e.claim 1}\\
|C(t) \cdot \tilde J_{n_i}| &> |\tilde I^*_{m_i}| \, . \label{e.claim 2}
\end{align}
On the one hand,
$|C(t) \cdot I_{n_i}^*|/|I_{n_i}^*| \to \kappa$
as $i \to \infty$ and $t \to 0$.
So, by~\eqref{e.kappa},
$$
\limsup_{i \to \infty, \  t\to 0} \frac{|C(t) \cdot I_{n_i}^*|}{|J_{m_i}|}
 \le \kappa(\kappa^{-1} -\eps) < 1,
$$
proving~\eqref{e.claim 1}.
On the other hand,
it is easy to see that
$$
\alpha-1 = \lim_{m \to +\infty} \frac{|J_m|}{|\tilde I^*_m|} \, , \qquad
\beta-1  = \lim_{n \to +\infty} \frac{|\tilde J_n|}{|I_n^*|} \, .
$$
So we can write
\begin{alignat*}{2}
\liminf_{i \to \infty, \ t \to 0} \frac{|C(t) \cdot \tilde J_{n_i}|}{|\tilde I^*_{m_i}|}
&=   \kappa \liminf_{i \to \infty} \frac{|\tilde J_{n_i}|}{|\tilde I^*_{m_i}|}
=   (\alpha-1)(\beta-1)\kappa \liminf_{i \to \infty} \frac{|I_{n_i}^*|}{|J_{m_i}|} && \\
&\ge (\alpha-1)(\beta-1)\kappa(\kappa^{-1} -2\eps) &\quad&\text{{(by \eqref{e.kappa})}}\\
&>   1    &&\text{(by \eqref{e.choice eps}),}
\end{alignat*}
proving~\eqref{e.claim 2}.

Now, it follows from \eqref{e.C}, \eqref{e.claim 1}, and \eqref{e.claim 2} that
for every sufficiently large~$i$,
there exists a small $t_i>0$ such that
\begin{equation}\label{e.fundamental}
C(t_i) \cdot I_{n_i}^*   \Subset J_{m_i} \quad \text{and} \quad
C(t_i) \cdot I_{n_i - 1} \Subset (s_A, u_C) \, .
\end{equation}
Indeed, it is sufficient to take $t_i$ such that $C(t_i)$ maps the right endpoint of $I_{n_i}^*$
inside the interval $J_{m_i}$ and close to its right endpoint.
(See Figure~\ref{f.intervals}.)
\begin{figure}[htb]
\psfrag{uB}{{\footnotesize $u_B$}}
\psfrag{Ins}{{\footnotesize $I_{n_i}^*$}}
\psfrag{Gn}{{\footnotesize $\tilde J_{n_i}$}}
\psfrag{In}{{\footnotesize $I_{n_i - 1}$}}
\psfrag{Jm}{{\footnotesize $J_{m_i}$}}
\psfrag{Lm}{{\footnotesize $\tilde I^*_{m_i}$}}
\psfrag{Tm}{{\footnotesize $\tilde I_{m_i - 1}$}}
\psfrag{sA}{{\footnotesize $s_A$}}
\psfrag{Cti}{{\footnotesize $C(t_i)$}}
\begin{center}
\includegraphics[width=9.0cm]{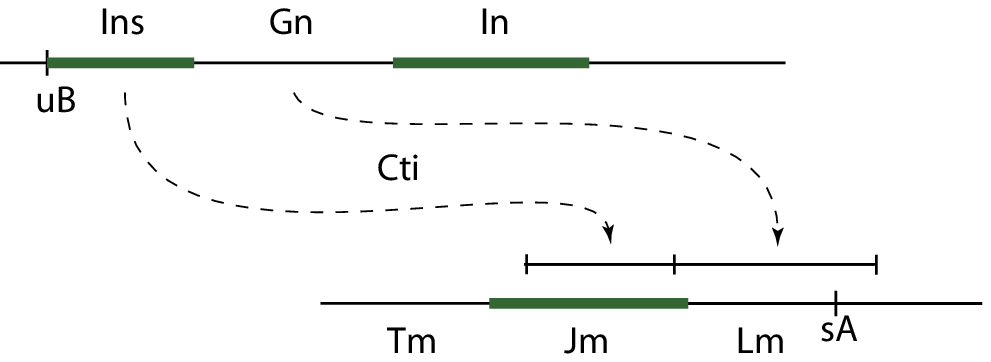}
\caption{{\small Proof of \eqref{e.fundamental}.}}
\label{f.intervals}
\end{center}
\end{figure}

Next we claim that for every sufficiently large $i$, the $3$-tuple $(A,B,C(t_i))$
is uniformly hyperbolic.
For simplicity of writing, let $i$ be fixed and
let $n=n_i$, $m=m_i$, $C=C(t_i)$.
Let $V = V_i$ be the interval $[C(t_i) \cdot B^{n} p, A q]$.
The set (see Figure~\ref{f.bifurcation})
\begin{equation}\label{e.complicated}
U_i = I_{n_i}^* \cup I_{n_i-1} \cup \cdots \cup I_0 \cup J_0 \cup \cdots \cup J_{m_i} \cup V_i \, .
\end{equation}
is mapped inside itself by each of the maps $A$, $B$, and $C$.
Indeed, the intervals are mapped into themselves as follows:
\begin{center}
\begin{tabular}[b]{|r|llllllllll|}
\hline
         & $I_n^*$ & $I_{n-1}$ & $I_{n-2}$ & \dots & $I_0$ & $J_0$ & $J_1$ & \dots & $J_m$    & $V$ \\
\hline
$A$      & $V$     & $V$       & $V$       & \dots & $V$   & $V$   & $J_0$ & \dots & $J_{m-1}$ & $V$ \\
$B$      & $I_n^*$ &$I_n^*$    & $I_{n-1}$ & \dots & $I_1$ & $I_0$ & $I_0$ & \dots & $I_0$     & $I_0$ \\
$C$      & $J_m$   & $V$       & $V$       & \dots & $V$   & $V$   & $V$   & \dots & $V$       & $V$ \\
\hline
\end{tabular}
\end{center}

We want to apply Lemma~\ref{l.core full} with $U=U_i$ given by \eqref{e.complicated};
thus we need to define also a set $S = S_i$.
We will make use of the symmetry of the example.
Define a new family of triples
$$
(\tilde A, \tilde B, \tilde C(t)) =
(B^{-1}, A^{-1}, C(t)^{-1}) \, ,
$$
We claim that the new triples meets all the requirements
we imposed on $(A,B,C(t))$,
if we consider on $\P^1$ the reverse cyclical order.
Indeed, let $\tilde p = p$ and $\tilde q = BAq$.
Define new cross-ratios $\tilde \alpha$, $\tilde \beta$ as
in \eqref{e.cross ratios} (but with reversed order);
then $\tilde \alpha = \beta$ and $\tilde \beta = \alpha$,
so the new \eqref{e.assumption} still holds.
Other conditions as \eqref{e.irrational} and \eqref{e.C} are easily checked.
Consider the new families of intervals
$\tilde I_m$, $\tilde J_n$, $\tilde I_m^*$ (it is convenient to swap the letters in the indices);
then $\tilde I_m$ is the gap between $J_m$ and $J_{m+1}$ and
$\tilde J_n$ is the gap between $I_n$ and $I_{n+1}$.
(In particular the notation~\eqref{e.dual intervals} is coherent.)
The relevant condition on $m_i$, $n_i$, $t_i$ is \eqref{e.fundamental}.
Its dual version is:
\begin{equation}\label{e.fundamental dual}
C(t_i)^{-1} \cdot \tilde I_{m_i}^*   \Subset \tilde J_{m_i} \quad \text{and} \quad
C(t_i)^{-1} \cdot \tilde I_{m_i - 1} \Subset (s_C, u_B) \, .
\end{equation}
An inspection of Figure~\ref{f.intervals} shows that it is true.
Let $\tilde V_i = [Aq, C(t_i)^{-1}A^{-m}p]$.
Then the set
$S_i = \tilde I_{m_i}^* \cup \tilde I_{m_i-1} \cup \cdots \cup \tilde I_0
\cup \tilde J_0 \cup \cdots \cup \tilde J_{m_i} \cup \tilde V_i$
is sent inside itself for $A^{-1}$, $B^{-1}$, and $C(t_i)^{-1}$.

This still not good if we want to apply Lemma~\ref{l.core full}
because $S_i$ is not disjoint from $U_i$.
To remedy that, it suffices for each $i$ to make $\tilde J_0$ slightly smaller
(making sure \eqref{e.fundamental dual} is still satisfied)
and modify the definition of $S_i$ accordingly.
In this way we can apply the lemma and conclude that $(A,B,C(t_i))$ is hyperbolic.

\medskip

Next, we claim that:
\begin{equation}\label{e.wind claim}
k, \ \ell \ge 0 \quad \Rightarrow \quad \tr C(t_i) B^\ell A^k
\begin{cases}
<-2 &\text{if $k \ge m_i + 1$ and $\ell \ge n_i+1$,} \\
>2  &\text{otherwise.} \\
\end{cases}
\end{equation}
Although the proof is not difficult, we prefer to postpone it to
\S\ref{ss.wind}.
Recall from \eqref{e.kappa} that the sequences $(n_i)$ and $(m_i)$ are strictly increasing.
Then it follows from~\eqref{e.wind claim} that
$(A,B, C(t_i))$ and $(A,B, C(t_j))$ do not belong to the same connected component of $\cH$
if $i \neq j$.
This proves assertion~(iii) of the proposition.

\medskip

At last, by \eqref{e.wind claim} again, %and the Intermediate Value Theorem,
for every $i$ there exists $s_i$ between $t_i$ and $t_{i+1}$ such that
$\tr C(s_i) B^{n_i+1} A^{m_i+1} =0$,
so $(A,B,C(s_i))$ belongs to the elliptic locus.
This proves the last assertion of the proposition.
\end{proof}

\begin{rem}
With a little additional work, one can find
the unstable and stable cores for $(A,B,C(t_i))$;
they are given by the subintervals below
(again we write $n=n_i$, $m=m_i$, $C=C(t_i)$ for simplicity):
\begin{alignat*}{4}
&[u_B,&\,u(B^{n+1}A^{m+1}C)] &\subset I_n^*         &\qquad &[s(B^{n+1}A^{m+1}C),&\,C^{-1}s_A] &\subset \tilde J_n \\
&[B^nA^{m}Cu_B,&\,u(B^n A^{m+1}CB)]&\subset I_{n-1} &\qquad &[s(B^n A^{m+1}CB),&\,B^{-1}C^{-1}s_A] &\subset\tilde J_{n-1} \\
&&&\cdots & &&&\dots  \\
&[BA^{m}Cu_B,&\,u(BA^{m+1}CB^n)] &\subset I_0       &\qquad &[s(BA^{m+1}CB^n),&\,B^{-n}C^{-1}s_A] &\subset\tilde J_0 \\
&[A^{m}Cu_B,&\,u(A^m CB^{n+1}A)] &\subset J_0       &\qquad &[s(A^m CB^{n+1}A),&\,A^{-1}B^{-n}C^{-1}s_A] &\subset\tilde I_0 \\
&&&\cdots & &&&\dots  \\
&[ACu_B,&\,u(ACB^{n+1}A^m)] &\subset J_{m-1}        &\qquad &[s(ACB^{n+1}A^m),&\,A^{-m}B^{-n}C^{-1}s_A] &\subset \tilde I_{m-1} \\
&[Cu_B,&\,u(CB^{n+1}A^{m+1})] &\subset J_m          &\qquad &[s(CB^{n+1}A^{m+1}),&\,s_A] &\subset \tilde I_m^* \\
&[CB^nA^mCu_B,&\,u(A^{m+1}CB^{n+1})] &\subset V     &\qquad &[s(A^{m+1}CB^{n+1}),&\,C^{-1}A^{-m}B^{-n}C^{-1}s_A] &\subset \tilde V
\end{alignat*}
In particular, the rank of the cores for the component $H_i$ is $m_i+n_i+3$;
so we get another proof that $H_i \neq H_j$ if $i \neq j$.
\end{rem}

%Outras coisas:
%
%* Obs: a proposicao 7 do \cite{Yoccoz_SL2R} diz que so tem elipticos
%do outro lado do bordo da principal.
%Sabemos generalizar isso?
%
%* Segunda bifurcacao?? (Acho que nao vale a pena)

%% file: aby_combin.tex
%%%%%%%%%%%%%%%%%%%%%%%%%%%%%%%%%%%%%%%%%%%%%%%%%%%%%%%%%%%
\section{Combinatorial Multicone Dynamics}\label{s.abstract combinatorics}
%%%%%%%%%%%%%%%%%%%%%%%%%%%%%%%%%%%%%%%%%%%%%%%%%%%%%%%%%%%

\subsection{The Setting}

\subsubsection{}
A \emph{pair of combinatorial multicones}
is a finite cyclically ordered set $\M$
which is partitioned into $2$ disjoint subsets $\M_s$, $\M_u$
of the same cardinality which are met
alternately according to the cyclic ordering.
The subset $\M_s$ is the \emph{stable combinatorial multicone},
the subset $\M_u$ is the \emph{unstable combinatorial multicone}
in the pair.
The integer $q = \# \M_s = \# \M_u = \frac{1}{2} \#\M$
is the \emph{rank} of $\M$.

\subsubsection{}
A \emph{correspondence} on $\M$ is a subset of $\M \times \M$.

Given two correspondences $C$, $C'$ on $\M$, their product
$C \circ C'$ is defined by
$$
C \circ C' = \{(x,z); \; \text{$\exists$  $y\in \M$ s.t.~$(x,y) \in C$, $(y,z)\in C'$}\}.
$$
This composition law is obviously \emph{associative};
the diagonal in $\M \times \M$ is an identity (both left and right).
Thus correspondences form a monoid.

\subsubsection{}
Let $C$ be a correspondence on $\M$.
We say that $C$ is \emph{monotonic}
if the following properties hold:
\begin{itemize}
\item $C \subset (\M_s \times \M_s) \sqcup (\M_u \times \M_u)$;

\item $C \cap (\M_s \times \M_s)$ is the graph
$\{(C_s(x_s),x_s); \; x_s \in \M_s\}$ of a map ${C_s : \M_s \to \M_s}$;

\item $C \cap (\M_u \times \M_u)$ is the graph
$\{(x_u, C_u(x_u)); \; x_u \in \M_u\}$ of a map ${C_u : \M_u \to \M_u}$;

\item $C$ can be endowed with a cyclic ordering such that
the element next to $(x,y)$ is either
$(x^{++},y)$ or $(x^+,y^+)$ or $(x,y^{++})$,
where $x^+$ (resp.~$y^+$, $x^{++}$, $y^{++}$)
denotes the element next to $x$ (resp.~to $y$, $x^+$, $y^+$).
\end{itemize}

Observe that the cyclic ordering on $C$ is uniquely defined by the latter property:
if for instance $(x,y) \in \M_u \times \M_u$, then either
$x^+$ belongs to the image of $C_s$ and the next element is $(x^+,y^+)$,
or it is not the case and the next element is $(x^{++},y)$.
Similarly, if $(x,y) \in \M_s \times \M_s$ then the next element
is $(x^+, y^+)$ if $y^+ \in \Im C_u$,
and $(x,y^{++})$ otherwise.

The last condition (existence of the cyclic ordering)
in the definition of monotonicity may be reformulated as follows:
\begin{itemize}
\item for $(x_u,y_u) \in C \cap (\M_u \times \M_u)$
we must have  $x_u^+ = C_s(y_u^+)$ if $x_u^+ \in \Im C_s$
and $y_u = C_u(x_u^{++})$ if $x_u^+ \not\in \Im C_s$;

\item for $(x_s,y_s) \in C \cap (\M_s \times \M_s)$
we must have $y_s^+ = C_u(x_s^+)$ if $y_s^+ \in \Im C_u$
and $x_s = C_s (y_s^{++})$ if $y_s^+ \not\in \Im C_u$.
\end{itemize}

Obviously, a monotonic correspondence must satisfy
\begin{gather*}
\# C = \# \M = 2 \; \mathrm{rk}(\M), \\
1 \le \# \Im C_s = \# \Im C_u \le \mathrm{rk}(\M).
\end{gather*}

\subsubsection{Examples}
\begin{itemize}
\item The diagonal (or identity) correspondence is monotonic.

\item Let $a_s \in \M_s$, $a_u \in \M_u$; set
$$
C_{a_s,a_u} = \M_u \times \{a_u\} \, \sqcup \,  \{a_s\} \times \M_s
$$
(i.e.~$C_s$, $C_u$ are the constant maps with values $a_s$,
$a_u$ respectively.)
This correspondence is monotonic and is called a \emph{constant}
correspondence (with values $a_s$, $a_u$).
The left or right composition of a
monotonic correspondence with any constant correspondence
is a constant correspondence.

\item See Figures~\ref{f.correspondences free} and \ref{f.correspondences 2/5} for more examples.
\end{itemize}

\begin{figure}[hbt]
\psfrag{A}{$A$}
\psfrag{B}{$B$}
\begin{center}
\includegraphics[keepaspectratio, scale=0.7]{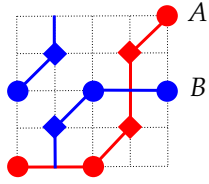}
\caption{{\small Two (constant) monotonic correspondences $A$ and $B$
(related to a free uniformly hyperbolic pair).
The rank of $\M$ is $2$.
The borders of the square should be identified in a torus-like way.
Circles and squares denote points in $\M_u \times \M_u$ and $\M_s \times \M_s$, respectively.}}
\label{f.correspondences free}
\end{center}
\end{figure}

\begin{figure}[hbt]
\psfrag{A}{$A$}
\psfrag{B}{$B$}
\begin{center}
\includegraphics[keepaspectratio, scale=0.7]{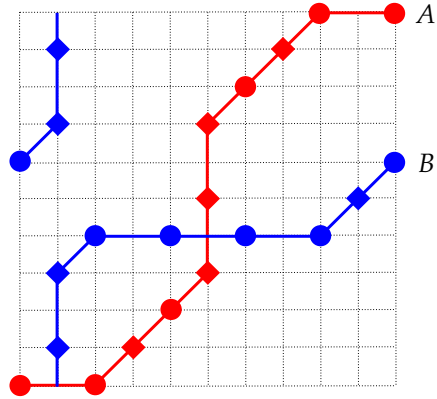}
\caption{{\small Two monotonic correspondences
$A$ and $B$ (related to the situation of Fig.~\ref{f.multicone 2/5}).
The rank of $\M$ is $5$.}}
\label{f.correspondences 2/5}
\end{center}
\end{figure}

\subsubsection{Elementary Properties}

\paragraph{}
The composition $C \circ C'$ of monotonic
correspondences is monotonic.

\begin{proof}
Let $C_s$, $C_u$, $C_s'$, $C_u'$ be the maps associated with $C$, $C'$.
From the definition of the composition law, we see that
$C \circ C' \subset (\M_s \times \M_s) \cup (\M_u \times \M_u)$ with
\begin{align*}
(C \circ C') \cap (\M_s \times \M_s) &= \{(C_s \circ C_s'(x_s), x_s); \; x_s \in \M_s\}, \\
(C \circ C') \cap (\M_u \times \M_u) &= \{(x_u, C_u' \circ C_u(x_u)); \; x_u \in \M_u\}.
\end{align*}
Let $(x_u, z_u) \in (C \circ C') \cap (\M_u \times \M_u)$;
set $y_u = C_u(x_u)$, so we have $z_u = C_u'(y_u)$.
\begin{itemize}
\item If $x_u^+ \not\in \Im C_s$, then also $x_u^+ \not\in \Im C_s \circ C_s'$
and we have $y_u = C_u(x_u^{++})$, $z_u = C_u' \circ C_u(x_u^{++})$.

\item Assume $x_u^+ \in \Im C_s$; then $x_u^+ = C_s(y)$
if and only if $y \in \M_s$ is between $y_u = C_u(x_u)$
and $C_u(x_u^{++})$.
If no such $y$ belongs to $\Im C_s'$, we must have
$$
C_u'(C_u(x_u^{++})) = C_u'(C_u(x_u)).
$$
Otherwise, let $y_s$ be the first $y$ in $\Im C_s'$
between $y_u$ and $C_u(x_u^{++})$;
we have
$$
C_u'(y_s^-) = C_u'(C_u(x_u)) = z_u, \qquad
y_s = C_s'(z_u^+), \qquad
x_u^+ = C_s(C_s'(z_u^+)) .
$$
\end{itemize}
We have checked the first half of the condition for the existence
of the cyclic ordering on $C \circ C'$;
the other half is checked in a symmetric way.
\end{proof}

\paragraph{} \label{ssss.Cs determines C}
We have seen that
$$
\# \Im C_s = \# \Im C_u \, .
$$
In particular, $C_s$ is a constant map iff $C_u$ is a constant map;
in this case, the values of $C_s$ and $C_u$ are independent.

However, when $C_s$ is \emph{not} a constant map,
there is \emph{at most} one monotonic correspondence $C$ such that
$C \cap (\M_s \times \M_s)$ is the graph of $C_s$.
(And similarly when we exchange the roles of $C_s$ and $C_u$.)
More precisely, such a monotonic correspondence exists
if and only if
the map $C_s$ is \emph{monotonic} (increasing)
in the following sense:
For any $x_s \in \M_s$, either $C_s(x_s^{++}) = C_s(x_s)$
or there is no point of the image of $C_s$ strictly between $C_s(x_s)$ and $C_s(x_s^{++})$;
we then have $x_s^+ = C_u(x_u)$ for $x_u \in \M_u$ between $C_s(x_s)$ and $C_s(x_s^{++})$.

\subsection{Free Monoids of Monotonic Correspondences}

\subsubsection{}
We have seen that the monotonic correspondences on a pair of combinatorial multicones
$\M = \M_s \sqcup \M_u$ form a monoid
that we denote by~$\cC(\M)$.

Let $N \ge 1$ and let $\cF_N$ be the free monoid on $N$ generators.
Let $\Phi: \cF_N \to \cC(\M)$ be a morphism, uniquely determined
by the images $C^{(1)}$, \ldots, $C^{(N)}$ of the canonical generators
of $\cF_N$.

\subsubsection{}
%\begin{defn} \margem{as outras inumeras definicoes do paper nao sao numeradas...}
The morphism is called \emph{hyperbolic} if there exists
$\ell \ge 1$ such that the image of any word of length $\ge \ell$ in the generators
is a constant correspondence.
%\end{defn}

\subsubsection{}
The morphism is called \emph{tight} if we have
$$
\bigcup_{i=1}^N \Im C_u^{(i)} = \M_u  \quad \text{and}\quad
\bigcup_{i=1}^N \Im C_s^{(i)} = \M_s \, .
$$

A justification for this definition and terminology is the following:
assume for instance that some
$x_u' \in \M_u$ does not belong to any $\Im C_u^{(i)}$, $1 \le i \le N$;
then we have
$$
C_s^{(i)} ((x_u')^-) = C_s^{(i)} ((x_u')^+) \quad
\text{for all $1 \le i \le N$.}
$$
Consider the pair of combinatorial multicones
$\M' = \M_s' \sqcup \M'_u$
where $\M_u' = \M_u \setminus \{x_u'\}$ and
$\M_s'$ is deduced from $\M_s$ by identifying
$(x_u')^-$ with $(x_u')^+$;
$\M'$ is equipped with the obvious cyclic ordering.
One can define in an obvious way correspondences
$C^{(i)\prime}$, $1 \le i \le N$ on $\M'$,
and the study of the morphism
$\Phi: \cF_N \to \cC(\M)$ reduces to a morphism
$\Phi': \cF_N \to \cC(\M')$ with a smaller pair of combinatorial multicones.

\subsubsection{}
We would like to analyze tight hyperbolic morphisms.

For $N=1$, a morphism is tight iff the correspondence $C^{(1)}$ is invertible,
and then it cannot be hyperbolic except in the trivial case where
the rank is~$1$.

In \S\ref{ss.morphisms two}, we will
determine all tight hyperbolic morphisms when $N=2$.

%%%%%%%%%%%%%%%%%%%%%%%%%%%%%%%%%%%%%%%%%%%%%%%%%%%%%%%%
\subsection{Relation with Matrices} \label{ss.relation}

Let us see how a uniformly hyperbolic $N$-tuple of matrices
induces a tight hyperbolic morphism.

Let $(A_1, \ldots, A_N) \in \G^N$  be uniformly hyperbolic.
Let $U$ and $S$
be respectively the unstable and stable cores.
Let $\M_u$, resp.~$\M_s$, be the set of connected components of $U$, resp.~$S$.
Give $\M = \M_u \sqcup \M_s$ the cyclic order induced from $\P^1$.
Then $\M$ is a pair of combinatorial multicones.

For each $i =1, \ldots, N$, let
$C^{(i)}$ be the subset of $(\M_u \times \M_u) \sqcup (\M_s \times \M_s)$
formed by the pairs $(x,y)$ such that $A_i(x) \cap y \neq \emptyset$.

\begin{lemma}
Each $C^{(i)}$ is a monotonic correspondence.
Moreover, the morphism $\Phi: \cF_N \to \cC(\M)$ determined by $C^{(1)}$, \ldots, $C^{(N)}$
is tight and hyperbolic.
\end{lemma}

(The same $\Phi$ could also be obtained from a tight multicone $M$
and its dual $\P^1 \setminus \overline{M}$ in an obvious way, see Proposition~\ref{p.tight}.)

\begin{proof}[Proof of the lemma]
Fix $i$, and let us show that $C^{(i)}$ is monotonic.
First, if $(x_u,y_u) \in C^{(i)} \cap (\M_u \times \M_u)$ then
$A_i(x_u) \subset y_u$, so $x_u$ uniquely determines $y_u$.
Write $y_u = C^{(i)}_u(x_u)$.
Analogously, if $(x,y) \in C^{(i)} \cap (\M_s \times \M_s)$ then
$A_i^{-1}(y_s) \subset x_s$, so $y_s$ determines $x_s = C^{(i)}_s(y_s)$.

Next, let $(x_u,y_u) \in C^{(i)} \cap (\M_u \times \M_u)$.
In the case that $x_u^+ \not\in \Im C_s^{(i)}$
then we must have $C_u^{(i)}(x_u^{++}) = y_u$.
(Because if $C_u^{(i)}(x_u^{++}) \neq C_u^{(i)}(x_u)$ then
there would exist a point in the unstable core $S$ between the intervals
$C_u^{(i)}(x_u)$ and $C_u^{(i)}(x_u^{++})$; this point would be sent
by $A_i^{-1}$ into a point in $S$ between $x_u$ and $x_u^{++}$, and hence in $x_u^+$,
contradicting the fact that $x_u^+ \not\in \Im C_s^{(i)}$.)
And in the case that $x_u^+ \in \Im C_s^{(i)}$
then we must have $C_s^{(i)}(y_u^+) = x_u^+$.
(Indeed, $x_u^+$ is the $C_s^{(i)}$ image of some $z_s$;
if $z_s = y_u^+$ we are done;
otherwise $y_u^+$ is between the $U$-interval $y_u$ and $S$-interval $z_s$;
then the interval $A_i^{-1}(y_u^+)$ is
between $A_i^{-1}(y_u) \supset x_u$, and $A_i^{-1}(z_s) \subset x_u^+$,
and so it must be contained in the interval $x_u^+$,
showing that $x_u^+ = C_s^{(i)}(y_u^+)$.)
This proves ``one half'' of the monotonicity of $C^{(i)}$,
and the other half is completely analogous.

The induced morphism $\Phi$ is clearly tight,
while hyperbolicity follows from Proposition~\ref{p.combin contraction}.
\end{proof}

In view of the lemma, we call $\Phi$ the morphism induced by $(A_1, \ldots, A_N)$.
Examples from Figures~\ref{f.correspondences free} and \ref{f.correspondences 2/5}
are induced by matrices.

Sometimes we call these data (ie, the morphism $\Phi$)
the \emph{combinatorics} of $(A_1, \ldots, A_N)$.
The combinatorics is an \emph{invariant} in the sense that in remains the same
inside each connected component of $\cH$.
(More precisely, if two $N$-tuples belong to the same connected component
then they induce conjugate morphisms.)

\medskip

Let us very briefly return to the topic of the boundary of the hyperbolic components:

\begin{thm}\label{t.disjoint boundaries}
Non-principal components of $\cH$
with different combinatorics
have disjoint boundaries.
\end{thm}

\begin{proof}
For each $N$-tuple in the boundary of a non-principal component $H$,
the limit cores are defined (by Propositions~\ref{p.no id} and \ref{p.limit cores}).
These limit cores induce a tight hyperbolic morphism $\Phi$ in an obvious way.
In fact $\Phi$ is the same (ie, conjugate to the) morphism determined by the component $H$ itself.
Now, if the $N$-tuple belongs also to the boundary of another component $H_1$,
then the limit cores relative to $H_1$ are exactly the same as before,
by Proposition~\ref{p.independence}.
It follows that $H$ and $H_1$ have the same combinatorics.
\end{proof}

\subsection{Winding Numbers} \label{ss.wind}

\subsubsection{The Winding Numbers for a Uniformly Hyperbolic $N$-tuple}
%\margem{A ordem das coisas ficou meio esquisita... Os winding numbers poderiam ate mesmo ter sido %definidos na pagina 2. E a\'i nao precisaria adiar a prova de \eqref{e.wind claim}.}

As mentioned above, the combinatorics $\Phi: \cF_N \to \cC(\M)$ is an invariant on $\cH$.
A much more elementary invariant was introduced in \cite{Yoccoz_SL2R};
it is the map $\tau: \cF_N \to \{+1, -1\}$
that gives the signs of the traces.

Here we will introduce another elementary (in the sense that it does not depend on the multicones)
invariant called the \emph{winding number};
it is a map $n: \cF_N \to \Z$.

\medskip

Fix a cyclic order on $\P^1$, and
identify $\P^1$ with $\R/\Z$ via an orientation-preserving homeomorphism.
So any $A \in \G$ induces a orientation-preserving homeomorphism $A: \R/\Z \to \R/\Z$.
Then we can lift $A$ with respect to the covering map $\R \to \R/\Z$ and obtain
a homeomorphism $\hat{A}: \R \to \R$.

Now, let a uniformly hyperbolic $N$-tuple $(A_1,\ldots,A_N)$ be given.
Since each $A_i$ is hyperbolic, it has a
unique lift $\hat{A}_i$ whose graph intersects the diagonal of $\R^2$.
Given a word $\omega = A_{i_j} \cdots A_{i_1}$,
its winding number $n(\omega)$ is defined as the only integer $n$ such that
$$
\hat{A}_{i_j} \circ \cdots \circ \hat{A}_{i_1} (x_0) = x_0 + n \quad
\text{for some $x_0 \in \R$.}
$$

It is clear that the winding number map $n: \cF_N \to \Z$ is an invariant, ie,
it depends only on the connected component of $\cH$ the hyperbolic $N$-tuple is in.

\medskip

Let us see that the trace signs $\tau$ essentially depend only on $n$.
More precisely,
if $\tr A_1$, \ldots, $\tr A_j$ are all positive,
then the sign of $\tr A_{i_j} \cdots A_{i_1}$ is $(-1)^n$,
where $n$ is the winding number of the word.
To see this fact,
first notice that if we substitute the covering map $\R \to \R/\Z = \P^1$
with the double covering $S^1 \to \P^1$
along the definition of the winding number,
then we obtain the invariant $n \bmod 2$.
And the relation between that invariant and signs of eigenvalues is transparent.

\medskip

To give an example, let us compute the winding numbers for the positive free
component of $\G^2$.
Consider a word $\omega$ in the letters $A$ and $B$ that contains both
(otherwise the winding number is zero).
Notice that the winding number of a word is left invariant by cyclic permutations.
(That is a general fact.)
So we can assume the word is of the form $\omega = A^{k_1} B^{\ell_1} A^{k_2} B^{\ell_2} \cdots A^{k_m} B^{\ell_m}$,
with all $k_i$, $\ell_i$ positive.
Then the winding number of $\omega$ is~$- m$.
%\margem{That component should have been called \emph{negative}! Do we want to change?}
(The winding numbers are opposite for the free component obtained from the positive by conjugation with an orientation-reversing linear map.)

\medskip

Let us pause our general discussion to give the:

\begin{proof}[Completion of the proof of Proposition~\ref{p.bifurcation}]
We need to prove \eqref{e.wind claim}.
Let $k$, $\ell \ge 0$
and consider the matrix $C(t_i) B^\ell A^k$.
Notice that its expanding direction is in $V$ if $\ell \le n_i$ and in $J_{m_i}$ otherwise.
Looking at the action of the lifts on that fixed point, we see
that if $k \ge m_i+1$ and $\ell \ge n_i+1$ then the winding number of is $-1$,
otherwise it is zero.
\end{proof}

\subsubsection{Combinatorial Definition of Winding Numbers}

Fix a pair of combinatorial multicones $\M$, and let $q$ be its rank.
Identify $\M$ with $\nicefrac{\Z}{2q\Z}$ via some bijection
that preserves the cyclic orders;
such identification will remain fixed in the sequel.
%For definitiveness, assume that $0 \in \M_u$.
Let $x \in \Z \mapsto \bar{x} \in \nicefrac{\Z}{2q\Z}$ be the quotient map.

A subset $\hat{C}$ of $\Z^2$ is called a \emph{lifted correspondence}
if there exists  a monotonic correspondence $C$ on $\M$ such that
the following properties hold:
\begin{itemize}
\item if $(x,y) \in \hat{C}$ then $(\bar x, \bar y) \in C$;
\item there is a bijection between $\Z$ and $\hat{C}$
such that if we endow $\hat{C}$ with the order induced from $\Z$ then
the element next to $(x,y)$
is $(x+2,y)$, or $(x+1,y+1)$, or $(x, y+2)$, according to whether the element in
$C$ next to $(\bar x, \bar y)$ is
$(\bar x^{++}, \bar y)$, or $(\bar x ^+, \bar y^+)$, or $(\bar x, \bar y^{++})$.
\end{itemize}
We also say $\hat{C}$ is a \emph{lift} of $C$.
Notice $\hat{C}$ is invariant by the translation of $\Z^2$ by $(2q,2q)$,
in other words, $\hat{C} = \hat{C} +(2q,2q)$.
Also notice that if $\hat{C}$, $\hat{C}_1$ are two lifts of the same monotonic correspondence $C$
then there is an unique $n\in\Z$ such that $\hat{C}_1 = \hat{C} + (2qn,0)$.

Composition of lifted correspondences is defined
in a similar manner as for monotonic correspondences.
Associativity holds (the proof is similar).
Also, the composition of lifts is a lift of the composition of two monotonic correspondences.

If a monotonic correspondence $C$ is hyperbolic (in the sense that some power of it is a constant)
then for every lift $\hat{C}$ of $C$ there is a unique $n \in \Z$
such that $\hat{C}$ contains a point of the form $(x, x + 2qn)$;
such number $n$ is called the \emph{height} of $\hat{C}$.

Now let $\Phi: \cF_N \to \cC(\M)$ be a tight hyperbolic morphism.
Let $a_1, \ldots, a_N$ be the canonical generators of $\cF_N$,
and let the correspondences $C^{(1)}, \ldots, C^{(N)}$ be their respective images by $\Phi$.
Let $\hat{C}^{(i)}$ be the unique lift of $C^{(i)}$ of height zero.

The \emph{winding number} $n(\omega)$ of a word $\omega = a_{i_j} \cdots a_{i_1}$ in $\cF_N$
is the height of the lifted correspondence
$\hat{C}^{(i_j)} \circ \cdots \circ \hat{C}^{(i_1)}$.
The winding number of the empty word is defined as zero.

(Notice that winding numbers do not depend on the identification between
$\M$ and $\nicefrac{\Z}{2q\Z}$.)

\smallskip

It is easy to see that if the morphism $\Phi: \cF_N \to \cC(\M)$ is
induced by a hyperbolic $N$-tuple, then
our two definitions of winding numbers give the same results. %\margem{precisa explicar?}

\subsubsection{A Non-Vanishing Property}\label{sss.principal and traces}

\begin{lemma}\label{l.one wind}
If the rank of $\M$ is bigger than $1$ then
there is a word $\omega$ such that $n(\omega) = \pm 1$.
\end{lemma}

\begin{proof}
It follows immediately from the definition of the winding number that,
for any word $\omega$ and any letter $a_i$, one has
$$
|n(\omega a_i) - n(\omega)| \le 1 \, , \qquad
|n(a_i \omega) - n(\omega)| \le 1 \, .
$$
On the other hand, let $e_1$, $e_2$, $e_3$, $e_4$ be elements of $\M$ such that
$$
e_1, \  e_3 \in \M_s \, , \quad
e_2, \  e_4 \in \M_u \, , \quad
e_1 < e_2 < e_3 < e_4 < e_1
$$
As $\Phi$ is hyperbolic and tight,
there exist words $\omega_{12}$ and $\omega_{34}$ such that
the image of $\omega_{12}$ is the constant correspondence $C_{e_1 e_2}$
and the image of $\omega_{34}$ is $C_{e_3 e_4}$.
We claim that
\begin{equation}\label{e.espertinhos}
n( \omega_{12} \omega_{34} ) =
n( \omega_{12} ) + n( \omega _{34} ) - 1 \, .
\end{equation}
Indeed, let $\hat{C}_{e_1 e_2}$ and $\hat{C}_{e_3 e_4}$ be the lifts of $C_{e_1 e_2}$ and $C_{e_3 e_4}$
whose heights are $n(\omega_{12})$, $n(\omega_{34})$, respectively.
Take integers $k_1 < k_2 < k_3 < k_4$ such that $\overline{k_i} = e_i$ and $k_4 - k_1 < 2q$.
Then
$$
\big(k_2, k_4 + 2q (n(\omega_{34}) -1) \big) \in \hat{C}_{e_3 e_4}
\quad \text{and} \quad
\big(k_4, k_2 + 2q n(\omega_{12}) \big) \in \hat{C}_{e_1 e_2} \, .
$$
Therefore
$$
\big(k_2, k_2 + 2q (n(\omega_{12}) + n(\omega_{34}) -1 )\big) \in \hat{C}_{e_1 e_2} \circ \hat{C}_{e_3 e_4} \, ,
$$
proving~\eqref{e.espertinhos}.
The lemma now follows at once.
\end{proof}

Lemma~\ref{l.one wind} has the following consequence:
If we restrict ourselves to $N$-tuples $(A_1, \ldots, A_N)$ with
$\tr A_1$, \ldots, $\tr A_N$ all positive,
then there is a unique component of $\cH$ where all products of $A_i$'s have positive trace,
namely the principal component.
This answers positively Question~1' of \cite{Yoccoz_SL2R}.

%%%%%%%%%%%%%%%%%%%%%%%%%%%%%%%%%%%%%%%%%%%%%%%%%%%%%%%%%%%%%%%%%%%%%%%%%%
\subsection{Tight Hyperbolic Morphisms for $N=2$}\label{ss.morphisms two}

The aim of this section is to prove the following result:

\begin{prop}\label{p.two is realizable}
Every tight hyperbolic morphism $\Phi: \cF_2 \to \cC(\M)$ is
induced by some uniformly hyperbolic pair of matrices.
\end{prop}

%\begin{proof}

\subsubsection{}
When the rank of $\M$ is $1$, there is only one monotonic correspondence
on $\M$, namely the identity (ie, the diagonal in $\M \times \M$).
Therefore, for any $N \ge 1$, there is exactly one morphism
$\Phi: \cF_N \to \cC(\M)$.
It is tight and hyperbolic.

From now on, we assume that the rank $q$ of $\M$ is at least $2$.

\subsubsection{}
Fix some tight hyperbolic morphism  $\Phi: \cF_2 \to \cC(\M)$
and write $A$, $B$ instead of $C^{(1)}$, $C^{(2)}$ for the images of the generators of $\cF_2$.

\begin{lemma}\label{l.4 points}
There exist two distinct points $x_s^{(0)}$, $x_s^{(1)}$ in $\M_s$ such that
$$
A_s(x_s^{(0)}) = A_s (x_s^{(1)}), \qquad
B_s(x_s^{(0)}) = B_s (x_s^{(1)}).
$$
Similarly, there exist two distinct points $x_u^{(0)}$, $x_u^{(1)}$ in $\M_u$ such that
$$
A_u(x_u^{(0)}) = A_u (x_u^{(1)}), \qquad
B_u(x_u^{(0)}) = B_u (x_u^{(1)}).
$$
\end{lemma}

\begin{rem}
We will see later that $\{x_s^{(0)}, x_s^{(1)}\}$, $\{x_u^{(0)}, x_u^{(1)}\}$
are uniquely determined by these properties.
\end{rem}

\begin{proof}[Proof of the lemma]
We prove the first half of the lemma.
Take two distinct points $x_s$, $x_s'$ in $\M_s$.
If the conclusion of the lemma does not hold, one can
construct inductively arbitrarily long words $w$ such that
$$
[\Phi(w)]_s (x_s) \neq [\Phi(w)]_s (x_s'),
$$
which contradicts hyperbolicity.
\end{proof}

\subsubsection{}
Let $x_s^{(0)}$, $x_s^{(1)}$, $x_u^{(0)}$, $x_u^{(1)}$ be as in Lemma~\ref{l.4 points}.
Renaming if necessary $x_s^{(0)}$ and $x_s^{(1)}$,
we can assume that the image of $A_u$ contains a point
between  $x_s^{(0)}$ and $x_s^{(1)}$.

\begin{lemma}\label{l.images}
The image of $A_u$ is the set of points in $\M_u$ between $x_s^{(0)}$ and $x_s^{(1)}$.
The image of $B_u$ is the set of points in $\M_u$ between $x_s^{(1)}$ and $x_s^{(0)}$.
\end{lemma}

\begin{proof}
As $A_s(x_s^{(0)}) = A_s(x_s^{(1)})$, it follows from the definition of monotonicity
that there cannot be any point of the image of $A_u$ between  $x_s^{(1)}$ and $x_s^{(0)}$.
Therefore, as $\M_u = \Im A_u \sqcup \Im B_u$, every point in $\M_u$ between
$x_s^{(1)}$ and $x_s^{(0)}$ belongs to the image of $B_u$.
Exchanging $A_u$, $B_u$ we get all the conclusions of the lemma.
\end{proof}

In the same manner, after renaming if necessary  $x_u^{(0)}$, $x_u^{(1)}$, we see that
$\Im A_s$ is the set of points in $\M_s$ between $x_u^{(1)}$ and $x_u^{(0)}$,
while $\Im B_s$ is the set of points in $\M_s$ between $x_u^{(0)}$ and $x_u^{(1)}$.

It follows immediately from Lemma~\ref{l.images} that
$x_s^{(0)}$, $x_s^{(1)}$, $x_u^{(0)}$, $x_u^{(1)}$ are now uniquely defined.

\subsubsection{}
\begin{lemma}
We have
$A_u([x_u^{(0)}, x_u^{(1)}]) \subset [x_u^{(0)}, x_u^{(1)}]$ and similarly
$B_u([x_u^{(1)}, x_u^{(0)}]) \subset [x_u^{(1)}, x_u^{(0)}]$,
$A_s([x_s^{(1)}, x_s^{(0)}]) \subset [x_s^{(1)}, x_s^{(0)}]$,
$B_s([x_s^{(0)}, x_s^{(1)}]) \subset [x_s^{(0)}, x_s^{(1)}]$.
\end{lemma}

\begin{proof}
We prove the first statement.
As the image of $A^{n+1}_u$ is contained in the image of $A_u^n$,
we deduce from the hyperbolicity of $\Phi$ that there exists $x^* \in \M_u$ such that $A_u(x^*)=x^*$
and $\Im A_u^n = \{x^*\}$ for large $n$.
If one had $x^* \not\in [x_u^{(0)}, x_u^{(1)}]$ then one would have $A_u^{-1}(x^*) = \{x^*\}$,
which is not compatible with $\Im A_u^n = \{x^*\}$.
Therefore $x^* \in [x_u^{(0)}, x_u^{(1)}]$ and
$A_u([x_u^{(0)}, x_u^{(1)}]) = \{x^* \}$.
\end{proof}

\begin{figure}[hbt]
\psfrag{6}[r][r]{{\small $\Fix A_s$}}
\psfrag{8}[r][r]{{\small $x_s^{(0)}$}}
\psfrag{9}[r][r]{{\small $x_u^{(0)}$}}
\psfrag{1}[r][r]{{\small $\Fix A_u$}}
\psfrag{2}[l][l]{{\small $\Fix B_s$}}
\psfrag{3}[l][l]{{\small $x_u^{(1)}$}}
\psfrag{4}[l][l]{{\small $x_s^{(1)}$}}
\psfrag{5}[l][l]{{\small $\Fix B_u$}}
\begin{center}
\includegraphics[keepaspectratio, scale=0.7]{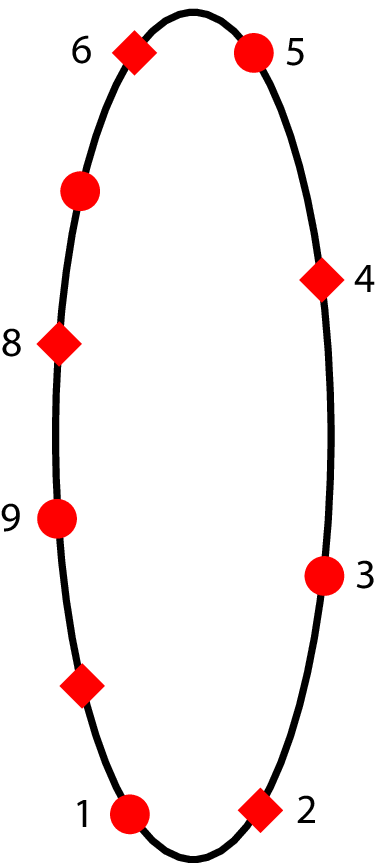}
\caption{{\small The pair of multicones $\M$ for $\nicefrac{p}{q} = \nicefrac{2}{5}$.}}
\label{f.ellipse}
\end{center}
\end{figure}

\subsubsection{}
Recall that we have denoted $q = \# \M_u = \# \M_s$ the rank of $\M$.
Let us denote
$$
p=   \# \Im B_u  = \# \Im B_s \, , \quad \text{hence} \quad
q-p= \# \Im A_u  = \# \Im A_s \, .
$$
If $q=2$ then $p=1$; both $A$ and $B$ are constant correspondences
and these are the dynamics associated to the free components.
We  will therefore assume that $q>2$.
By exchanging $A$ and $B$ we can assume that $p \le \nicefrac{q}{2}$.

\begin{lemma}
One has $p< \nicefrac{q}{2}$ and $x_u^{(0)}$, $x_1^{(1)} \in \Im A_u$.
\end{lemma}

\begin{proof}
$\Im A_u$ and $[x_u^{(1)}, x_u^{(0)}] \cap \M_u$ are intervals in $\M_u$
with respective cardinalities $q-p$ and $q-p+1$;
therefore at least  one of the two points $x_u^{(0)}$, $x_u^{(1)}$ belongs to $\Im A_u$,
and exactly one if $q = 2p$.
Assume that only one of the points $x_u^{(0)}$, $x_u^{(1)}$ belongs to $\Im A_u$.
Starting with $x_0$, $x_0' \in \M_u$ with $x_0 \neq x_0'$, $\{x_0, x_0'\} \neq \{x_u^{(0)}, x_u^{(1)}\}$,
we can construct sequences $(x_n)_{n \ge 0}$, $(x_n)_{n \ge 0}$ in $\M_u$ such that
$x_n \neq x_n'$ and $x_n = C_n x_{n-1}$, $x'_n = C_n x'_{n-1}$
for some $C_n \in \{A,B\}$: indeed one can never have
$\{x_{n-1}, x_{n-1}'\} \neq \{x_u^{(0)}, x_u^{(1)}\}$
as both $x_{n-1}$, $x_{n-1}'$ belong to $\Im C_{n-1}$.
Such sequences would contradict hyperbolicity.
Therefore the lemma is proved.
\end{proof}

\subsubsection{}
Let us summarize what we know so far about the correspondences $A$, $B$.
(See Figure~\ref{f.so far} for $\nicefrac{p}{q} = \nicefrac{2}{5}$.)

\begin{figure}[hbt]
\psfrag{A}[l][l]{{\footnotesize $A$}}
\psfrag{B}[l][l]{{\footnotesize $B$}}
\psfrag{1}[c][c]{{\footnotesize $\Fix A_u$}}
\psfrag{2}[c][c]{{\footnotesize $\Fix B_s$}}
\psfrag{3}[c][c]{{\footnotesize $x_u^{(1)}$}}
\psfrag{4}[c][c]{{\footnotesize $\Fix B_u$}}
\psfrag{5}[c][c]{{\footnotesize $\Fix A_s$}}
\psfrag{6}[c][c]{{\footnotesize $x_u^{(0)}$}}
\psfrag{a}[r][r]{{\footnotesize $\Fix B_s$}}
\psfrag{b}[r][r]{{\footnotesize $x_s^{(1)}$}}
\psfrag{c}[r][r]{{\footnotesize $\Fix B_u$}}
\psfrag{d}[r][r]{{\footnotesize $\Fix A_s$}}
\psfrag{e}[r][r]{{\footnotesize $x_s^{(0)}$}}
\psfrag{f}[r][r]{{\footnotesize $\Fix A_u$}}
\begin{center}
\includegraphics[keepaspectratio, scale=0.7]{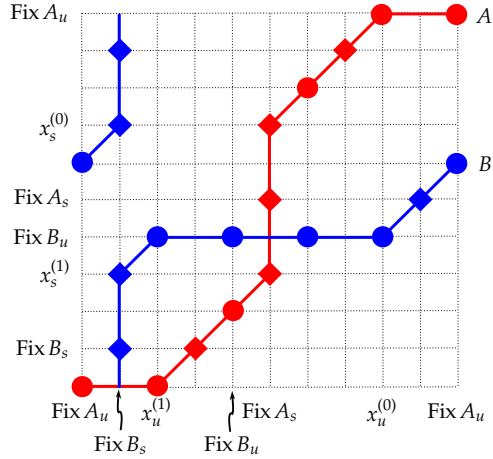}
\caption{{\small The correspondences $A$ and $B$ for $\nicefrac{p}{q} = \nicefrac{2}{5}$.}}
\label{f.so far}
\end{center}
\end{figure}

As a subset of $\M \times \M$, $A$ is made of:
\begin{itemize}
\item a horizontal segment from $(x_u^{(0)}, \Fix A_u)$ to $(x_u^{(1)}, \Fix A_u)$;
\item a vertical segment from $(\Fix A_s, x_s^{(1)})$ to $(\Fix A_s, x_s^{(0)})$;
\item two diagonal segments from $(x_u^{(1)}, \Fix A_u)$ to $(\Fix A_s, x_s^{(1)})$
and from $(\Fix A_s, x_s^{(0)})$ to $(x_u^{(0)}, \Fix A_u)$.
\end{itemize}
Here we have
\begin{gather*}
x_u^{(0)} < x_u^{(1)} < x_s^{(1)} < x_s^{(0)} < x_u^{(0)} \, , \\
x_u^{(0)} \le \Fix A_u \le x_u^{(1)} \, , \quad
x_s^{(1)} \le \Fix A_s \le x_s^{(0)} \, .
\end{gather*}
Similarly, $B$ is made of:
\begin{itemize}
\item a horizontal segment from $(x_u^{(1)}, \Fix B_u)$ to $(x_u^{(0)}, \Fix B_u)$;
\item a vertical segment from $(\Fix B_s, x_s^{(0)})$ to $(\Fix B_s, x_s^{(1)})$;
\item two diagonal segments from $(x_u^{(0)}, \Fix B_u)$ to $(\Fix B_s, x_s^{(0)})$
and from $(\Fix B_s, x_s^{(1)})$ to $(x_u^{(1)}, \Fix B_u)$.
\end{itemize}
We also have
$$
x_s^{(1)} < \Fix B_u < x_s^{(0)} \, , \quad
x_u^{(0)} < \Fix B_s < x_u^{(1)} \, .
$$

We would like to show that $q$ and $p$ are relatively prime
and that $(A,B)$ is obtained from the component described in Subsection~\ref{ss.full2 combin}
(or its mirror image).
This will be done by induction on $q$, the case $q=2$ having been checked already.

Changing the cyclic orientation if necessary, we may also assume that
$$
x_u^{(0)} \le \Fix A_u < \Fix B_s < x_u^{(1)} \, .
$$
Observe that the pair $(A,B)$ is completely determined by the following data (besides $p$, $q$):
\begin{itemize}
\item the number
$\bar p_0 := \# (\M_u \cap [x_u^{(0)}, \Fix A_u))$;
%= \# (\M_s \cap [x_s^{(1)}, \Fix A_s))
\item the number
$\bar q_0 := \bar p_0 + \# (\M_s \cap [x_s^{(0)}, \Fix B_s))$;
%= \bar p_0 +  \# (\M_s \cap [x_u^{(1)}, \Fix B_u))
\item the number
$\delta := \# (\Fix A_u, \Fix B_s)$.
\end{itemize}
Indeed these numbers determine the relative positions of
$x_u^{(0)}$, $x_u^{(1)}$, $x_s^{(0)}$, $x_s^{(1)}$,
$\Fix A_u$, $\Fix A_s$, $\Fix B_u$, $\Fix B_s$ on $\M$.
Setting $\bar p_1 = p - \bar p_0$, $\bar q_1 = q - \bar q_0$, we have
\begin{gather*}
\bar p_1 = \# (\M_u \cap [\Fix A_u, x_u^{(1)}))\,,\\
%= \# (\M_s \cap [\Fix A_s, x_s^{(0)})) \, , \\
\bar q_1 = \bar p_1 + \# (\M_s \cap [\Fix B_s, x_s^{(1)}))\,.
%= \bar p_1 + \# (\M_u \cap [\Fix B_u, x_u^{(0)})) \, .
\end{gather*}
For the component described in Subsection~\ref{ss.full2 combin},
one checks that $\bar p_0 = p_0$, $\bar q_0 = q_0$, $\bar p_1 = p_1$, $\bar q_1 = q_1$, $\delta=0$,
where $\nicefrac{p}{q}$ is the Farey center of the Farey interval $[\nicefrac{p_0}{q_0}, \nicefrac{p_1}{q_1}]$.
We have to prove these relations in our case.

\subsubsection{}
From $A$, $B$ we will construct a new par of combinatorial multicones
$\M' = \M'_s \sqcup \M'_u$ of rank $q' := q-p$,
and two monotone correspondences $A'$, $B'$ on $\M'$ which generate a tight hyperbolic morphism.
Applying the induction hypothesis will allow us to conclude.

We define $\M' := \M_s' \sqcup \M_u'$ where
$\M_u' := \Im A_u = (x_s^{(0)}, x_s^{(1)}) \cap \M_u$
and $\M_s'$ is obtained from $\M_s$
by collapsing the interval $[x_s^{(1)}, x_s^{(0)}] \cap \M_s$ into a point
denoted by $\bar x '$.
We write $\pi$ for the canonical map from $\M_s$ to $\M_s'$.
Observe that $A_s$ is constant on $[x_s^{(1)}, x_s^{(0)}] \cap \M_s$,
with value $\Fix A_s$.
Therefore the composition $A_s \circ \pi^{-1}$ is well defined
and is a bijection from $\M'_s$ to $\Im A_s$.
(This shows that the asymmetry of the definition of $\M'$ is only apparent.)

We equip $\M'$ with the obvious cyclic order inherited from $\M$.
We define:
\begin{alignat*}{2}
A_u' &= A_u|\M_u' \, ,          &\quad A_s' = \pi\circ A_s \circ \pi^{-1} \, , \\
B_u' &= A_u\circ B_u|\M_u' \, , &\quad B_s' = \pi\circ B_s \circ A_s \circ \pi^{-1} \, .
\end{alignat*}
One checks easily that this defines monotone correspondences $A'$, $B'$ on $\M'$.
Let $\Phi' : \cF_2 \to \cC(\M')$ be the morphism generated by $A'$, $B'$.

Let us check that $\Phi'$ is hyperbolic:
for any long enough word $w'$ in $A'$, $B'$, the unstable part
$w_u'$ is an even longer word in $A_u$, $B_u$;
as $\Phi$ is hyperbolic, the image is reduced to a point.
This proves that $w'$ is a constant correspondence.

Let us check that $\Phi'$ is tight.
Any $x_u' \in \M_u'$ can be written as $A_u(x_u)$ with $x_u \in \M_u$;
as $\Phi$ is tight, either $x_u \in \M_u'$ and $x_u' \in \Im A_u'$
or $x_u \in \Im B_u$;
as $B_u(\M'_u) = \Im B_u$, we have
$x_u' \in \Im B_u'$ in this case.
Similarly,
let $x_s' \in M_s'$;
if $x_s' \in \pi(\Im A_s)$ then $x'_s\in \Im A_s'$;
if  $x_s' \in \pi(\Im B_s)$ then, as $\Im B_s = \Im B_s A_s$, we have $x'_s\in \Im B_s'$.
Therefore $\Phi'$ is tight.

As $A_u$ is injective on $(x_s^{(1)}, x_s^{(0)}) \cap \M_u$
and the image of this set is disjoint from $A_u (\Im A_u)$, we have
$$
\# \Im A_u' = \# \Im A_u - p = q - 2p \, ,
$$
and therefore (as $\Im A_u' \cap \Im B_u' = \emptyset$)
$$
p' := \Im B_u' = p \, , \quad
\# \Im A_u' = q'-p' \, .
$$
We will apply the induction hypothesis to the tight hyperbolic morphism $\Phi'$
and therefore we have to identify the parameters
$\bar p_0'$, $\bar q_0'$, $\delta'$ for this morphism.

We have
\begin{align*}
A_u'(x_u^{(0)}) &= A_u'(x_u^{(1)}) = \Fix A_u \, , \\
B_u'(x_u^{(0)}) &= B_u'(x_u^{(1)}) \, ,
\end{align*}
therefore $\Fix A_u' = \Fix A_u$ and $\bar p_0' = \bar p_0$.
Let $x^{\prime (0)}_s$, $x^{\prime (1)}_s$ be the points in $\M_s$ such that
$A_s(x^{\prime (0)}_s) = x^{(0)}_s$, $A_s(x^{\prime (1)}_s) = x^{(1)}_s$
(if $\Fix A_s \neq x_s^{(i)}$ then $x^{\prime (i)}_s$ is uniquely determined
by this condition;
if $\Fix A_s = x_s^{(i)}$ then we take $x^{\prime (i)}_s = \Fix A_s$).
It is easy to see that $\pi(x^{\prime (0)}_s) \neq \pi(x^{\prime (1)}_s)$.
We have then
\begin{align*}
A_s'(\pi(x^{\prime (0)}_s)) &= A_s'(\pi(x^{\prime (1)}_s)) = \bar x' \, , \\
B_s'(\pi(x^{\prime (0)}_s)) &= B_s'(\pi(x^{\prime (1)}_s)) = \pi(\Fix B_s) \, .
\end{align*}
This shows that $\bar q_0' = \bar q_0 - \bar p_0$,
$\Fix B_s' = \pi(\Fix B_s)$ and therefore $\delta' = \delta$.
From the inductive hypothesis, we must have
$\delta'=0$, $\bar q_0' = q_0'$, $\bar p_0' = p_0'$,
where $\left[\nicefrac{p_0'}{q_0'}, \nicefrac{(p' - p_0')}{(q' - q_0')}\right]$ is the Farey interval with
center $\nicefrac{p'}{q'}$.
But then we have also
$\delta=0$, $\bar p_0 = p_0$, $\bar q_0 = q_0$.
This is the end of the proof of Proposition~\ref{p.two is realizable}.
%\end{proof}

%%%%%%%%%%%%%%%%%%%%%%%%%%%%%%%%%%%%%%%%%%%%%%
\subsection{Non-Realizable Multicone Dynamics}

Here we will show that Proposition~\ref{p.two is realizable}
does not extend to every $N$:

\begin{prop}\label{p.non realizable}
There exists a tight hyperbolic morphism $\Phi:\cF_N \to \cC(\M)$
which is not induced by any uniformly hyperbolic $N$-tuple.
\end{prop}

Recall our definition of cross-ratio \eqref{e.def cross ratio} from \S\ref{ss.if}.
It may be useful to bear in mind that
$1 < [a,b,c,d]< \infty$ if $a<b<c<d\mathrel{(\, <}a)$
(where $<$ is the cyclic ordering on $\R \cup \{\infty\}$)
%, and that $[b,c,d,a] = [a,b,c,d]/([a,b,c,d]-1)$.
The following lemma compares certain cross-ratios:

\begin{lemma}\label{l.cross}
Take eight distinct points in $\P^1$:
$$
a' < a < b < b' < c' < c < d < d' \mathrel{(\, <}a') \, .
$$
Then $[a', b', c', d'] < [a, b, c, d]$.
\end{lemma}

\begin{figure}[hbt]
\psfrag{a}{{\footnotesize $a$}} \psfrag{b}{{\footnotesize $b$}}
\psfrag{c}{{\footnotesize $c$}} \psfrag{d}{{\footnotesize $d$}}
\psfrag{al}{{\footnotesize $a'$}} \psfrag{bl}{{\footnotesize $b'$}}
\psfrag{cl}{{\footnotesize $c'$}} \psfrag{dl}{{\footnotesize $d'$}}
\begin{center}
\includegraphics[width=4.0cm]{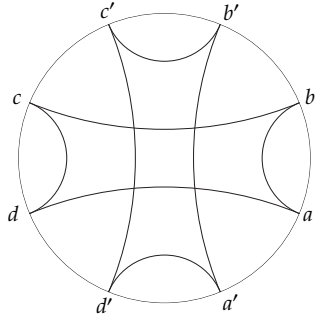}
\caption{{\small Cross-ratio comparison.}}
\label{f.cross}
\end{center}
\end{figure}

\begin{proof}
Using a orientation-preserving projective chart (see \S\ref{ss.if}) we can identify $\P^1$
with the extended line $\R \cup\{\infty\}$,
and also assume that $d' = \infty$.
Then
$$
[a,b,c,d] = \frac{c-a}{b-a} \cdot \frac{d-b}{d-c}
              > \frac{c-a}{b-a} > \frac{c-a'}{b-a'} > \frac{c'-a'}{b'-a'} = [a', b', c', d'].
              \qedhere
$$
\end{proof}

\begin{proof}[Proof of Proposition~\ref{p.non realizable}]
Consider a pair of combinatorial multicones $\M = \M_s \sqcup \M_u$ of order $15$.
Write the unstable combinatorial multicone as:
$$
\M_u = \{\alpha<a<b<\omega<c<d<\beta<\beta'<d'<o<a'<\omega'< b' < c'< \alpha' < \alpha \}
$$
\begin{figure}[!hbt]
\psfrag{a}{{\footnotesize $a$}} \psfrag{b}{{\footnotesize $b$}}
\psfrag{c}{{\footnotesize $c$}} \psfrag{d}{{\footnotesize $d$}}
\psfrag{al}{{\footnotesize $a'$}} \psfrag{bl}{{\footnotesize $b'$}}
\psfrag{cl}{{\footnotesize $c'$}} \psfrag{dl}{{\footnotesize $d'$}}
\psfrag{ga}{{\footnotesize $\alpha$}} \psfrag{gb}{{\footnotesize $\beta$}}
\psfrag{gal}{{\footnotesize $\alpha'$}} \psfrag{gbl}{{\footnotesize $\beta'$}}
\psfrag{w}{{\footnotesize $\omega$}} \psfrag{wl}{{\footnotesize $\omega'$}}
\psfrag{o}{{\footnotesize $o$}}
\psfrag{Q}{{\footnotesize $Q$}} \psfrag{AQ}{{\footnotesize $AQ$}}
\psfrag{BQ}{{\footnotesize $BQ$}} \psfrag{CBQ}{{\footnotesize $CBQ$}}
\begin{center}
\includegraphics[width=8.0cm]{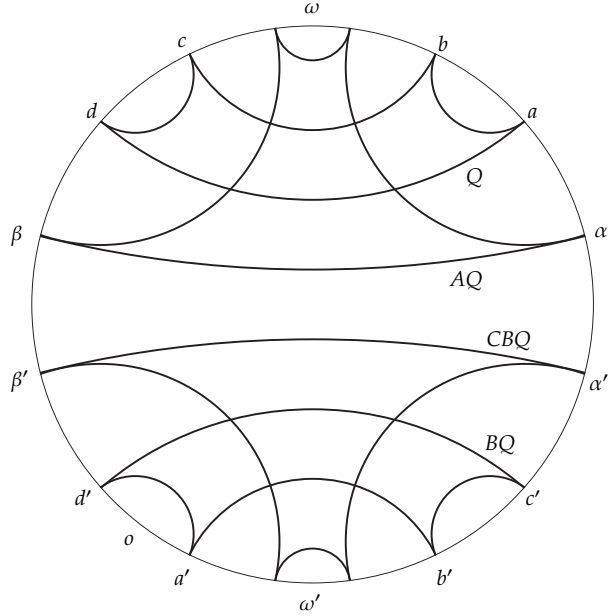}
\caption{{\small The unstable combinatorial multicone.
We will compare cross-ratios of the four rectangles $Q$, $AQ$, $BQ$, $CBQ$.}}
\label{f.non realizable}
\end{center}
\end{figure}

Let maps $A_u$, $B_u$, $C_u: \M_u \to \M_u$ be defined by:
\begin{center}
\begin{tabular}[b]{|r|lllllllllllllll|}
\hline
$x_u$& $\alpha$& $a$& $b$& $\omega$& $c$& $d$& $\beta$& $\beta'$& $d'$& $o$& $a'$& $\omega'$& $b'$& $c'$& $\alpha'$ \\
\hline
$A_u(x_u)$& $\omega$& $\beta$& $\alpha$& $\omega$& $\omega$& $\omega$& $\omega$& $\omega$& $\omega$& $\omega$& $\omega$& $\omega$& $\omega$& $\omega$& $\omega$ \\
$B_u(x_u)$& $a'$& $a'$& $b'$& $c'$& $c'$& $d'$& $d'$& $o$& $o$& $o$& $o$& $o$& $o$& $o$& $o$ \\
$C_u(x_u)$& $\omega'$& $\omega'$& $\omega'$& $\omega'$& $\omega'$& $\omega'$& $\omega'$& $\omega'$& $\omega'$& $\omega'$& $\omega'$& $\omega'$&  $\alpha'$& $\beta'$& $\omega'$ \\
\hline
\end{tabular}
\end{center}
The maps above are monotonic in the sense of \S\ref{ssss.Cs determines C}.
Therefore there exist unique correspondences $A$, $B$, $C$ on $\M$
whose respective $u$-maps are $A_u$, $B_u$, $C_u$, respectively.

Choose some constant correspondences $C^{(4)}$, \ldots, $C^{(N)}$ such that the
morphism $\Phi$ determined by $A$, $B$, $C$, $C^{(4)}$, \ldots, $C^{(N)}$ is tight.

Let us see that the morphism is hyperbolic.
We only need to consider products of the correspondences $A$, $B$, $C$, because
the others are constant.
Inspecting the following diagram,
one sees that any product of length $\ge 4$ of the maps $A_u$, $B_u$, $C_u$
is constant:
$$
\xymatrix@C=50pt{
     & \{\alpha,\omega,\beta\} \ar[r] \ar@{.>}[d] \ar@{-->}[rdd] & \{\omega\} \\
\M_u  \ar[ur]^{A_u} \ar@{.>}[r]^{B_u} \ar@{-->}[dr]_{C_u}
& \{a',b',c',d',o\} \ar[ur] \ar@{.>}[r] \ar@{-->}[d] & \{o\} \\
     & \{\alpha',\omega',\beta'\} \ar[ruu] \ar@{.>}[ru] \ar@{-->}[r] & \{\omega'\}
}
$$
Since any correspondence is constant iff so is its unstable map,
we conclude that the morphism $\Phi$ is hyperbolic.

By contradiction, assume that the morphism $\Phi$
is induced by some hyperbolic $N$-tuple.
Then %(recall \S\ref{ss.tightness})
there is a tight multicone composed of $15$ intervals,
and each element $\xi \in \M_u$ corresponds to one of those intervals, say $I_\xi$.

With abuse of notation, let $A$, $B$, $C$ indicate the first three matrices
of the $N$-tuple.
Choose four points in the circle:
$\mathbf{a} \in I_a$, $\mathbf{b} \in I_b$, $\mathbf{c} \in I_c$, $\mathbf{d} \in I_d$.
Then their images by $A$ belong respectively to $I_\beta$, $I_\alpha$, $I_\omega$, $I_\omega$.
So
$$
A(\mathbf{b}) < \mathbf{a} < \mathbf{b} < A(\mathbf{c}) < A(\mathbf{d})
< \mathbf{c} < \mathbf{d} < A(\mathbf{a}) \mathrel{(\, <} A(\mathbf{b}) ) \, ,
$$
and therefore Lemma~\ref{l.cross} gives
$$
[\mathbf{b}, \mathbf{c}, \mathbf{d}, \mathbf{a}] =
[A(\mathbf{b}), A(\mathbf{c}), A(\mathbf{d}), A(\mathbf{a})] <
[\mathbf{a}, \mathbf{b}, \mathbf{c}, \mathbf{d}] \, .
$$
On the other hand, defining
$\mathbf{a}' = B(\mathbf{a}) \in I_{a'}$,
$\mathbf{b}' = B(\mathbf{b}) \in I_{b'}$,
$\mathbf{c}' = B(\mathbf{c}) \in I_{c'}$,
$\mathbf{d}' = B(\mathbf{d}) \in I_{d'}$,
then
$$
C(\mathbf{a}') < \mathbf{b}' < \mathbf{c}' < C(\mathbf{b}') < C(\mathbf{c}')
< \mathbf{d}' < \mathbf{a}' < C(\mathbf{d}') \mathrel{(\, <} C(\mathbf{a}') ) \, ,
$$
and so using Lemma~\ref{l.cross} again:
$$
[\mathbf{a}, \mathbf{b}, \mathbf{c}, \mathbf{d}] =
[\mathbf{a}', \mathbf{b}', \mathbf{c}', \mathbf{d}'] <
[\mathbf{b}', \mathbf{c}', \mathbf{d}', \mathbf{a}'] =
[\mathbf{b}, \mathbf{c}, \mathbf{d}, \mathbf{a}].
$$
We have reached a contradiction.
\end{proof}

%%%%%%%%%%%%%%%%%%%%%%%%%%%%%%%%%%%%%%%%%%%%
\subsection{Non-Linear Realization of Multicone Dynamics}

We will now see that any combinatorial multicone dynamics has a
non-linear realization.

Given $N$ homeomorphisms $f_1$, \ldots, $f_N : \P^1 \to \P^1$,
we define a skew-product homeomorphism $F: N^\Z \times \P^1 \to N^\Z \times \P^1$
over the shift $\sigma: N^\Z \to N^\Z$
by $(\omega,x) \mapsto (\sigma(\omega), f_{\omega_0}(x))$.

\begin{prop}\label{p.non lin realization}
Let $C^{(1)}$, \ldots, $C^{(N)}$ be correspondences on a pair of combinatorial multicones $\M$.
Then there exist:
\begin{itemize}
\item orientation-preserving diffeomorphisms $f_1$, \ldots, $f_N : \P^1 \to \P^1$;
\item a family of disjoint closed intervals
$I_\xi \subset \P^1$, for $\xi \in \M$, such that
the order inherited from $\M$ is compatible with an orientation of the circle $\P^1$;
\end{itemize}
with the following properties:
\begin{enumerate}
\item
for each $\xi \in \M_u$, we have $f_i (I_\xi) \Subset I_{C^{(i)}_u(\xi)}$ and $f_i'|I_\xi < 1$;

\item
for each $\xi \in \M_s$, we have $f_i^{-1} (I_\xi) \Subset I_{C^{(i)}_s(\xi)}$ and $(f_i^{-1})'|I_\xi < 1$;

\item
if ${F: N^\Z \times \P^1 \hookleftarrow}$ is the skew-product homeomorphism induced by the $f_i$'s
then its non-wandering set $\Omega(F)$ is the union of two disjoint compact $F$-invariant sets
$\Lambda_s$ and $\Lambda_u$, contained respectively in
$N^\Z \times \bigcup_{\xi \in \M_s} I_\xi$ and $N^\Z \times \bigcup_{\xi \in \M_u} I_\xi$;

\item
if the morphism  $\Phi: \cF_N \to \cC(\M)$ induced by the correspondences $C^{(i)}$'s is hyperbolic then
the $F$-invariant sets $\Lambda_u$ and $\Lambda_s$ are topologically transitive.

\item
if the morphism $\Phi$ is tight then $\Omega(F)$ intersects $N^\Z \times I_\xi$ for every $\xi \in \M$;
\end{enumerate}
\end{prop}

\begin{proof}
Let $C^{(1)}, \ldots, C^{(N)}$ be correspondences on a pair of multicones $\M$.

Choose a family  $I_\xi$, indexed by $\xi \in \M$,
of disjoint closed intervals contained in the circle $\P^1$, all with the same positive length,
and such that
the order inherited from $\M$ is compatible with an orientation of the circle.

Fix some $i =1$, \dots, $N$.
For each $\eta \in \Im C_u^{(i)}$, there exist a unique connected component $J_\eta^{(i)}$
of $\P^1 \setminus \bigsqcup_{\xi \in \Im C_s^{(i)}} I_\xi$ that contains all the intervals
$I_x$ such that $(x, \eta) \in C^{(i)} \subset \M \times \M$.
We have%\margem{too fast?}
\begin{equation}\label{e.partition 1}
\P^1 = \bigsqcup_{\xi \in \Im C_s^{(i)}} I_\xi \sqcup \bigsqcup_{\eta \in \Im C_u^{(i)}} J_\eta^{(i)} \, .
\end{equation}
Analogously, for each $\xi \in \Im C_s^{(i)}$, there exist a unique
connected component $J_\xi^{(i)}$ of $\P^1 \setminus \bigsqcup_{\eta \in \Im C_u^{(i)}} I_\eta$
that contains all the intervals $I_y$ for which $(\xi,y) \in C^{(i)}$.
In addition,
\begin{equation}\label{e.partition 2}
\P^1 = \bigsqcup_{\xi \in \Im C_s^{(i)}} J_\xi^{(i)} \sqcup \bigsqcup_{\eta \in \Im C_u^{(i)}} I_\eta \, .
\end{equation}
Let $f_i: \P^1 \to \P^1$
be an orientation-preserving diffeomorphism
such that
$$
f_i \big(\cl J_\eta^{(i)}\big) = I_\eta          \quad \forall \eta \in \Im C_u^{(i)}, \qquad
f_i (I_\xi)                    = \cl J_\xi^{(i)}  \quad \forall \xi  \in \Im C_s^{(i)}.
$$
Then for each $\xi \in \M_u$, we have
$f_i (I_\xi) \subset f_i \big( \cl J_{C_u^{(i)}(\xi)}^{(i)} \big) = I_{C_u^{(i)}(\xi)}$.
Also, $f_i$ can be chosen to be linear in $I_\xi$.
Analogously, for each $\xi \in \M_s$ we have $f_i^{-1} (I_\xi) \subset I_{C^{(i)}_s(\xi)}$,
and we can take $f_i^{-1} | I_\xi$ linear.
Then the maps $f_i$ satisfy properties (i) and (ii) of the proposition.

Define two disjoint subsets of $\P^1$ by
$S = \bigsqcup_{\xi \in \M_s} I_\xi$ and $U = \bigsqcup_{\xi \in \M_u} I_\xi$.
Next we claim that for any $i$,
\begin{equation}\label{e.omega}
f_i \left( \P^1 \setminus S \right) \subset U \, , \qquad
f_i^{-1} \left( \P^1 \setminus U \right) \subset S \, .
\end{equation}
Indeed, if $x \in \P^1 \setminus S$ then by \eqref{e.partition 1}
$x$ belongs to $J_\eta^{(i)}$ for some $\eta \in \Im C_u^{(i)}$.
In particular, $f_i(x) \in I_\eta$, proving the first part of \eqref{e.omega}.
The second part follows by symmetry.

It follows from \eqref{e.omega} that all points in $N^\Z \times \big( \P^1 \setminus (U \cup S) \big)$
are wandering. Hence assertion (iii)  holds.

Now assume the morphism $\Phi$ is hyperbolic.
Given symbols $i_0$, \ldots, $i_{n-1}$,
the set $f_{i_{n-1}} \circ \cdots f_{i_0} (\P^1 \setminus S)$ is
contained in the union of the intervals $I_\xi$ such that
$\xi$ belongs to the image of $C_u^{(n-1)} \circ \cdots \circ C_u^{(0)}$.
So $\xi$ becomes uniquely determined if $n$ is large enough.
By the contraction property (i), we get that
$$
\mathrm{dist}\, \big( F^n(\omega, x), F^n(\omega , y) \big)\xrightarrow[n \to +\infty]{} 0
\quad \text{uniformly for $\omega \in N^\Z$, $x$, $y \in \P^1 \setminus S$.}
$$
Using this, it is easy to show that the $F$-invariant set
$\bigcap_{n \ge 0} F^n \left(N^\Z \times U \right)$ is topologically transitive.
In particular, this set must be equal to $\Omega(F) \cap \big(N^\Z \times U\big)$, that is, $\Lambda_u$.
Analogously, one shows that
$\Lambda_s = \bigcap_{n \ge 0} F^{-n}\left(N^\Z \times S \right)$
is topologically transitive.
This proves part (iv).

The simple proof of assertion (v) is left to the reader.
\end{proof}

%\margem{Em vista disso, poderiamos no perguntar se existe uma versao nao-linear
%do Teorema dos multicones. (Pugh uma vez me perguntou algo desse tipo.)} 

%% file: aby_questions.tex
%%%%%%%%%%%%%%%%%%%%%%%%%%%%%%%%%%%%%%%%%%%%%%%%%%%%%%%%%%%
\section{Questions}\label{s.questions}
%%%%%%%%%%%%%%%%%%%%%%%%%%%%%%%%%%%%%%%%%%%%%%%%%%%%%%%%%%%

The questions and problems proposed in~\cite{Yoccoz_SL2R}
are solved for the full $2$-shift,
but for the general case many questions remain unanswered.
To summarize:

\begin{center}
{\footnotesize
\begin{tabular}{|l|l|l|}
\hline
\emph{Question or Problem from \cite{Yoccoz_SL2R}} & \emph{Full $2$-shift} & \emph{General case}\\
\hline
Q1 (trace signs)                     & yes      & unknown \\
P1 (trace signs)                     & easy now -- use \S\ref{ss.full2 combin}, \S\ref{ss.wind} & unknown \\
Q1' (trace signs $\times$ principal) & no       & no -- see \S\ref{sss.principal and traces}\\
P2 (principal)                       & --       & unknown \\
Q2 (boundary)                        & no       & no, if Q3' is ``yes'' -- see Thm.~\ref{t.general boundary}\\
Q3 (boundary)                        & yes      & no (in general) -- see Prop.~\ref{p.bifurcation}\\
Q3'(boundary)                        & yes      & unknown \\
Q4  (elliptic products)              & yes      & unknown\\
\hline
\end{tabular}
}
\end{center}

%Q1 (sinais dos tracos): resp = sim for full $2$-shift
%Problem 1 (descrever os sinais): feito pra full $2$-shift
%            -- pra achar os sinais dos tracos eh soh contar numero de voltas mod 2
%Q1': resp = nao (for full $2$-shift)
%Problem 2 (descrever a principal e seu bordo para subshift geral): Unsolved
%Q2 (bifurcation types): nao para full $2$-shift.
%          Caso geral: ``quase'' respondida (com nao) --
%          para responder completamente falta provar que todo pt do bordo de $\cH$ esta
%          no bordo de alguma componente.
%Q3, Q3', Q4 (boundaries): sim pra full $2$-shift

We will recall and discuss some of those questions, and also propose new ones.

We return to the general situation where $\Sigma$ is some subshift of finite type,
and $\cH$ is associated hyperbolic locus.

\subsubsection{Boundaries of the Components}

\begin{question}\label{q.boundaries 1}
Are the boundaries of the connected components of $\cH$ disjoint?
\end{question}

A result that goes in the direction of answering (positively) Question~\ref{q.boundaries 1}
is Theorem~\ref{t.disjoint boundaries}.

\begin{question}(Question~3' in \cite{Yoccoz_SL2R})\label{q.boundaries 2}
Is the union of the boundaries of the components equal to the boundary of $\cH$?
\end{question}

A positive answer to Question~\ref{q.boundaries 2}
would answer Question~2 from \cite{Yoccoz_SL2R} negatively
(using Theorem~\ref{t.general boundary}).

\begin{question}\label{q.analytic curve}
If $\gamma : [a,b] \to \G^N$ is an analytic curve,
does the set $\gamma^{-1}(\partial\cH)$ necessarily have countably many components?
\end{question}

A negative answer to Question~\ref{q.analytic curve} would
answer Question~\ref{q.boundaries 2} negatively
(because the components of $\cH$ are semialgebraic).

%\margem{Do we say something about complexity?
%(``If Q\ref{q.analytic curve} = yes then the complexity is locally bounded.'')
%How to define complexity?}

\subsubsection{Elliptic Products}

Denote by $\cE \subset \G^N$ the set of $N$-tuples such that
there exists a periodic point for the subshift over which
the corresponding product is an elliptic matrix.

It is shown in  \cite{Yoccoz_SL2R} that $\overline{\cE} = \cH^c$.

%tem a ver com o tal multicone parabolico, mas nao vale a pena mencionar
\begin{question}(Question~4 in \cite{Yoccoz_SL2R}) \label{q.H and E}
Is $\overline{\cH} = \cE^c$\,?
Equivalently, is $\partial \cH = \partial \cE = (\cH \cup \cE)^c$\,?
\end{question}

We remark that $\cE$ is connected: see Proposition~\ref{p.E connected} in the Appendix.

\subsubsection{Unboundedness of the Components} \label{sss.unboundedness}

Let us say that a set $Z \subset \G^N$ is \emph{bounded modulo conjugacy}
if there exists a compact set $K \subset \G^N$ such that
every $N$-tuple in $Z$ is of the form
$(R A_1 R^{-1}, \ldots, R A_N R^{-1})$,
for some $(A_1,\ldots,A_N) \in K$ and $R \in \G$.
Otherwise, we say that $Z$ is \emph{unbounded modulo conjugacy}.

\begin{question}\label{q.unbounded 1}
Is every connected component of $\cH$ unbounded modulo conjugacy?
\end{question}

Theorem~\ref{t.compactness} in the Appendix
says that a set of $N$-tuples $(A_1,\ldots,A_N)$ is bounded modulo conjugacy iff
the traces of $A_i$'s and $A_i A_j$'s are all bounded.
Motivated by it, we pose a stronger version of Question~\ref{q.unbounded 2}:

\begin{question}[For full shifts]\label{q.unbounded 2}
Are \emph{all} functions $\tr A_i$ and
$\tr A_i A_j$ unbounded in each component?
\end{question}

If $A$ is a uniformly hyperbolic $N$-tuple w.r.t.~some subshift $\Sigma$,
we define its (least) \emph{hyperbolicity rate} as
%\margem{In fact, we can drop periodicity from the definition; but it seems better like this.}
$$
\rho(A) = \liminf_{n \to \infty}
\min \big\{\,\|A^n(x)\|^{1/n} ; \text{ $x \in \Sigma$ has period $n$}\big\}.
$$
Of course, $\rho(A)>1$.

\begin{question}\label{q.unbounded 3}
Is $\rho$ unbounded in each component?
\end{question}

A positive answer to Question~\ref{q.unbounded 3}
implies positive answers to Questions~\ref{q.unbounded 1}
(because of Theorem~\ref{t.compactness}) and \ref{q.unbounded 2}
(because $\rho(A)$ is a lower bound for the modulus of the trace of \emph{any} product of the
matrices in the $N$-tuple $A$).

%**resposta evidente para a principal e provavelmente sim para as componentes tipo grupo.**

%**perguntinha meio deslocada: algum produto minimiza $\rho$? (trocando norma por raio espectral).
%Acho q tem a ver com Bousch-Mairesse.**

It is easy to see that $\rho$ is unbounded in principal components
(for full shifts, of course).
The case $\Sigma = 2^\Z$ is also easily settled:

\begin{prop}
For the case of the full $2$-shift, the answer of Question~\ref{q.unbounded 3} is positive.
\end{prop}

\begin{proof}
It suffices to see that $\rho$ is unbounded on non-principal components.

First consider a free component $H$.
Let $\Sigma$ be the subshift on four symbols considered in \S\ref{ss.ping pong}.
If $(A,B) \in H$,
then $(A,B,A^{-1}, B^{-1})$ is uniformly hyperbolic with respect to $\Sigma$;
see Lemma~\ref{l.free is gr hyp};
let $\rho_\Sigma(A,B)$ indicate the hyperbolicity rate of $(A,B,A^{-1}, B^{-1})$ with respect to $\Sigma$.
It is easy to see that $\rho_\Sigma$ (and in particular, $\rho$) is unbounded in $H$.

Now, consider any other component $H_F = F^{-1}(H)$, where $F \in \cM$.
Given $\tau>1$, take $(A_0,B_0) \in H$ with $\rho_\Sigma(A_0,B_0) > \tau$,
and let $(A,B) = F^{-1}(A_0, B_0) \in H_F$.
In the notations of \S\ref{ss.length} we have that
there exist $c>0$ such that
$$
\| \langle \omega, (A_0,B_0) \rangle \| \ge c \exp( \tau |\omega|)
\quad \text{for every $\omega \in \F_2$.}
$$
As in the proof of Proposition~\ref{p.is hyperbolic}, it follows that
$$
\| \langle \omega, (A, B) \rangle \|
\ge c \exp \left( 2^{-k} \tau |\omega| \right)
\quad \text{for any $\omega \in \F_2$}
$$
(where $k$ depends only on $F$).
Therefore
$$
\rho(A,B) \ge \liminf_{|\omega| \to \infty} \| \langle \omega, (A, B) \rangle \| ^{1/ |\omega|}
\ge \exp \left( 2^{-k} \tau \right).
$$
Hence $\rho$ is unbounded on $H_F$.
\end{proof}

%\subsubsection{Groups $\times$ monoids}
%
%\begin{question}\label{q.monoid and group}
%Does every monoid component contain a group component?
%\end{question}
%
%Notice that a positive answer to Question~\ref{q.monoid and group}
%would, by Proposition~\ref{p.group is unbounded}, imply
%a positive answer to Question~\ref{q.unbounded 1}.

\subsubsection{Topology of the Components}

\begin{question}%\margem{Ok to state this question?}
What are the possible homotopy types of the hyperbolic components?
What about the elliptic locus $\cE$?
\end{question}

In the case of the full $2$-shift, each component has the homotopy type
of a circle.

In Appendix~\ref{ss.ell connected}, we show that $\cE$ is connected.

\subsubsection{Combinatorial Characterization of the Components}

Assume the subshift is full in $N$ letters.

An uniformly hyperbolic $N$-tuple induces a multicone $\M$
in the sense of Section~\ref{s.abstract combinatorics},
and a tight hyperbolic morphism $\Phi$.

Recall that if two uniformly hyperbolic $N$-tuples belong to the same connected component
then they have the same combinatorics,
in the sense the respective morphisms $\Phi$ are conjugate.

\begin{question}\label{q.combinatorics}
Does the combinatorics characterize the connected components of $\cH$,
modulo reflections $(A_1,\ldots,A_N) \mapsto (\pm A_1, \ldots, \pm A_N)$?
\end{question}

In the case $\Sigma = 2^\Z$, our description of the multicone dynamics
(see \S\ref{ss.full2 combin}) gives a positive answer to Question~\ref{q.combinatorics}.

%\margem{NONLINEARITY.
%Another (vague) question: Is there a non-linear version of Theorem~\ref{t.multicone sub}?}

%************************************************
%
%FALTA FALAR:
%
%* Os bordos das componentes nao-principais sao unioes finitas de (pedacos de) hipersuperficies %algebricas.
%(tb sao algebricos passando o quociente por classes de conjugacao)
%
%* Lembre a proposicao 7 de \cite{Yoccoz_SL2R} sobre o bordo de principal.
%Obs: o bordo da principal nao pode ter equacoes em termos de tracos,
%porque existem pelo menos 3 classes de conjugacao em $\G^2$,
%contidas respectivamente em $\cH$ (principal), $\partial \cH$ (bordo da principal), e $\cE$,
%que tem os mesmos valores de $\tr A$, $\tr B$, e $\tr AB$.
%(Note tb que essas classes nao sao fechadas, e seus fechos nao sao disjuntos.)
%
%**outra pergunta: quantas componentes conexas tem $\cE$?**

%% file: aby_appendix.tex
%%%%%%%%%%%%%%%%%%%%%%%%%%%%%%%%%%%%%%%%%%%%%%%%%%%%%%%%
\section{Appendices}
%%%%%%%%%%%%%%%%%%%%%%%%%%%%%%%%%%%%%%%%%%%%%%%%%%%%%%%%

\subsection{A Compactness Criterion for Finite Families of Matrices
in $\SL(2,\R)$ Modulo Conjugacy}\label{ss.compactness criterium}

Let $K$ be a compact subset of $\G$.
Then there exists $C=C(K)>0$ such that, for any $A$, $B \in K$,
$$
|\tr A| \le C, \quad |\tr AB| \le C.
$$
This also holds if $A$, $B$ belong to some conjugate $R^{-1}KR$, $R\in\G$.
We prove that the converse is true:

\begin{thm}\label{t.compactness}
Let $C>0$.
There exists a compact set $K=K(C)$ with the following property:
If $A_1$, \ldots, $A_N \in \G$ satisfy
\begin{alignat}{2}
|\tr A_i|     &\le C, &\quad &1 \le i   \le N, \label{e.criterium 1} \\
|\tr A_i A_j| &\le C, &\quad &1 \le i<j \le N, \label{e.criterium 2}
\end{alignat}
then there exists $R\in\G$ such that $R A_i R^{-1} \in K$ for $1 \le i \le N$.
\end{thm}

%In other words, a set $Z \subset \G^N$ is bounded modulo conjugacy iff
%there exists $C>0$ such that \eqref{e.criterium 1} and \eqref{e.criterium 2}
%hold on all $Z$.

\begin{rem}
It follows that if the inequalities \eqref{e.criterium 1},  \eqref{e.criterium 2}
are satisfied over a subset $Z$  of $\SL(2,\R)^N$ then 
there is a compact set $K \subset \SL(2,\R)^N$ such that 
the union of conjugacy classes of elements of $K$ covers $Z$.
This result does \emph{not } hold for infinite families $(A_i)_{i \in \N}$.
More precisely, consider in $\SL(2,\R)^\N$ the product topology.
If  $f:\N \to \N$ is any map, let $A_i^f = \begin{pmatrix} 1 & f(i) \\ 0 & 1 \end{pmatrix}$.
Let $Z \subset \SL(2,\R)^\N$ be the set of $A^f = (A^f_i)_{i \in \N}$ for all possible $f$.
We have $\tr A_i^f = \tr A_i^f A_j^f = 2$ for all $i$, $j$, $f$.
On the other hand, given any compact set $K \subset \SL(2,\R)^\N$,
there exist $c_i >0$ such that $(B_i)\in K$ implies $\|B_i\| \le c_i$ for every $i\in \N$.
Now, if $f:\N\to \N$ is such that $f(i)/c_i \to \infty$
then $A^f \in Z$ does not belong to any conjugacy class of elements of $K$.
\end{rem}

\begin{proof}[Proof of Theorem~\ref{t.compactness}]
Write
$$
A_i = \begin{pmatrix} x_i & y_i \\ z_i & t_i \end{pmatrix}.
$$
We have
\begin{alignat}{2}
x_i t_i - y_i z_i &= 1 &\quad &\forall i,                     \label{e.1} \\
|x_i + t_i| &\le C &\quad &\forall i,                         \label{e.2} \\
|x_i x_j + t_i t_j + y_i z_j + y_j z_i|&\le C &\quad &\forall i<j. \label{e.3}
\end{alignat}
We want to find a common conjugacy after which all coefficients are bounded
by $C_1 = C_1(C)$.

We start with a particular case:

\noindent \emph{Special case:} Assume that we have moreover
\begin{equation}\label{e.4}
|x_i| \le C_2, \quad \forall i,
\end{equation}
for some $C_2$ depending only on $C$.
We will then conjugate all $A_i$ by the same diagonal matrix.
Observe that from \eqref{e.1}, \eqref{e.2}, \eqref{e.3}, \eqref{e.4},
we get (for some $C_3 = C_3(C)$)
\begin{alignat}{2}
|t_i| &\le C_3 &\quad &\forall i,                               \\
|y_i z_i| &\le C_3 &\quad &\forall i,              \label{e.6}  \\
|y_i z_j + y_j z_i| &\le C_3 &\quad &\forall i<j.  \label{e.7}
\end{alignat}
From \eqref{e.6}, \eqref{e.7} we also get
\begin{alignat}{2}
|y_i z_i y_j z_j| &\le C_3^2 &\quad &\forall i<j,              \\
|y_i z_j| &\le C_4 &\quad &\forall i, \forall j.  \label{e.9}
\end{alignat}
Let
$$
R_\lambda = \begin{pmatrix}
\lambda & 0 \\ 0 & \lambda^{-1}
\end{pmatrix}, \quad
A'_i = R_\lambda A_i R_\lambda^{-1} =
\begin{pmatrix}
x_i & \lambda^2 y_i \\
\lambda^{-2} z_i & t_i \end{pmatrix}.
$$
From \eqref{e.9}, we have
$$
\max_i |y_i| \cdot  \max_j |z_j| \le C_4.
$$
Thus we can choose $\lambda$ such that
\begin{align*}
\max_i |\lambda^2 y_i| \le C_4^{1/2}, \\
\max_i |\lambda^{-2} z_i| \le C_4^{1/2},
\end{align*}
which concludes the proof in the special case.

Let $S_\theta = \begin{pmatrix} \cos \theta  & \sin \theta \\ -\sin\theta & \cos \theta \end{pmatrix}$.
Write
$
S_\theta A_i S_\theta^{-1} =
\begin{pmatrix} x_i(\theta) & y_i(\theta) \\ z_i(\theta) & t_i(\theta) \end{pmatrix}
$.
We have
$$
x_i(\theta) = x_i \cos^2 \theta + t_i \sin^2 \theta + (y_i+z_i)\sin\theta \cos\theta.
$$
We want to prove that there exists $C_2 = C_2(C)$ and $\theta$ such that
\begin{equation}\label{e.10}
|x_i(\theta)| \le C_2 \quad \forall i.
\end{equation}
Indeed, in this case we are reduced to the special case above.
From~\eqref{e.2}, we see that \eqref{e.10} is equivalent to
\begin{equation}
\left| x_i \cos 2\theta + \frac{y_i+z_i}{2} \sin 2\theta \right| \le C_2'
\quad \forall i.
\end{equation}
Observe that
$$
x_i^2(\theta) + y_i^2(\theta) + z_i^2(\theta) + t_i^2(\theta) =
\tr S_\theta \; A_i \; {}^t \! A_i \; S_\theta^{-1}
$$
does not depend on $\theta$.
We can assume that
\begin{equation}\label{e.11}
x_1^2 + y_1^2 + z_1^2 + t_1^2 \ge x_i^2 + y_i^2 + z_i^2 + t_i^2,
\quad \forall i \ge 1.
\end{equation}
Choose $\theta$ such that
\begin{equation}\label{e.12}
x_1 \cos 2\theta + \frac{y_1+z_1}{2} \sin 2\theta = 0.
\end{equation}
Replacing $A_i$ by $S_\theta A_i S_\theta^{-1}$, we can assume that
\begin{equation}\label{e.13}
|x_1| \le C.
\end{equation}
We will show that \eqref{e.1}, \eqref{e.2}, \eqref{e.3}, \eqref{e.12}, \eqref{e.13}
together imply \eqref{e.4}.
Actually, we only need \eqref{e.3} for $i=1$, i.e.,
\begin{equation}\label{e.14}
|x_1 x_i + t_1 t_i + y_1 z_i + y_i z_1| \le C, \quad \forall i>1.
\end{equation}
Observe first that from \eqref{e.1}, \eqref{e.2}, \eqref{e.13} we get
\begin{align}
|t_1| &\le 2C,           \label{e.15} \\
|y_1 z_1| &\le 1 + 4C^2. \label{e.16}
\end{align}
Replacing if necessary all $A_i$ by ${}^t \! A_i$, we can assume that
\begin{equation}\label{e.17}
|y_1| \ge |z_1|
\end{equation}
From \eqref{e.13}, \eqref{e.15}, \eqref{e.16}, \eqref{e.17}, we have
\begin{equation}\label{e.18}
x_1^2+z_1^2+t_1^2 \le C_5 = C_5(C).
\end{equation}
From \eqref{e.11}, we then get
\begin{equation}\label{e.19}
\max \left( |x_i|, |y_i|, |z_i|, |t_i| \right) \le |y_1| + C_6.
\end{equation}
In particular,
\begin{equation}\label{e.20}
|y_i z_1| \le |y_1 z_1| + C |z_1| \le C_7
\end{equation}
and thus, from \eqref{e.14},
\begin{equation}\label{e.21}
|x_1 x_i + t_1 t_i + y_1 z_i| \le C_7 + C.
\end{equation}
From \eqref{e.2}, \eqref{e.15} we also have
\begin{equation}\label{e.22}
|t_1 t_i + t_1 x_i| \le 2 C^2
\end{equation}
and therefore, using \eqref{e.13}, \eqref{e.15}, \eqref{e.21},
\begin{equation}\label{e.23}
|y_1 z_i| \le C_8 (|x_i|+1).
\end{equation}
If $|y_1| \le C_2'''$, we conclude directly from \eqref{e.19}
that $|x_i| \le C_2$.
Assume therefore that $|y_1|$ is large.
Then, from \eqref{e.23} we have
\begin{equation}\label{e.24}
|z_i| \le C_8 \frac{1+|x_i|}{|y_1|}
\end{equation}
We have also, from \eqref{e.19}, \eqref{e.2},
\begin{align}
|y_i| &\le |y_1| + C_6, \label{e.25} \\
|t_i| &\ge |x_i| - C.   \label{e.26}
\end{align}
Therefore, from \eqref{e.1},
\begin{align*}
|x_i| \left( |x_i| - C \right)
&\le 1 + |y_i z_i| \\
&\le 1 + C_8 \frac{|y_1|+C_6}{|y_1|} \left(1 + |x_i| \right) \\
&\le C_9 |x_i|,
\end{align*}
which gives finally \eqref{e.4}.
\end{proof}

\subsection{Connectivity of the Elliptic Locus} \label{ss.ell connected}

Recall that in the case of the full shift in $N$ symbols,
$\cE$ denotes the (open) subset of $\G^N$ formed by the $N$-tuples which
have an elliptic product.

\begin{prop}\label{p.E connected}
$\cE$ is connected.
\end{prop}

Let $R_\theta \in \G$ denote the rotation by angle $\theta$.
The proof of connectivity of $\cE$ needs the following:
\begin{lemma}\label{l.rotation}
Fix $B_1, \ldots, B_n \in \G$, and let
$$
F(\theta) = \tr \big( B_1 R_\theta B_2 R_\theta \cdots B_n R_\theta \big)
$$
Then for every parameter $\theta$ for which $|F(\theta)| < 2$ we have
$F'(\theta) \neq 0$.
\end{lemma}

\begin{proof}
This lemma is essentially proved in \cite{AvilaBochi_Israel}.
Complexification gives a rational function $Q(z)$ such that
$Q(e^{i\theta}) = F(\theta)$ for real $\theta$.
Moreover, $Q(z) = P(z) / z^n$ where $P(z)$ is a polynomial of degree at most $2n$.

First assume that the matrices $B_i$ satisfy:
\begin{equation}\label{e.israel}
B_1 R_\theta B_2 R_\theta \cdots B_n R_\theta \neq \pm \id \quad
\text{for all $\theta \in \R$.}
\end{equation}
A topological argument then gives that
the intersection of $Q^{-1}([-2,2])$ with the unit circle $S^1$ has at least $2n$
connected components -- this is Lemma~10 in \cite{AvilaBochi_Israel}.
On the other hand, $Q$ restricted to $S^1$ is real-valued
and thus each connected component of $S^1 \setminus Q^{-1}([-2,2])$
contains at least one zero of $Q'(z) = (zP'(z) - nP(z))/z^{n+1}$.
It follows that all the zeros of $Q'$ are simple and contained in
$S^1 \setminus Q^{-1}([-2,2])$.
Moreover, $Q^{-1}([-2,2])$ consists of exactly $2n$ intervals in $S^1$,
each with length at least $4\|Q'|S^1\|_\infty^{-1}$.

Now it follows by a perturbation argument that
even if condition~\eqref{e.israel} is not satisfied,
all the zeros of $Q'$ are simple
and contained in $S^1 \setminus Q^{-1}((-2,2))$.
This concludes the proof of the lemma.
\end{proof}

\begin{proof}[Proof of Proposition~\ref{p.E connected}]
First notice that the set of elliptic matrices is connected,
that is, the proposition is true for $N=1$.

Now let $N \ge 2$.
Take $(A_1, \ldots, A_N)$ in $\cE$, so
some product $A_{j_1} \cdots A_{j_m}$ is elliptic.
To prove the proposition, it suffices to find
a path $t \in [0,1] \mapsto (A_i(t))_i$ in $\cE$ starting from $(A_i)$
such that $A_\ell(1)$ is elliptic for \emph{some} $\ell$.
Let $\ell$ be any of $j_1, \ldots, j_m$.
We can assume some $j_i$ is different from $\ell$,
because otherwise there is nothing to prove.

Take a path $t \in [0,1] \mapsto A_\ell(t)$ starting at $A_\ell$
and ending at some elliptic matrix.
Let $A_i(t,\theta)$ be equal to
$R_\theta A_i$ if $i\neq \ell$,
and $A_\ell(t,\theta) = A_\ell(t)$.
Also, let $F(t, \theta)$ be the trace of
$A_{j_1}(t,\theta) \cdots A_{j_m}(t,\theta)$.
Lemma~\ref{l.rotation} (together with the assumption that some $j_i$ is different from $\ell$)
guarantees that $\frac{\partial F}{\partial \theta} \neq 0$ when $|F|<2$.
Therefore the differential equation
$\frac{d}{dt}F(t,\theta(t)) = 0$ with initial condition $\theta(0)=0$
has a solution $\theta(t)$ defined for $t\in [0,1]$.
Consider the path $t \mapsto (A_i(t))$ where $A_i(t) = A_i(t,\theta(t))$.
The path is contained in $\cE$ because the trace of $A_{j_1}(t) \cdots A_{j_m}(t)$
is constant; also, $A_\ell(1)$ is elliptic.
So we are done.
\end{proof}